\documentclass[reqno,11pt,letterpaper]{amsart}

\usepackage{amsmath,amssymb,amsthm,mathrsfs,bbm}
\usepackage{accents}
\usepackage{arydshln}
\usepackage{upgreek}
\usepackage{slashed}
\usepackage{xifthen}
\usepackage{graphicx}
\usepackage{longtable}
\usepackage[inline]{enumitem}
\usepackage[all]{xy}

\usepackage{xcolor}
\definecolor{winered}{rgb}{0.6,0,0}
\definecolor{lessblue}{rgb}{0,0,0.7}

\usepackage[pdftex,colorlinks=true,linkcolor=winered,citecolor=lessblue,breaklinks=true,bookmarksopen=true]{hyperref}

\makeatletter
\newcommand{\myitem}[3]{\item[#2]\def\@currentlabel{#3}\label{#1}}
\makeatother

\usepackage{titletoc}

\makeatletter

\def\@tocline#1#2#3#4#5#6#7{
\begingroup
  \par
    \parindent\z@ \leftskip#3 \relax \advance\leftskip\@tempdima\relax
                  \rightskip\@pnumwidth plus 4em \parfillskip-\@pnumwidth
    \ifcase #1 
       \vskip 0.6em \hskip 0em 
       \or
       \or \hskip 0em 
       \or \hskip 1em 
    \fi%
    #6
    \nobreak\relax{\leavevmode\leaders\hbox{\,.}\hfill}
    \hbox to\@pnumwidth {\@tocpagenum{#7}}
  \par
\endgroup
}

 \def\l@section{\@tocline{0}{0pt}{0pc}{}{}}

\renewcommand{\tocsection}[3]{%
  \indentlabel{\@ifnotempty{#2}{ 
    \ignorespaces\bfseries{#2. #3}}}
  \indentlabel{\@ifempty{#2}{\ignorespaces\bfseries{#3}}{}} 
    \vspace{1.5pt}}

\renewcommand{\tocsubsection}[3]{%
  \indentlabel{\@ifnotempty{#2}{
    \ignorespaces#2. #3}}
  \indentlabel{\@ifempty{#2}{\ignorespaces #3}{}}
    \vspace{1.5pt}}

\renewcommand{\tocsubsubsection}[3]{%
  \indentlabel{\@ifnotempty{#2}{
    \ignorespaces#2. #3}}
  \indentlabel{\@ifempty{#2}{\ignorespaces #3}{}}
    \vspace{1.5pt}}

\makeatother

\makeatletter
\def\@nomenstarted{0}
\newlength{\@nomenoldtabcolsep}

\newcommand{\nomenstart}
  {%
    \def\@nomenstarted{1}%
    \setlength{\@nomenoldtabcolsep}{\tabcolsep}%
    \setlength{\tabcolsep}{3.5pt}%
    \begin{longtable}{p{0.11\textwidth} p{0.86\textwidth}}
  }

\newcommand{\nomenitem}[2]{%
    \ifcase\@nomenstarted%
      \or 
      \or \\ 
    \fi%
    #1\,{\leavevmode\leaders\hbox{\,.}\hfill} & #2%
    \def\@nomenstarted{2}%
  }%
\newcommand{\nomenend}
  {\\%
      \end{longtable}%
      \setlength{\tabcolsep}{\@nomenoldtabcolsep}%
      \def\@nomenstarted{0}%
  }
\makeatother

\makeatletter
\newcommand{\vast}{\bBigg@{4}}
\newcommand{\Vast}{\bBigg@{5}}
\makeatother

\allowdisplaybreaks

\numberwithin{equation}{section}
\numberwithin{figure}{section}
\newtheorem{thm}{Theorem}[section]

\newtheorem{prop}[thm]{Proposition}
\newtheorem{lemma}[thm]{Lemma}
\newtheorem{cor}[thm]{Corollary}
\newtheorem*{thm*}{Theorem}
\newtheorem*{prop*}{Proposition}
\newtheorem*{cor*}{Corollary}
\newtheorem*{conj*}{Conjecture}

\theoremstyle{definition}
\newtheorem{definition}[thm]{Definition}

\theoremstyle{remark}
\newtheorem{rmk}[thm]{Remark}

\newcommand{\mc}{\mathcal}
\newcommand{\cA}{\mc A}
\newcommand{\cB}{\mc B}
\newcommand{\cC}{\mc C}
\newcommand{\cD}{\mc D}
\newcommand{\cE}{\mc E}
\newcommand{\cF}{\mc F}
\newcommand{\cG}{\mc G}

\newcommand{\cO}{\mc O}

\newcommand{\cQ}{\mc Q}

\newcommand{\cU}{\mc U}
\newcommand{\cV}{\mc V}

\newcommand{\C}{\mathbb{C}}
\newcommand{\N}{\mathbb{N}}
\newcommand{\R}{\mathbb{R}}

\newcommand{\Sph}{\mathbb{S}}

\newcommand{\bfX}{\mathbf{X}}

\renewcommand{\Re}{\operatorname{Re}}
\renewcommand{\Im}{\operatorname{Im}}
\newcommand{\Id}{\operatorname{Id}}

\newcommand{\supp}{\operatorname{supp}}

\newcommand{\eps}{\epsilon}
\newcommand\ff{{\mathrm{ff}}}
\newcommand\fff{{\mathrm{fff}}}
\newcommand{\ftrans}{\;\!\wh{\ }\;\!}

\newcommand{\hra}{\hookrightarrow}
\newcommand{\la}{\langle}

\DeclareMathOperator{\extcup}{\ol\cup}
\newcommand{\bigextcup}{\ol{\phantom{\bigcup}}\hspace{-1.58em}\bigcup}
\newcommand{\ol}{\overline}
\newcommand{\pa}{\partial}
\newcommand{\ra}{\rangle}
\newcommand{\spec}{\operatorname{spec}}

\newcommand{\wh}{\widehat}
\newcommand{\wt}{\widetilde}
\newcommand{\xra}{\xrightarrow}

\newcommand{\dd}{{\mathrm{d}}}

\newcommand{\bop}{{\mathrm{b}}}
\newcommand{\cop}{{\mathrm{c}}}
\newcommand{\scop}{{\mathrm{sc}}}
\newcommand{\lb}{{\mathrm{lb}}}
\newcommand{\rb}{{\mathrm{rb}}}
\newcommand{\tlb}{{\mathrm{tlb}}}
\newcommand{\trb}{{\mathrm{trb}}}
\newcommand{\bface}{{\mathrm{bf}}}

\newcommand{\tbface}{{\mathrm{tbf}}}

\newcommand{\dface}{{\mathrm{df}}}
\newcommand{\dfface}{{\mathrm{dff}}}
\newcommand{\sface}{{\mathrm{sf}}}
\newcommand{\sfface}{{\mathrm{sff}}}
\newcommand{\tface}{{\mathrm{tf}}}
\newcommand{\tfface}{{\mathrm{tff}}}

\newcommand{\cp}{{\mathrm{c}}}

\newcommand{\scl}{{\mathrm{sc}}}

\newcommand{\semi}{\hbar}

\newcommand{\Diff}{\mathrm{Diff}}

\DeclareMathOperator{\Op}{Op}
\newcommand{\Oph}{\Op_h}
\newcommand{\Vb}{\cV_\bop}

\newcommand{\Diffb}{\Diff_\bop}

\newcommand{\Psib}{\Psi_\bop}

\newcommand{\Diffbh}{\Diff_{\bop,\semi}}
\newcommand{\Psibh}{\Psi_{\bop,\semi}}

\newcommand{\Psih}{\Psi_\semi}

\newcommand{\Diffh}{\Diff_\semi}

\newcommand{\WF}{\mathrm{WF}}

\newcommand{\Shom}{S_{\mathrm{hom}}}

\newcommand{\Omegab}{{}^{\bop}\Omega}

\newcommand{\Tb}{{}^{\bop}T}

\newcommand{\WFh}{\WF_{\semi}}

\newcommand{\half}{{\tfrac{1}{2}}}
\newcommand{\mfrac}[2]{\genfrac{}{}{}{3}{#1}{#2}}
\newcommand{\mhalf}{{\mfrac{1}{2}}}

\newcommand{\sigmab}{{}^\bop\sigma}

\newcommand{\specb}{\mathrm{spec}_\bop}
\newcommand{\specbfull}{\wt{\mathrm{spec}}_\bop}

\newcommand{\sigmach}{{}^{\cop\semi}\sigma}

\newcommand{\dscop}{{\cop\semi\tilde\semi}}
\newcommand{\sigmadsc}{{}^\dscop\sigma}
\newcommand{\Psidsc}{\Psi_\dscop}

\newcommand{\sigmahth}{{}^{\semi\tilde\semi}\sigma}

\newcommand{\loc}{{\mathrm{loc}}}
\newcommand{\CI}{\cC^\infty}
\newcommand{\CIdot}{\dot\cC^\infty}

\newcommand{\CIc}{\cC^\infty_\cp}

\newcommand{\Hb}{H_{\bop}}

\newcommand{\Hbh}{H_{\bop,h}}

\newcommand{\Hsc}{H_{\scop}}

\newcommand{\phg}{{\mathrm{phg}}}

\newcommand{\openbigpmatrix}[1]
  {%
    \def\@bigpmatrixsize{#1}%
    \addtolength{\arraycolsep}{-#1}%
    \begin{pmatrix}%
  }
\newcommand{\closebigpmatrix}
  {%
    \end{pmatrix}%
    \addtolength{\arraycolsep}{\@bigpmatrixsize}%
  }

\newlength{\enummargin}

\newcommand{\usref}[1]{{\upshape\ref{#1}}}

\DeclareGraphicsExtensions{.mps}
\newcommand{\inclfig}[1]{\includegraphics{#1}}

\makeatletter
\newcounter{@enumsave}

\newcommand{\ctrset}{\setcounter{enumi}{\the@enumsave}}
\makeatother

\makeatletter
\newcommand*{\fwbw}[1]{\expandafter\@fwbw\csname c@#1\endcsname}
\newcommand*{\@fwbw}[1]{\ifcase #1 \or {\rm fw}\or {\rm bw}\fi}
\AddEnumerateCounter{\fwbw}{\@fwbw}
\makeatother

\begin{document}

\title{Resolvents and complex powers of semiclassical cone operators}

\begin{abstract}
  We give a uniform description of resolvents and complex powers of elliptic semiclassical cone differential operators as the semiclassical parameter $h$ tends to $0$. An example of such an operator is the shifted semiclassical Laplacian $h^2\Delta_g+1$ on a manifold $(X,g)$ of dimension $n\geq 3$ with conic singularities. Our approach is constructive and based on techniques from geometric microlocal analysis: we construct the Schwartz kernels of resolvents and complex powers as conormal distributions on a suitable resolution of the space $[0,1)_h\times X\times X$ of $h$-dependent integral kernels; the construction of complex powers relies on a calculus with a second semiclassical parameter. As an application, we characterize the domains of $(h^2\Delta_g+1)^{w/2}$ for $\Re w\in(-\tfrac{n}{2},\tfrac{n}{2})$ and use this to prove the propagation of semiclassical regularity through a cone point on a range of weighted semiclassical function spaces.
\end{abstract}

\date{\today}

\author{Peter Hintz}
\address{Department of Mathematics, Massachusetts Institute of Technology, Cambridge, Massachusetts 02139-4307, USA}
\email{phintz@mit.edu}

\maketitle

\section{Introduction}
\label{SI}

Consider a compact conic manifold of dimension $n\geq 3$ with a single cone point. Upon introducing polar coordinates around the cone point, this can be described as a compact manifold $X$ with connected embedded boundary $\pa X$, and a smooth Riemannian metric $g$ on $X^\circ$ which in a collar neighborhood $[0,\eps)_x\times\pa X$ of $\pa X$ takes the form
\[
  g = \dd x^2 + x^2 k(x).
\]
Here, $k(x)$ is a smooth family of metrics on $\pa X$. See Figure~\ref{FigI}.

\begin{figure}[!ht]
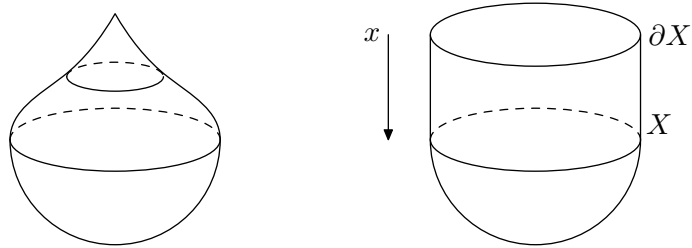

\inclfig{FigI}
\caption{\textit{Left:} a two-dimensional conic manifold. \textit{Right:} its resolution, a manifold $X$ with boundary $\pa X=x^{-1}(0)$. The boundary is metrically a point.}
\label{FigI}
\end{figure}

Let $\Delta_g\geq 0$ denote the Friedrichs extension. In this paper, we give a precise description of Schwartz kernels of resolvents and complex powers of elliptic semiclassical operators related to $\Delta_g$. Denote the diagonal in $X\times X$ by $\Delta_X=\{(p,p)\colon p\in X\}$.

\begin{thm}[Rough version of Theorem~\usref{ThmLF}.]
\label{ThmI}
  Let
  \begin{equation}
  \label{EqI}
    A_h=h^2\Delta_g+1.
  \end{equation}
  Let $w\in\C$. Then the complex power $A_h^{w/2}$, $h\in(0,1)$, defined using the functional calculus for $\Delta_g$ and restricted to the domain $\CIc(X^\circ)$, is a weighted semiclassical cone pseudodifferential operator. This means that it has a distributional Schwartz kernel of the following form: on a resolution $X^2_{\cop\semi}$ of $[0,1)_h\times X\times X$, see Figure~\usref{FigRCDouble}, it can be written as $A_h^{w/2}=B(w)+C(w)$, where $B(w)$ is conormal of order $w$ to the lift $\Delta_{\cop\semi}$ of $[0,1)_h\times\Delta_X$ to $X^2_{\cop\semi}$ and vanishes to infinite order at $\lb$, $\rb$, $\sface$, while $C(w)$ vanishes to infinite order at $\sface,\dface$ and has polyhomogeneous expansions at all other boundary hypersurfaces.
\end{thm}

Restriction to various parts of the double space $X^2_{\cop\semi}$ gives standard objects: \begin{enumerate*} \item for any fixed $h=h_0>0$, $A_{h_0}^{w/2}$ is a b-pseudodifferential operator \cite[\S5]{MelroseAPS}, as observed previously by Loya \cite{LoyaConicHeat}; \item if $\phi\in\CI(X)$ vanishes near $\pa X$, then $\phi A_h^{w/2}\phi$ is a semiclassical ps.d.o.\ \cite[Part~4]{ZworskiSemiclassical}.\end{enumerate*} See Remark~\ref{RmkRCOther} for further details.

\begin{rmk}
  We in fact describe inverses and complex powers for a general class of fully elliptic semiclassical cone differential operators, see Definition~\ref{DefOpEll} as well as Theorems~\ref{ThmR} and \ref{ThmPh}; for simplicity of presentation, we shall restrict ourselves to the case~\eqref{EqI} in the introduction.
\end{rmk}

The semiclassical resolvent $(h^2\Delta_g-\tilde\lambda)^{-1}$ for $\tilde\lambda\notin[0,\infty)$ has the same structure. This is closely related to \cite{LoyaConicResolvent}, where Loya constructed $(\Delta_g-\lambda)^{-1}$ as an element of a suitable large parameter cone calculus, with uniform control as $\lambda\to\infty$ in suitable sectors in the complex plane. Large parameter calculi are slightly more precise than semiclassical ones (related via $h=|\lambda|^{-1/2}$, $\tilde\lambda=\lambda/|\lambda|$), see also \cite[\S2]{VasyMicroKerrdS}. The advantage of semiclassical calculi however is that they are directly amenable to geometric microlocal techniques, allowing us to significantly simplify and streamline the constructions in \cite{LoyaConicResolvent}, and paving the way also for the subsequent somewhat involved construction of the doubly semiclassical calculus needed for the construction of the complex powers of Theorem~\ref{ThmI} (see also the discussion following equation~\eqref{EqIContour} below).

We stress that the present paper concerns the \emph{off-spectrum} behavior of the high energy resolvent of conic Laplacians. On the other hand, in recent work, Xi \cite{XiConeParametrix} gives a detailed construction of the high energy resolvent of conic Laplacians \emph{near} the spectrum, with applications to Strichartz estimates for Schr\"odinger operators. Combining the present paper with \cite{XiConeParametrix} thus gives a rather complete description of the high energy resolvent.

\begin{figure}[!ht]
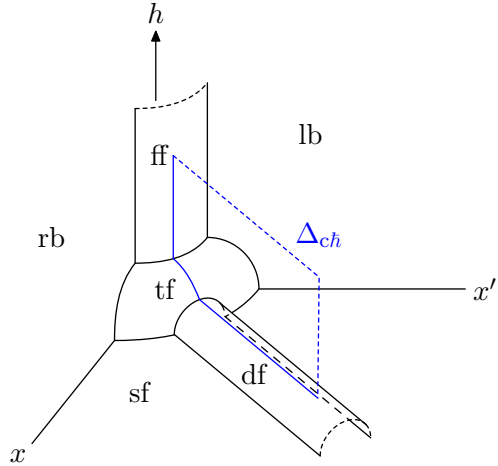

\inclfig{FigRCDouble}
\caption{The semiclassical cone double space $X^2_{\cop\semi}$, see Definition~\ref{DefRCDouble}, and its boundary hypersurfaces. Here, $x$ and $x'$ are the pullbacks of the boundary defining function to the first and second factor of $X\times X$. Also shown is the diagonal $\Delta_{\cop\semi}$. For dimensional reasons, variables in $\pa X\times\pa X$ cannot be depicted.}
\label{FigRCDouble}
\end{figure}

Our construction of complex powers follows Seeley's original approach \cite{SeeleyPowers} (see also \cite{SchrohePowers1986}). In the non-semiclassical conic setting, Loya~\cite{LoyaConicPower} employed an argument using the heat kernel constructed in \cite{LoyaConicHeat}, see also \cite{GilConicHeat}. Imaginary powers were constructed by Coriasco--Schrohe--Seiler \cite{CoriascoSchroheSeilerCone} (with quantitative operator norm bounds as $|\Im w|\to\infty$) based on Schulze's cone calculus \cite{SchulzePsdoSing,SchulzePsdoBVP94,SchulzeBVP98}; see also \cite{GilKrainerMendozaResolvents}. Constructions on general classes of manifolds can be found in \cite{SchrohePowersNoncompact,AmmannLauterNistorVasyPowers}. For manifolds with boundaries or corners, see \cite{GrubbFunctionalCalculus,RempelSchulzePowers,LoyaPowersB}. The relationship between Schulze's cone calculus and Melrose's geometric microlocal point of view (on which the present paper is based on) is explained by Lauter--Seiler in \cite{LauterSeilerComparison}.

We give an application of Theorem~\ref{ThmI} to the propagation of semiclassical regularity (amplitudes of oscillations of frequency $h^{-1}$) through cone points for solutions of the equation
\begin{equation}
\label{EqIMoti}
  (h^2\Delta_g-1)u=f;
\end{equation}
notice the minus sign, which makes this into a semiclassically non-elliptic equation. Such estimates have been obtained in a number of different settings: Cheeger--Taylor \cite{CheegerTaylorConicalI,CheegerTaylorConicalII}, see also \cite{KalkaMenikoffCone}, analyzed wave propagation on exact cones, followed by Melrose--Wunsch \cite{MelroseWunschConic}, using geometric microlocal techniques, on general conic manifolds, see also \cite{VasyPropagationCorners,MelroseVasyWunschEdge}; these papers also prove diffractive regularity improvements via the use of the edge calculus \cite{MazzeoEdge,MazzeoVertmanEdge}. (For recent results on the fine properties of the diffracted wave, see \cite{FordHassellHillairet2dCone,YangDiffraction}.) More recently, Baskin--Marzuola \cite{BaskinMarzuolaCone} proved semiclassical propagation estimates on unweighted spaces (relative to the quadratic form domain). We also mention the closely related works by Gannot--Wunsch \cite{GannotWunschPotential} on semiclassical diffraction by conormal potentials, and by Hillairet--Wunsch \cite{HillairetWunschConic} on high energy resonances generated by conic diffraction.

In order to obtain sharp propagation estimates on function spaces which admit weights $x^l$ (relative to the quadratic form domain) at the cone point, it is convenient to conjugate the PDE~\eqref{EqIMoti} by an operator which \begin{enumerate*} \item commutes with the operator $h^2\Delta_g-1$ and \item is, roughly, homogeneous of degree $l$ with respect to dilations in $x$;\end{enumerate*} for the analysis of the resulting equation, one can then appeal to results of~\cite{BaskinMarzuolaCone}. A natural choice for such an operator is the power $(h^2\Delta_g+1)^{-l/2}$. A crucial input, needed to guarantee that microlocal assumptions and conclusions on geodesics going into or coming out of the cone point are preserved under such a conjugation, is that such a power preserves semiclassical wave front sets (see~\S\ref{SsLProp})---which semiclassical ps.d.o.s do indeed; Theorem~\ref{ThmI} provides this input.

\begin{thm}
\label{ThmIProp}
  For $A_h=h^2\Delta_g+1$ as in~\eqref{EqI}, put $\cD_h^w := \cD(A_h^{w/2})$ for $w\in\C$. Let $l\in(-\tfrac{n-2}{2},\tfrac{n-2}{2})$. Suppose $u\in h^{-N}\cD_h^{1+l}$ solves $(h^2\Delta_g-1)u=f$ with $f\in h\cD_h^{-1+l}$ (i.e.\ $h^{-1}f$ has uniformly bounded norm in $\cD_h^{-1+l}$). If $u$ is uniformly bounded in $L^2$ when microlocalized to any open bounded subset of $T^*X^\circ$, then $u\in\cD_h^{1+l}$.
\end{thm}

See Theorem~\ref{ThmLProp} for a more natural statement on a truncated cone: the assumption on $u$ is then that $u$ bounded when microlocalized away from directions which are outgoing from the cone point, and the conclusion gives uniform bounds on $u$ everywhere; that is, one can propagate uniform control through the cone point to the outgoing directions.

As a consequence of our analysis, and as briefly motivated below, there is a scale of weighted Sobolev spaces naturally associated with the algebra of semiclassical cone operators,
\[
  H_{\cop,h}^{s,\alpha,\tau}(X;|\dd g|) = \bigl(\tfrac{x}{x+h}\bigr)^\alpha(x+h)^\tau H_{\cop,h}^s(X;|\dd g|),\qquad s,\alpha,\tau\in\R.
\]
As a vector space, this is the \emph{fixed} weighted b-Sobolev space $x^\alpha H_b^s(X;|\dd g|)$, but its norm is $h$-dependent and captures precise semiclassical behavior near the cone point. For example, for $k\in\N_0$, an $h$-dependent family of distributions $u$ lies in $H_{\cop,h}^{k,\alpha,\tau}(X;|\dd g|)$ if and only if $(\tfrac{x}{x+h})^{-\alpha}(x+h)^{-\tau}V_1\dots V_j u\in L^2(X;|\dd g|)$ for all $j=0,\ldots,k$ and each $V_i$ of the form $V_i=\tfrac{h}{x+h}W_i$, where $W_i$ is a smooth vector field on $X$ tangent to $\pa X$ (i.e.\ a \emph{b-vector field}). Our final main result identifies the domains $\cD_h^w=\cD(A_h^{w/2})$ of powers $A_h^{w/2}$ uniformly as $h\to 0$ (i.e.\ with uniformly equivalent norms) in terms of these spaces:

\begin{thm}[Brief version of Theorem~\ref{ThmLF}]
\label{ThmIDom}
  Let $A_h=h^2\Delta_g+1$ and put $\cD_h^w=\cD(A_h^{w/2})$. Then for $|\Re w|<\frac{n}{2}$, we have\footnote{In the main body of the paper, we will use a b-density on $X$ rather than the metric density $|\dd g|$, which leads to a shift of the weights in Theorem~\ref{ThmLF} by $\frac{n}{2}$ relative to the present statement.}
  \[
    \cD_h^w=H_{\cop,h}^{\Re w,\Re w,0}(X;|\dd g|).
  \]
  This is false without the restriction on $\Re w$; see \cite[\S3]{MelroseWunschConic} and Remark~\usref{RmkLFDom}.
\end{thm}

Our proof of Theorem~\ref{ThmI} uses techniques from geometric microlocal analysis; see Appendix~\ref{SB} for a brief review. Thus, in~\S\ref{SR} we define the space $\Psi_{\cop\semi}(X)$ of \emph{semiclassical cone pseudodifferential operators} on $X$ as distributions on $X^2_{\cop\semi}$, conormal to $\Delta_{\cop\semi}$ smoothly down to $\ff,\tface,\dface$, which vanish to infinite order at $\lb,\rb,\sface$; we also consider a \emph{large calculus} which allows for general polyhomogeneous asymptotics at $\ff,\tface,\lb,\rb$. Typical vector fields $V\in\Psi_{\cop\semi}^1(X)$ are of the form $V=\tfrac{h}{x+h}W$ with $W$ a b-vector field on $X$. Concretely, working in local coordinates $[0,\eps)_x\times\R^{n-1}_y$ near a point in $\pa X$, we have
\begin{equation}
\label{EqIVf}
  \tfrac{h}{x+h}x D_x,\ 
  \tfrac{h}{x+h}D_{y^1},\ \ldots,\ \tfrac{h}{x+h}D_{y^{n-1}}\in\Psi_{\cop\semi}^1(X).
\end{equation}
For fixed $h>0$, these vector fields are smooth multiples of $x D_x$ and $D_{y^j}$, thus typical b-vector fields, while for $x$ bounded away from $0$, they are smooth multiples of $h D_x$ and $h D_{y^j}$, thus typical semiclassical vector fields. Finally, dividing them by $h$ and formally setting $h=0$, they are equal to $D_x$ and $x^{-1}D_y$, thus typical vector fields near a cone point. The span of the vector fields~\eqref{EqIVf} over a suitable space of smooth functions forms a Lie algebra, see Remark~\ref{RmkRCLie}; thus, the present paper can be regarded as a contribution to the investigation of classes of (pseudo)differential operators associated with Lie algebras of vector fields with controlled degenerations; see e.g.\ \cite{MazzeoMelroseHyp,MelroseDiffOnMwc,MazzeoMelroseFibred,AlbinGellRedmanDirac}.

More generally, one can work with weighted operators; we then have
\begin{equation}
\label{EqIElement}
  A_h = h^2\Delta_g + 1 \in \bigl(\tfrac{x}{x+h}\bigr)^{-2}\Psi_{\cop\semi}^2(X).
\end{equation}
Indeed, in the case of a product cone $[0,\infty)_x\times\pa X$ with metric $g_0=\dd x^2+x^2 k$, $k$ a fixed metric on $\pa X$, we have
\begin{equation}
\label{EqIOp}
  h^2\Delta_{g_0} + 1 = (h D_x)^2 - i(n-1)h x^{-1} h D_x + h^2 x^{-2}\Delta_k + 1,\quad D=\frac{1}{i}\partial;
\end{equation}
note then that $h^2 x^{-2}\Delta_k=(\tfrac{x}{x+h})^{-2}\cdot\bigl(\frac{h}{x+h}\bigr)^2\Delta_k$, with $\bigl(\frac{h}{x+h}\bigr)^2\Delta_k$ being an operator created out of the vector fields~\eqref{EqIVf} (with suitably smooth coefficients), similarly for the other terms of~\eqref{EqIOp}. The term $B(w)$ in Theorem~\ref{ThmI} then lies in $(\tfrac{x}{x+h})^{-w}\Psi_{\cop\semi}^w(X)$, while $C(w)$ has differential order $-\infty$ but nontrivial expansions at $\ff,\lb,\rb$, thus lies in the large calculus.

The construction of the inverse of $A_h$ in Theorem~\ref{ThmI} proceeds in three steps: \begin{enumerate*} \item a symbolic parametrix construction, analogous to the elliptic parametrix construction on closed manifolds; \item an improvement of the parametrix at $\ff$ (requiring inversion of a model operator $N_\ff(A)$) which is very similar to the construction of precise parametrices for elliptic b-operators in \cite[\S5]{MelroseAPS}; \item an improvement of the parametrix at $\tface$, which requires inversion of a model operator $N_\tface(A)$ which captures the transition between the elliptic b-behavior at $x=0$, $h>0$ and the conic behavior at $h=0$, $x\to 0+$. Formally, $N_\tface(A)$ is constructed by passing to the coordinates $(h,\hat x)$ with $\hat x=\frac{x}{h}$ and restricting to $h=0$.\end{enumerate*} In the setting of Theorem~\ref{ThmI}, this produces a shifted Laplacian on an \emph{exact} cone,
\[
  N_\tface(A) = \Delta_{g_0}+1 = D_{\hat x}^2 - i(n-1)\hat x^{-1}D_{\hat x} + \hat x^{-2}\Delta_{k(0)} + 1,\quad g_0=\dd\hat x^2+\hat x^2 k(0),
\]
on $[0,\infty)_{\hat x}\times\pa X$. This is a weighted elliptic b-differential operator at the cone point $\hat x=0$, and an elliptic scattering operator \cite{MelroseEuclideanSpectralTheory} at the large end of the cone, and its analysis is a standard application of the b- and scattering  calculi.

Next, following Seeley's approach, we define $A_h^{w/2}$ as the integral
\begin{equation}
\label{EqIContour}
  A_h^{w/2} = \frac{i}{2\pi}\int_\gamma \tilde\lambda^{w/2}(A_h-\tilde\lambda)^{-1}\,\dd\tilde\lambda
\end{equation}
along a suitable (unbounded) contour $\gamma$. Thus, one needs to understand the behavior of $(A_h-\tilde\lambda)^{-1}$ uniformly as $|\tilde\lambda|\to\infty$. Upon factoring out $|\tilde\lambda|^{-1}$ and defining $(\tilde h,\tilde\omega)=(|\tilde\lambda|^{-1/2},\tilde\lambda/|\tilde\lambda|)$, we must therefore control $(\tilde h^2 A_h-\tilde\omega)^{-1}$, the resolvent of a conic operator with \emph{two} semiclassical parameters. The development of a suitable pseudodifferential calculus $\Psi_{\cop\semi\tilde\semi}(X)$ for such operators is the technical heart of the paper. Roughly speaking, the inversion of elliptic elements of $\Psi_{\cop\semi\tilde\semi}(X)$ again requires the inversion of certain model operators, the most complicated one of which is now itself an elliptic semiclassical cone operator; in this sense, this analysis has an iterative character. We give details in~\S\ref{SC}. Theorem~\ref{ThmI} then follows from standard push-forward theorems for (polyhomogeneous) conormal distributions \cite{MelrosePushFwd,MelroseDiffOnMwc}. Obtaining the sharp structure and order of $A_h^{w/2}$ relies on a simple trick relating semiclassical (in $\tilde h$) and large parameter (in $\tilde\lambda$) symbol expansions; as this appears not to have been explicitly described in the literature before, we explain this in~\S\ref{SsPP} for powers of differential operators on closed manifolds.

The structure of the paper is as follows. In \S\ref{SOp}, we define the class of fully elliptic (semiclassical) cone differential operators, following \cite{LoyaConicResolvent,GilConicHeat}. In~\S\ref{SR}, we develop the semiclassical cone ps.d.o.\ calculus; in particular, we give a streamlined proof of (a generalization of) the main result of \cite{LoyaConicResolvent} in~\S\ref{SsRR}. The doubly semiclassical cone calculus is developed in~\S\ref{SC} and subsequently used in~\S\ref{SP} in the construction of complex powers of semiclassical cone operators. The application to complex powers of semiclassical conic Laplacians and propagation estimates is given in~\S\ref{SL}.

The results regarding inversions and complex powers presented here admit straightforward extensions to the case of semiclassical cone operators acting on sections of a smooth vector bundle on $X$. Moreover, one can allow $X$ to have any finite number of cone points. We leave the (notational) details to the reader.

\subsection*{Acknowledgments}

I am grateful to Jared Wunsch, Andr\'as Vasy, and Boris Vertman for helpful conversations, and to Chen Xi for sharing his manuscript \cite{XiConeParametrix} with me. I gratefully acknowledge support from the NSF under Grant No.\ DMS-1955614 and from a Sloan Research Fellowship. Part of this research was conducted during the time I served as a Clay Research Fellow.

\section{Semiclassical cone differential operators}
\label{SOp}

Denote by $X$ a connected compact $n$-dimensional manifold with connected and embedded boundary. Fix a boundary defining function $x\in\CI(X)$, so $\pa X=x^{-1}(0)$, $\dd x\neq 0$ at $\pa X$.

\begin{definition}
\label{DefOp}
  Let $m\in\N_0$. Then
  \[
    \Diff_{\cop,\semi}^m(X) := \sum_{j\leq m} \Bigl(\frac{h}{x}\Bigr)^j\CI\bigl([0,1)_h;\Diffb^j(X)\bigr).
  \]
  For $A\in\Diff_{\cop,\semi}^m(X)$ and $h_0\in(0,1)$, $A_{h_0}\in x^{-m}\Diffb^m(X)$ is the restriction of $A$ to $h=h_0$.
\end{definition}

This space of operators is independent of the choice of $x$. An interesting class of examples is given by operators of the form
\begin{equation}
\label{EqOp}
  A_{h,\omega} := h^m x^{-m} A_\bop - \omega,\quad
  A_\bop\in\Diffb^m(X),\ 
  h\in(0,1],\ \omega\in\C;
\end{equation}
for $m=2$, this class includes the semiclassical conic Laplace operators in Theorem~\ref{ThmI} in view of the expressions~\eqref{EqIOp} (for exact cones) and~\eqref{EqLOp} (for general cones).

Denote local coordinates on $\pa X$ by $y\in\R^{n-1}$. Then in a collar neighborhood $[0,\eps)_x\times\pa X$ of $\pa X$, a general element $A\in\Diff_{\cop,\semi}^m(X)$ takes the form
\begin{equation}
\label{EqOpLocal}
  A = \sum_{k+|\alpha|\leq j\leq m} \Bigl(\frac{h}{x}\Bigr)^j a_{j k \alpha}(h,x,y)(x D_x)^k D_y^\alpha,\quad a_{j k \alpha}(h,x,y)\in\CI([0,1)_h\times[0,\eps)_x\times\R^{n-1}_y).
\end{equation}

There is a standard notion of ellipticity for the b-differential operators $x^m A_h$, $h>0$, namely the nonvanishing of $\sum_{k+|\alpha|=m}a_{m k\alpha}(h,x,y)\xi_\bop^k\eta_\bop^\alpha$ for $(0,0)\neq(\xi_\bop,\eta_\bop)\in\R\times\R^{n-1}$. Similarly, in $x>x_0>0$, i.e.\ away from $\pa X$, the $h$-dependent operator $A$ is a semiclassical operator, $A\in\Diff_h^m(X^\circ)$, for which there is again a standard notion of ellipticity which requires that the semiclassical principal symbol $\sum_{k+|\alpha|=j\leq m} x^{-j}a_{j k\alpha}(0,x,y)\xi^k\eta^\alpha$ be bounded from below in absolute value by $c_0(1+|\xi|+|\eta|)^m$, $c_0>0$, for all $(\xi,\eta)\in\R\times\R^{n-1}$.

Fredholm or invertibility properties of elliptic b-operators depend on a choice of weight for which the b-normal operator is invertible. Note that
\[
  x^m A_h = \sum_{k+|\alpha|\leq m} h^m a_{m k \alpha}(h,x,y)(x D_x)^k D_y^\alpha + \sum_{k+|\alpha|\leq j\leq m-1} x^{m-j}h^j a_{j k\alpha}(h,x,y)(x D_x)^k D_y^\alpha.
\]
hence the normal operator of $x^m A_h$, given by freezing coefficients at $x=0$, only depends on the coefficients $a_{j k\alpha}(h,0,y)$ for $j=m$.

\begin{definition}
\label{DefOpNorm}
  Fix a collar neighborhood of $\pa X$. The b-normal operator of $A\in\Diff_{\cop,\semi}^m(X)$ is then the $h$-dependent operator
  \[
    N_{h,\pa X}(A):=x^{-m}N_{\pa X}(h^{-m}x^m A_h)\in\CI\bigl([0,1);x^{-m}\Diffb^m([0,\infty)_x\times\pa X)\bigr).
  \]
  If $N_{h,\pa X}(A)$ is independent of $h$, we put, for any $h_0\in(0,1)$,
  \[
    N_{\pa X}(A) := N_{h_0,\pa X}(A) \in x^{-m}\Diffb^m([0,\infty)_x\times\pa X).
  \]
\end{definition}

In terms of~\eqref{EqOpLocal}, we have $N_{h,\pa X}(A)=x^{-m}\sum_{k+|\alpha|\leq m} a_{m k\alpha}(h,0,y)(x D_x)^k D_y^\alpha$, which is $h$-independent if and only if the coefficients $a_{m k\alpha}(h,0,y)$ are $h$-independent; all semiclassical cone operators $A$ in this paper will have $h$-independent b-normal operators. For such $A$, and in the notation of Appendix~\ref{SB}, we recall the definition of its \emph{boundary spectrum}:
\[
  \specb(A) := \bigl\{ \sigma \in \C \colon \bigl(x^m N_{\pa X}(A)\bigr)\ftrans(\sigma)\in\Diff^m(\pa X)\ \text{is not invertible on $\CI(\pa X)$} \bigr\}.
\]

Lastly, when both $h$ and $x$ are small, let us introduce the rescaled variable $\hat x=x/h$, then
\[
  A = \sum_{k+|\alpha|\leq j\leq m} \hat x^{-j} a_{j k\alpha}(h,h\hat x,y)(\hat x D_{\hat x})^k D_y^\alpha.
\]
This can be restricted to $h=0$, giving rise to the model operator
\begin{equation}
\label{EqOpNtf}
  N_\tface(A) := \sum_{k+|\alpha|\leq j\leq m} \hat x^{-j} a_{j k\alpha}(0,0,y)(\hat x D_{\hat x})^k D_y^\alpha\quad\text{on}\ \ol{{}^+N\pa X} := [0,\infty]_{\hat x}\times\pa X.
  \end{equation}
Here we compactified $[0,\infty)_{\hat x}$ radially at infinity, with $\hat x^{-1}$ a boundary defining function of $\{\infty\}$. Note that in $\hat x<3$, we have
\[
  \hat x^m N_\tface(A) = \sum_{k+|\alpha|\leq j\leq m} \hat x^{m-j} a_{j k\alpha}(0,0,y)(\hat x D_{\hat x})^k D_y^\alpha \in \Diffb^m([0,3)_{\hat x}\times\pa X),
\]
which is thus a b-differential operator with smooth coefficients (and has normal operator equal to the restriction to $h=0$ of the pullback of $N_{h,\pa X}(x^m A)$ along $(h,\hat x)\mapsto(h,h\hat x)$). On the other hand, for $\hat x>1$, we have
\begin{equation}
\label{EqOpNtf2}
  N_\tface(A) = \sum_{k+|\alpha|\leq j\leq m}\hat x^{k+|\alpha|-j}a_{j k\alpha}(0,0,y)\bigl(\hat x^{-k}(\hat x D_{\hat x})^k\bigr) (\hat x^{-1}D_y)^\alpha,
\end{equation}
which is thus an unweighted $m$-th order scattering operator near $\hat x=\infty$, that is, a finite sum of up to $m$-fold compositions of scattering vector fields $D_{\hat x}$, $\hat x^{-1}D_y$ with smooth (down to $\hat x^{-1}=0$) coefficients. A natural function space for $N_\tface(A)$ to act on is thus:

\begin{definition}
  Fix a cutoff $\chi\in\CIc([0,3)_{\hat x})$, $\chi\equiv 1$ on $[0,2]$, and a smooth volume density $\mu_{\pa X}$ on $\pa X$. Then the space\footnote{This space is called a \emph{cone Sobolev space} in \cite[\S6]{LoyaConicResolvent}.} $H_{\bop,\scop}^{s,\alpha}([0,\infty]_{\hat x}\times\pa X)$ consist of all distributions $u$ such that
  \[
    \chi u \in \hat x^\alpha\Hb^s\bigl([0,3)_{\hat x} \times \pa X; \left|\tfrac{\dd\hat x}{\hat x}\right|\mu_{\pa X}\bigr),\quad
    (1-\chi)u \in \Hsc^s\bigl((1,\infty]_{\hat x}\times \pa X, \hat x^{n-1}|\dd\hat x|\mu_{\pa X}\bigr).
  \]
  Here, for $s\in\N_0$, $\Hb^s$ consists of distributions which remain in $L^2$ upon application of up to $s$ smooth vector fields tangent to $\hat x^{-1}(0)$ (that is, $\hat x D_{\hat x}$ and vector fields on $\pa X$), while $\Hsc^s$ consists of distributions which remain in $L^2$ upon application of up to $s$ smooth scattering vector fields, i.e.\ vector fields of the form $\hat x^{-1}V$ with $V$ tangent to $\hat x=\infty$ (that is, $D_{\hat x}$ and $\hat x^{-1}$ times b-vector fields on $\pa X$).
\end{definition}

The notion of full ellipticity is then:
\begin{definition}
\label{DefOpEll}
  Let $\alpha\in\R$. We say that $A\in\Diff_{\cop,\semi}^m(X)$ with $h$-independent b-normal operator $N_{\pa X}(A)$ is \emph{fully elliptic at weight $\alpha$} if the following conditions hold:
  \begin{enumerate}
  \item\label{ItOpEll1} For all $h>0$, $x^m A_h\in\Diffb^m(X)$ is elliptic as a b-differential operator on $X$.
  \item\label{ItOpEll2} Restricted to $X^\circ$, $A\in\Diffh(X^\circ)$ is elliptic as a semiclassical differential operator on $[0,1)_h\times X^\circ$.
  \item\label{ItOpEll3} $\alpha\notin-\Im\specb(A)$, and $N_\tface(A)\colon H_{\bop,\scop}^{s,\alpha}(\ol{{}^+N\pa X})\to H_{\bop,\scop}^{s-m,\alpha-m}(\ol{{}^+N\pa X})$ is invertible for some $s\in\R$, and elliptic as a b-scattering operator. (Restricted to ${}^+N\pa X=[0,\infty)_{\hat x}\times\pa X$, $\hat x^m N_\tface(A)$ is an elliptic b-differential operator, and the scattering principal symbol of $N_\tface(A)$ at $\hat x^{-1}(\infty)$ is elliptic, meaning in terms of~\eqref{EqOpNtf2} that $|\sum_{k+|\alpha|=j\leq m} a_{j k\alpha}(0,0,y)\xi_\scop^k\eta_\scop^\alpha|\geq c_0(1+|\xi_\scop|+|\eta_\scop|)^m$, $c_0>0$, for all $(\xi_\scop,\eta_\scop)\in\R\times\R^{n-1}$.)
  \end{enumerate}
\end{definition}

\begin{rmk}
\label{RmkOpEll}
  Condition~\eqref{ItOpEll3} is equivalent to just the invertibility of $N_\tface(A)$: by standard elliptic b-theory, the invertibility of $N_\tface(A)$ in part~\eqref{ItOpEll3} in fact \emph{implies} $\alpha\notin-\Im\specb(A)$ (see \cite[\S6.2]{MelroseAPS}), and the symbolic ellipticity follows by oscillatory testing.
\end{rmk}

See~\S\ref{SsRC} for an cleaner formulation of conditions~\eqref{ItOpEll1}--\eqref{ItOpEll3} and the definition~\eqref{EqOpNtf} of the normal operator, using in particular a notion of principal symbol which captures the ellipticity of $A$ in the b-, semiclassical, and b-scattering senses. See~\S\ref{SL} for the detailed analysis of shifted conic Laplacians $h^2\Delta_g+1$ (with $m=2$) and the verification of the conditions of Definition~\ref{DefOpEll} in this context.

\begin{rmk}
  We are \emph{not} studying operators which are elliptic in the sense that $x^m A$ is an elliptic element of $\Diffbh^m(X)$, such as $x^{-2}(h^2\Delta_\bop+1)$, $m=2$. For such an operator, the boundary spectrum depends nontrivially on $h$, and the Schwartz kernel of its inverse fails to be polyhomogeneous at $[0,1)_h\times(\lb\cup\rb)$. See also \cite{LoyaBResolvent}.
\end{rmk}

\section{Semiclassical cone calculus and semiclassical resolvent}
\label{SR}

Fix an operator $A\in\Diff_{\cop,\semi}^m(X)$ which is fully elliptic at weight $\alpha\in\R$ in the sense of Definition~\ref{DefOpEll}. In this section, we shall give a precise description the semiclassical resolvent $A_h^{-1}$. The main result is Theorem~\ref{ThmR} below, describing $A_h^{-1}$ as an element of the (large) semiclassical cone calculus $\Psi_{\cop\semi}(X)$. Complex powers of $A_h$ will lie in the same calculus, as we will show in~\S\ref{SsPh}. The associated scale of Sobolev spaces is defined in~\S\ref{SsRM}.

We remark that in the special case $A=h^m A_\cop-\omega$, $A_\cop=x^{-m}A_\bop$, $A_\bop\in\Diffb^m(X)$ as in~\eqref{EqOp}, we have $(h^{-m}A_h)^{-1} = (A_\cop-\lambda)^{-1}$ for $\lambda=h^{-m}\omega$, which is the resolvent of an operator in a calculus of cone operators with large parameter $\lambda$. Now, large parameter ps.d.o.s are slightly more precise than semiclassical ps.d.o.s in that the latter do not encode the joint symbolic behavior in $(\zeta,\lambda)$ (with $\zeta$ denoting a suitable `cone momentum') and rather only keep track of symbolic behavior in $\zeta$ and powers of $h\sim|\lambda|^{-1/m}$ separately. (See also \cite[\S2.1]{VasyMicroKerrdS} for more on the relationship.) Therefore, Loya's work \cite{LoyaConicResolvent} gives a slightly more precise description of $A^{-1}$, though the precision lost here is easily recovered, cf.\ \S\ref{SsPP}. The benefit of working with semiclassical ps.d.o.s is the geometric simplicity of their Schwartz kernels: they are conormal distributions on a suitable double space, here a resolution of $[0,1)_h\times X^2$. By contrast, the Fourier transform in $\zeta$ does \emph{not} identify the space of joint symbols in $(\zeta,\lambda)$ with a parameterized (by $\lambda$) space of conormal distributions. By working semiclassically, we can thus streamline the arguments of~\cite{LoyaConicResolvent}, yet maintaining the precision of~\cite{LoyaConicPower} in the description of complex powers in~\S\ref{SP}. We will also simplify the treatment of the normal operator $N_\tface(A)$ arising in~\eqref{EqOpNtf} compared to \cite{LoyaConicResolvent} by analyzing it within the b-scattering calculus, see in particular~\eqref{EqRCNtf}.

\subsection{Definition of the semiclassical cone calculus}
\label{SsRC}

We denote by $X$ a connected compact $n$-dimensional manifold with connected embedded boundary $\pa X$.

\subsubsection{Double space; Schwartz kernels}

Recall the definition of the \emph{b-double space}
\begin{equation}
\label{EqRCX2b}
  X^2_\bop := [X^2;(\pa X)]^2,
\end{equation}
whose boundary hypersurfaces are denoted $\lb_\bop$ (the lift of $\pa X\times X$), $\rb_\bop$ (the lift of $X\times\pa X$), and $\ff_\bop$ (the front face). The b-diagonal $\Delta_\bop\subset X^2_\bop$ is the lift of the diagonal $\Delta_X\subset X^2$.

\begin{definition}
\label{DefRCDouble}
  The \emph{semiclassical cone double space} is
  \[
    X^2_{\cop\semi} := \left[ [0,1)_h\times X^2_\bop; \{0\}\times\ff_\bop; \{0\}\times\Delta_\bop \right].
  \]
  Its boundary hypersurfaces are denoted $\lb$ (lift of $[0,1)_h\times\lb_\bop$), $\ff$ (lift of $\{0\}\times\ff_\bop$), $\rb$ (lift of $[0,1)_h\times\rb_\bop$), $\tface$ (`transition face', front face of the first blow-up), $\dface$ (`diagonal face', front face of the second blow-up), and $\sface$ (`semiclassical face', lift of $h=0$). By $\Delta_{\cop\semi}$ we denote the lift of $[0,1)_h\times\Delta_\bop$. See Figure~\ref{FigRCDouble}.
\end{definition}

We denote by $\rho_\lb$, $\rho_\ff$, etc.\ defining functions of $\lb$, $\ff$, etc. Denote by $\pi_\bop\colon X^2_{\cop\semi}\to X^2_\bop$ the lift of the projection $[0,1)_h\times X^2_\bop\to X^2_\bop$. The space of semiclassical cone ps.d.o.s is then:

\begin{definition}
\label{DefRCPsdo}
  Let $m\in\R$. Then\footnote{Locally, such elements are Fourier transforms of symbols of $m$; the shift of the order by $\tfrac14$ is due to the standard normalization of conormal distributions, see \cite[\S2.4]{HormanderFIO1}.}
  \[
    \Psi_{\cop\semi}^m(X) := \left\{ \kappa\in I^{m-\mfrac14}\left(X^2_{\cop\semi};\Delta_{\cop\semi};\rho_\dface^{-n}\pi_\bop^*\bigl(\Omegab^\mhalf(X^2_\bop)\bigr)\right) \colon \kappa\equiv 0\ \text{at}\ \lb\cup\rb\cup\sface \right\},
  \]
  where `$\equiv$' means equality of Taylor series; that is, $\kappa$ vanishes to infinite order at $\lb$, $\rb$, and $\sface$.
  If $\cE_\lb,\cE_\ff,\cE_\rb,\cE_\tface\subset\C\times\N_0$ are index sets and $\cE=(\cE_\lb,\cE_\ff,\cE_\rb,\cE_\tface)$, we set
  \[
    \Psi_{\cop\semi}^{-\infty,\cE}(X) := \cA_\phg^\cE(X^2_{\cop\semi}),
  \]
  where $\cA_\phg^\cE(X^2_{\cop\semi})$ consists of polyhomogeneous distributions with specified index sets at $\lb,\ff,\rb,\tface$, and with index set $\emptyset$ (not made explicit in the notation) at $\dface$ and $\sface$.
\end{definition}

Schwartz kernels of elements of $\Psi_{\cop\semi}^{-\infty,\cE}(X)$ are pullbacks to $X^2_{\cop\semi}$ of polyhomogeneous distributions on the simpler space
\begin{equation}
\label{EqRCPsdoRes}
  X^2_{\cop\semi,\infty} = \left[ [0,1)_h\times X^2_\bop; \{0\}\times\ff_\bop \right]
\end{equation}
vanishing to infinite order at the lift of $h=0$. The lift of $[0,1)_h\times\Delta_\bop$ to $X^2_{\cop\semi,\infty}$ is denoted
\begin{equation}
\label{EqRCPsdoResDiag}
  \Delta_{\cop\semi,\infty} \subset X^2_{\cop\semi,\infty}.
\end{equation}

To make Definition~\ref{DefRCPsdo} concrete for elements of $\Diff_{\cop,\semi}(X)$, let us work in local coordinates $(x,y)$ on $X$ lifted to the left factor of $X^2$, and with $(x',y')$ denoting the lifts to the right factor. We can use $(h,x',s=x/x',y,y')$ as local coordinates on $[0,1)_h\times(X_\bop^2\setminus\rb)$ in which the Schwartz kernel of, say, $A_h=h^m x^{-m}(x D_x)^j P_{m-j}(y,D_y)$, $P_{m-j}\in\Diff^{m-j}(\pa X)$, acting on b-half-densities, is given by
\begin{subequations}
\begin{equation}
\label{EqRCEx1}
  A_h = h^m (x')^{-m}s^{-m}(s D_s)^j P_{m-j}(y,D_y)\bigl(\delta(s-1)\delta(y-y')\bigr)\left|\frac{\dd s}{s}\frac{\dd x'}{x'}\dd y\,\dd y'\right|^{\mfrac12}.
\end{equation}
Near $\tface\cap\ff$ and in coordinates $(h,\hat x'=\frac{x'}{h},s,y,y')$, this is equal to
\begin{equation}
\label{EqRCEx2}
  A_h = (\hat x')^{-m}s^{-m}(s D_s)^j P_{m-j}(y,D_y)\bigl(\delta(\log s)\delta(y-y')\bigr)\left|\frac{\dd s}{s}\frac{\dd\hat x'}{\hat x'}\dd y\,\dd y'\right|^{\mfrac12},
\end{equation}
with $\hat x'$ a local defining function of $\ff$. Near $\tface\cap\dface$ and in coordinates $(\hat h=\tfrac{h}{x'},x',s_h=\hat h^{-1}\log s,y,Y_h=\frac{y-y'}{\hat h})$, this is
\begin{equation}
\label{EqRCEx3}
  A_h = e^{-m\hat h s_h}D_{s_h}^j \hat h^{m-j}P_{m-j}(y,\hat h^{-1}D_{Y_h})\bigl(\delta(s_h)\delta(Y_h)\bigr)\cdot \hat h^{-n}\left|\frac{\dd s}{s}\frac{\dd\hat x'}{\hat x'}\dd y\,\dd y'\right|^{\mfrac12}.
\end{equation}
\end{subequations}
The operator $A_h$ is thus an element of $\rho_\ff^{-m}\Psi_{\cop\semi}^m(X)$. This implies the relationship
\begin{equation}
\label{EqRCDiffPsdo}
  \Diff_{\cop,\semi}^m(X) \subset \rho_\ff^{-m}\Psi_{\cop\semi}^m(X) = \bigl(\tfrac{x}{x+h}\bigr)^{-m}\Psi_{\cop\semi}^m(X).
\end{equation}
One could use different normalizations in the definition of $\Psi_{\cop\semi}^m(X)$ (e.g.\ by multiplying the Schwartz kernels by $\rho_\ff^{-m}$); we choose the present normalization as it directly matches the familiar b-calculus for $h$ bounded away from $0$ (see Remark~\ref{RmkRCOther} below), and as it leads to simpler definitions of symbol and normal operator maps later on.

From the calculations~\eqref{EqRCEx1}--\eqref{EqRCEx3}, one can also deduce that the vector fields $\frac{h}{h+x}x D_x$, $\frac{h}{h+x}D_{y^j}$ from~\eqref{EqIVf} are nondegenerate elements of $\Psi_{\cop\semi}^1(X)$, that is, their principal symbols are nonzero linear functions in every fiber of $N^*\Delta_{\cop\semi}$. Thus, a choice of boundary defining function $x\in\CI(X)$ gives a bundle isomorphism $N^*\Delta_{\cop\semi}\cong[[0,1)_h\times\Tb^*X;\{0\}\times\Tb_{\pa X}^*X]$, given by identifying the dual bundles via mapping $h$-independent b-vector fields $V\in\Vb(X)$ to the lifts to the left factor in $X^2_{\cop\semi}$ of $\frac{h}{h+x}V$. See Remark~\ref{RmkRCLie} below for a better, invariant, point of view.

We stress that $\Psi_{\cop\semi}(X)$ is \emph{not} a resolution of the algebra $\Psibh(X)$ of semiclassical b-ps.d.o.s: the semiclassical parameter of the latter is $h$, compared to $\frac{h}{h+x}$ for semiclassical cusp ps.d.o.s. The fact that $\frac{h}{h+x}$ does not vanish at $\tface$ causes the failure of the $\tface$-normal operator homomorphism to map into a commutative algebra, as discussed below. (By contrast, the vanishing of $\frac{h}{h+x}$ at $\dface$ is the origin of the leading order commutativity at $\dface$ captured by~\eqref{EqRCSy}.) One may thus think of $\Psi_{\cop\semi}^m(X)$ as consisting of operators of the form
\[
  a\bigl(h,x,y,\tfrac{h}{h+x}x D_x,\tfrac{h}{h+x}D_y\bigr),
\]
where $a(h,x,y,\xi,\eta)$ is a symbol of order $m$ in $(\xi,\eta)$ with smooth dependence on the point $z=(h,x,y)\in X_{\cop\semi}$, where
\begin{equation}
\label{EqRCSingle}
  X_{\cop\semi}:=[[0,1)_h\times X;\{0\}\times\pa X]
\end{equation}
is the \emph{semiclassical cone single space}, discussed in detail in~\S\ref{SsRM}.

\begin{rmk}
\label{RmkRCLie}
  A simple calculation shows that the span of $\frac{h}{h+x}\Vb(X)$ over $\CI(X_{\cop\semi})$ is a Lie algebra (which is independent of the choice of $x$) which one could reasonably call $\cV_{\cop\semi}(X_{\cop\semi})$. (The corresponding algebra of differential operators contains $(\frac{x}{x+h})^m\Diff_{\cop,\semi}^m(X)$ for all $m$, cf.\ \eqref{EqRCDiffPsdo}.) This is then the space of smooth sections of a natural vector bundle ${}^{\cop\semi}T X_{\cop\semi}\to X_{\cop\semi}$; the corresponding cotangent bundle ${}^{\cop\semi}T^*X_{\cop\semi}$ is then naturally isomorphic to $N^*\Delta_{\cop\semi}$. However, operators occurring in practice are typically expressed as in Definition~\ref{DefOp} rather than in terms of the vector fields~\eqref{EqIVf}, hence we do not develop this invariant point of view further here.
\end{rmk}

\begin{rmk}
\label{RmkRCOther}
  The relationship of the semiclassical cone calculus with other ps.d.o.\ calculi is the following. Let $A\in\Psi_{\cop\semi}^m(X)+\Psi_{\cop\semi}^{-\infty,\cE}(X)$, identified with its Schwartz kernel, and let $\phi\in\CI([0,\infty))$ be identically $0$ near $0$.
  \begin{enumerate}
  \item In $h>0$, $A$ is a smooth family of b-pseudodifferential operators,
    \[
      \phi(h)A \in \CI\left((0,1)_h;\Psib^m(X)+\Psib^{-\infty,\cE}(X)\right).
    \]
  \item\label{ItRCOtherB} Localizing away from the preimage of $x=0$ or $x'=0$, $A$ is a semiclassical b-pseudodifferential operator,
    \begin{equation}
    \label{EqRCOtherB}
    \begin{split}
      \phi(x)A &\in x^\infty\Psibh^m(X) + h^\infty\Psibh^{-\infty,(\emptyset,\emptyset,\cE_\rb)}(X), \\
      \phi(x')A &\in x^\infty\Psibh^m(X) + h^\infty\Psibh^{-\infty,(\cE_\lb,\emptyset,\emptyset)}(X).
    \end{split}
    \end{equation}
    Here, $\Psibh$ consists of distributions conormal to the lift of $[0,1)_h\times\Delta_\bop$ to the semiclassical b-double space $[[0,1)_h\times X^2_\bop;\{0\}\times\Delta_\bop]$, see \cite[Appendix~A]{HintzVasyKdSStability}.
  \item As a consequence of~\eqref{ItRCOtherB}, $\phi(x)\phi(x')A\in\Psih^m(X^\circ)$ is a semiclassical ps.d.o.
  \end{enumerate}
\end{rmk}

\subsubsection{Symbol and normal operator maps}
\label{SssRCSymbol}

There are a number of symbol and normal operator maps for semiclassical cone ps.d.o.s. The high frequency principal symbol is valued in $(S^m/S^{m-1})(N^*\Delta_{\cop\semi})$. It captures $A\in\Psi_{\cop\semi}^m(X)$ modulo $\Psi_{\cop\semi}^{m-1}(X)$. As usual for semiclassical operators, there is an additional (commutative) symbol which captures $A$ modulo $\rho_\dface\Psi_{\cop\semi}^m(X)$. For a fixed choice of $\rho_\dface$, it is invariantly obtained by restricting $\kappa\in\Psi_{\cop\semi}^m(X)$ to $\dface$ which is naturally identified with the radial compactification $\ol{\Tb}X$, and Fourier transforming in the fibers of $\ol{\Tb}X$ (using the fiber density induced by the half-density of the Schwartz kernel to integrate). Since this computes the restriction of the full symbol of $\kappa$ to $\dface$, we combine this symbol map with the high frequency principal symbol to the principal symbol\footnote{Over $x>0$, this is the usual semiclassical principal symbol ${}^\semi\sigma_m\colon\Psih^m(X^\circ)\to(S^m/h S^{m-1})(N^*\Delta_{X^\circ})$.}
\[
  \sigmach_m \colon \Psi_{\cop\semi}^m(X) \to (S^m/\rho_\dface S^{m-1})(N^*\Delta_{\cop\semi}).
\]
This fits into the short exact sequence
\begin{equation}
\label{EqRCSy}
  0 \to \rho_\dface\Psi_{\cop\semi}^{m-1}(X) \to \Psi_{\cop\semi}^m(X) \xra{\sigmach_m} (S^m/\rho_\dface S^{m-1})(N^*\Delta_{\cop\semi}) \to 0.
\end{equation}

Next, restriction to $\ff\cong[0,1)_h\times\ff_\bop$ gives the usual b-normal operator, with smooth dependence on the parameter $h$,
\begin{equation}
\label{EqRCNff}
  N_\ff \colon \Psi_{\cop\semi}^m(X) \to \CI\bigl([0,1)_h; \Psi_{\bop,I}^m({}^+N\pa X) \bigr),
\end{equation}
where ${}^+N\pa X$ is the inward pointing normal bundle of $X$ (including the zero section), and $\Psi_{\bop,I}$ denotes the space of b-ps.d.o.s whose Schwartz kernel is invariant under the lift of the diagonal action of $\R_+$ on ${}^+N\pa X\times{}^+N\pa X$ to $({}^+N\pa X)^2_\bop$. (The definition~\eqref{EqRCNff} uses that there is a natural diffeomorphism of $\ff_\bop$ and the front face of $({}^+N\pa X)^2_\bop$, and that there is a natural restriction map $\Omegab^\mhalf(X^2_\bop)|_{\ff_\bop}\cong\Omegab^\mhalf(\ff_\bop)$, see~\cite[Equation~(4.90)]{MelroseAPS}.) The corresponding short exact sequence is
\[
  0 \to \rho_\ff\Psi_{\cop\semi}^m(X) \to \Psi_{\cop\semi}^m(X) \xra{N_\ff} \CI\bigl([0,1)_h;\Psi_{\bop,I}^m({}^+N\pa X)\bigr) \to 0.
\]

\begin{rmk}
\label{RmkRCEll}
  For $A\in\Diff_{\cop,\semi}^m(X)$, the symbolic ellipticity conditions in all three parts of Definition~\ref{DefOpEll} taken together are equivalent to the ellipticity of the principal symbol $\sigmach_m(\rho_\ff^m A)$, as follows from an inspection of~\eqref{EqRCEx1}--\eqref{EqRCEx3}. Moreover, the first part of condition~\eqref{ItOpEll3} of Definition~\ref{DefOpEll} implies the invertibility of $N_\ff(\rho_\ff^m A)$ on weighted b-Sobolev spaces on $\ol{{}^+N\pa X}$ with weight $-\alpha$. (Note that in a trivialization of ${}^+N\pa X$ via the choice of a boundary defining function, $N_\ff(\rho_\ff^m A)$ is a smooth positive multiple of a conjugation of $N_{\pa X}(A)$ by a smooth nonzero function on $[0,1)\times\pa X$.) Together, this gives a symbolic characterization of full ellipticity.
\end{rmk}

The third and final model operator arises by restriction to $\tface$ as in~\eqref{EqOpNtf}; to obtain a convenient description, we shall first show that $\tface$ is, in a natural fashion, the double space for a ps.d.o.\ algebra consisting of operators which are b-operators at $\hat x=0$ and scattering operators at $\hat x=\infty$ in the notation of~\eqref{EqOpNtf}. The underlying manifold for this algebra will be the compactification $\ol{{}^+N\pa X}$, defined in a trivialization ${}^+N\pa X\cong[0,\infty)_{\hat x}\times\pa X$ as $[0,\infty]\times\pa X$, where $[0,\infty]=([0,\infty)_{\hat x}\sqcup[0,\infty)_{\hat h})/\sim$ with equivalence relation $0<\hat x\sim\hat h=\hat x^{-1}$. Its two boundary hypersurfaces are $\pa_0\ol{{}^+N\pa X} =\hat x^{-1}(0)$, $\pa_\infty\ol{{}^+N\pa X} = \hat h^{-1}(0)$.

We first observe then that
\[
  \tface' := \text{front face of}\ \left[[0,1)\times X^2_\bop;\{0\}\times\ff_\bop\right] \cong \text{front face of}\ \left[[0,\infty)\times({}^+N\pa X)^2_\bop;\{0\}\times\ff_\bop\right]
\]
is canonically diffeomorphic to the radial compactification $\ol{({}^+N\pa X)^2_\bop}$.\footnote{One can take this as the definition of the radial compactification, similarly to~\cite[Equation~(4.92)]{MelroseAPS}. Alternatively, if ${}^+N\pa X\cong[0,\infty)_x\times\pa X$ is a local trivialization, giving $({}^+N\pa X)^2\cong[0,\infty)_x\times[0,\infty)_{x'}\times(\pa X)^2$, we have $({}^+N\pa X)^2_\bop\cong[0,\infty)_r\times(\Sph^1_{+ +})_\theta\times(\pa X)^2$ where $r=|(x,x')|$, $\theta=(x,x')/|(x,x')|$ (lying in the first quadrant), and then $\ol{({}^+N\pa X)^2_\bop}=[0,\infty]_r\times(\Sph^1_{+ +})_\theta\times(\pa X)^2$.} The latter has 4 boundary hypersurfaces which we denote $\lb',\ff',\rb',\ff'_\infty$: the first three are the closures of the left boundary, front face, and right boundary of $({}^+N\pa X)^2_\bop$ in the radial compactification, and $\ff'_\infty$ is the boundary `at infinity' (i.e.\ the complement of $({}^+N\pa X)^2_\bop$ in the radial compactification), see the left panel of Figure~\ref{FigRCtfM1}. Denote by $\Delta_\bop'$ the closure of the lift of the diagonal in $({}^+N\pa X)^2$ to $\ol{({}^+N\pa X)^2_\bop}$.

\begin{figure}[!ht]
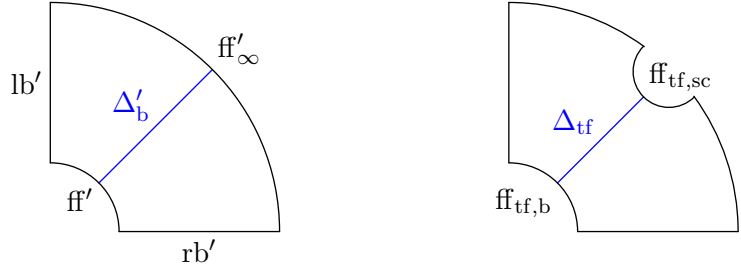

\inclfig{FigRCtfM1}
\caption{\textit{Left:} the radial compactification $\ol{({}^+N\pa X)^2_\bop}\cong\tface'$. \textit{Right:} its blow-up along the b-diagonal at infinity is naturally diffeomorphic to the transition face $\tface$.}
\label{FigRCtfM1}
\end{figure}

Now, $\tface\subset X^2_{\cop\semi}$ arises by blowing up the intersection of $\tface'$ with the lift of $\{0\}\times\Delta_\bop$. We have thus established a natural diffeomorphism
\begin{equation}
\label{EqRCtfModel}
  \tface \cong (\ol{{}^+N\pa X})^2_{\bop,\scop} := \left[ \ol{({}^+N\pa X)^2_\bop}; \Delta_\bop'\cap\ff'_\infty \right]
\end{equation}
of $\tface$ with the double space for operators which are b-operators at $\pa_0\ol{{}^+N\pa X}$ and scattering operators at $\pa_\infty\ol{{}^+N\pa X}$.\footnote{The definition~\eqref{EqRCtfModel} is more economical than the standard definition of the b-scattering-double space, which is the blow-up of the b-double space of $\ol{{}^+N\pa X}$, defined as $(\ol{{}^+N\pa X})^2_\bop=[(\ol{{}^+N\pa X})^2;(\pa_0\ol{{}^+N\pa X})^2,(\pa_\infty\ol{{}^+N\pa X})^2]$, at the intersection of the lifted diagonal with $(\pa_\infty\ol{{}^+N\pa X})^2$; let us call the resulting space the \emph{extended b-scattering-double space} ${}'(\ol{{}^+N\pa X})^2_{\bop,\scop}$. This is the blow-up of $(\ol{{}^+N\pa X})^2_{\bop,\scop}$ at (the lifts of) $\lb'\cap\ff'_\infty$ and $\rb'\cap\ff'_\infty$, i.e.\ the blow-up of the top left and bottom right corners in the right panel of Figure~\ref{FigRCtfM1}. But Schwartz kernels of our large b-scattering calculus vanish to infinite order at the lift of $\ff'_\infty$, hence passage to the extended double space ${}'\ol{({}^+N\pa X)}^2_{\bop,\scop}$, and demanding infinite order vanishing of Schwartz kernels at the two new front faces, does not enlarge the space of b-scattering operators. We remark, with foresight, that if we consider the \emph{extended} semiclassical cone double space ${}'X^2_{\cop\semi}$, defined in Definition~\ref{DefRCExt} below, its transition face, which arises from $\tface$ by blowing up $\tface\cap(\sface\cap\lb)$ and $\tface\cap(\sface\cap\rb)$, is naturally diffeomorphic to ${}'(\ol{{}^+N\pa X})^2_{\bop,\scop}$; and it is this latter double space which is a more convenient double space for the study of compositions of b-scattering ps.d.o.s; in the semiclassical cone context, this foreshadows Lemma~\ref{LemmaRC3Proj} and Proposition~\ref{PropRCComp} below.} Denote by $\Delta_\tface$ the b-scattering-diagonal, i.e.\ the lift of $\Delta_\bop'$ to $\tface$; denote furthermore the lift of $\ff'$ by $\ff_{\tface,\bop}$ (b-front face), and the lift of $\Delta_\bop'\cap\ff'_\infty$ by $\ff_{\tface,\scop}$ (sc-front face), see the right panel of Figure~\ref{FigRCtfM1}. We remark that the transition from b- to scattering behavior on $\tface$ is similar to the analysis at the transition face in \cite{GuillarmouHassellResI} (denoted $\bface_0$ there).

Let $\rho_{\tface,\scop}\in\CI(\tface)$ denote a defining function of $\ff_{\tface,\scop}$. Schwartz kernels of operators in $\Psi_{\bop,\scop}^m(\ol{{}^+N\pa X})$ are characterized as those distributions which lift to be distributional sections of $\rho_{\tface,\scop}^{-n}\cdot\Omegab^\mhalf\bigl(\ol{({}^+N\pa X)^2_\bop}\bigr)$, conormal of order $m$ to $\Delta_\tface$, which vanish to infinite order at all boundary hypersurfaces other than $\ff_{\tface,\bop}$ and $\ff_{\tface,\scop}$. Schwartz kernels of elements of the \emph{large} b-scattering calculus $\Psi_{\bop,\scop}^{m,(\cE_\lb,\cE_\ff,\cE_\rb)}(\ol{{}^+N\pa X})$ are sums of such distributions and elements of
\[
  \Psi_{\bop,\scop}^{-\infty,(\cE_\lb,\cE_\ff,\cE_\rb)}(\ol{{}^+N\pa X}) = \cA_\phg^{(\cE_\lb,\cE_\ff,\cE_\rb)}\left((\ol{{}^+N\pa X})^2_{\bop,\scop}\right),
\]
where the index sets refer to the lifts of the boundary hypersurfaces $\lb'$, $\ff'$, $\rb'$ of $\ol{({}^+N\pa X)^2_\bop}$, while the index set at each of the other boundary hypersurfaces ($\ff_{\tface,\scop}$ and the lift of $\ff'_\infty$) is the trivial index set $\emptyset$.

Using that there is a natural restriction map $\rho_\dface^{-n}\pi_\bop^*(\Omegab^\mhalf(X^2_\bop))|_{\tface'}\cong\rho_{\tface,\scop}^{-n}\Omegab^\mhalf(\tface')$, we thus get a normal operator map
\begin{equation}
\label{EqRCNtf}
  0 \to \rho_\tface\Psi_{\cop\semi}^m(X) \to \Psi_{\cop\semi}^m(X) \xra{N_\tface} \Psi_{\bop,\scop}^m(\ol{{}^+N\pa X}) \to 0.
\end{equation}

\begin{rmk}
  The structural reason for the appearance of the b-scattering calculus in~\eqref{EqRCNtf} is the fact that the vector fields in~\eqref{EqIVf} restrict to b-scattering vector fields at the front face of the space $X_{\cop\semi}$ defined in~\eqref{EqRCSingle}; indeed, in terms of $\hat x=x/h$, they are equal to $(1+\hat x)^{-1}\hat x D_{\hat x}$ and $(1+\hat x)^{-1}D_{y^j}$, which are indeed b-vector fields near $\hat x=0$ and scattering vector fields near $\hat x=\infty$.
\end{rmk}

\subsection{Composition}

To describe compositions, we will define a suitable triple space. First, recall the definition of the \emph{b-triple space},
\begin{equation}
\label{EqRCX3b}
  X^3_\bop := [X^3;(\pa X)^3;X\times(\pa X)^2;\pa X\times X\times\pa X;(\pa X)^2\times X].
\end{equation}
The lift of $(\pa X)^3$ is denoted $\fff_\bop$. Denote the stretched projections by
\[
  \pi_{\bop,F},\ \pi_{\bop,S},\ \pi_{\bop,C} \colon X^3_\bop\to X^2_\bop,
\]
defined by continuous extension from $(X^\circ)^3\to(X^\circ)^2$, projecting onto the first two factors, last two factors, and first and third factor, respectively; they are b-fibrations. We denote the preimages of $\Delta_\bop$ under these maps by $\Delta_{\bop,\bullet}$, $\bullet=F,S,C$; their intersection is the triple diagonal $\Delta_{\bop,3}$. Note furthermore that $\pi_{\bop,\bullet}^{-1}(\ff_\bop)$ is the union of two boundary p-submanifolds,
\[
  \pi_{\bop,\bullet}^{-1}(\ff_\bop)=\fff_\bop\cup\ff_{\bop,\bullet}, \quad \bullet=F,S,C,
\]
where $\ff_{\bop,\bullet}$ is the lift of the appropriate one of the three submanifolds blown up in~\eqref{EqRCX3b}. We denote the lift of $X^2\times\pa X$, $X\times\pa X\times X$, $\pa X\times X^2$ by $\bface_{\bop,F}$, $\bface_{\bop,C}$, $\bface_{\bop,S}$, respectively. See Figure~\ref{FigRCb3}.

\begin{figure}[!ht]
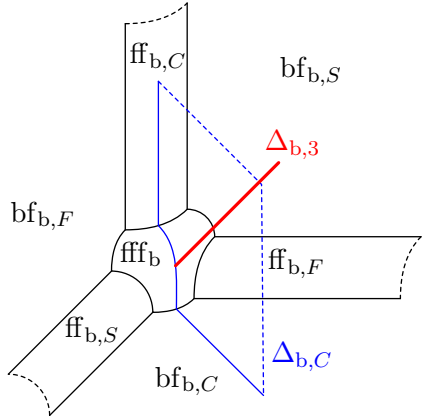

\inclfig{FigRCb3}
\caption{The b-triple space $X^3_\bop$, see~\eqref{EqRCX3b}, and its boundary hypersurfaces, including one of the three lifted diagonals $\Delta_{\bop,\bullet}$, and the triple diagonal $\Delta_{\bop,3}$.}
\label{FigRCb3}
\end{figure}

\begin{definition}
\label{DefRCExt}
  Define the collections of submanifolds of $[0,1)_h\times X^3_\bop$
  \[
    \cF_\bop=\{\{0\}\times\ff_{\bop,\bullet}\},\quad
    \cB_\bop=\{\{0\}\times\bface_{\bop,\bullet}\},\quad
    \cD_\bop=\{\{0\}\times\Delta_{\bop,\bullet}\},\qquad
    \bullet=F,S,C.
  \]
  Recall $\lb_\bop$, $\rb_\bop\subset X^2_\bop$, defined after equation~\eqref{EqRCX2b}. The \emph{extended semiclassical cone double} and \emph{triple spaces} are then
  \begin{align}
  \label{EqRC2Ext}
    {}'X^2_{\cop\semi} &:= \left[X_{\cop\semi}^2; \{0\}\times\lb_\bop; \{0\}\times\rb_\bop\right], \\
  \label{EqRC3Ext}
    {}'X^3_{\cop\semi} &:= \left[ [0,1)_h\times X^3_\bop; \{0\}\times\fff_\bop; \cF_\bop; \cB_\bop; \{0\}\times\Delta_{\bop,3}; \cD_\bop \right],
  \end{align}
  The lifts of $\{0\}\times\lb_\bop$, $\{0\}\times\rb_\bop$ to ${}'X^2_{\cop\semi}$ are denoted $\tlb,\trb$. We moreover label the boundary hypersurfaces of ${}'X^3_{\cop\semi}$ as follows:
  \begin{itemize}
  \item $\fff$, $\ff_\bullet$, $\bface_\bullet$ are the lifts of $[0,1)_h\times\fff_\bop$, $[0,1)_h\times\ff_{\bop,\bullet}$, $[0,1)_h\times\bface_{\bop,\bullet}$;
  \item $\tfface$, $\tface_\bullet$ are the lifts of $\{0\}\times\fff_\bop$, $\{0\}\times\ff_{\bop,\bullet}$;
  \item $\tbface_\bullet$ is the lift of $\{0\}\times\bface_{\bop,\bullet}$;
  \item $\dfface$, $\dface_\bullet$ are the lifts of $\{0\}\times\Delta_{\bop,3}$, $\{0\}\times\Delta_{\bop,\bullet}$;
  \item $\sfface$ is the lift of $\{0\}\times X^3_\bop$.
  \end{itemize}
  Lastly, $\Delta_{\cop\semi,3}$, $\Delta_{\cop\semi,\bullet}$ are the lifts of $[0,1)_h\times\Delta_{\bop,3}$, $[0,1)_h\times\Delta_{\bop,\bullet}$.
\end{definition}

The idea behind the definition~\eqref{EqRC3Ext} of ${}'X^3_{\cop\semi}$ is that we need to blow up each of the three preimages under $\Id_h\times\pi_{\bop,\bullet}$ of those submanifolds of $[0,1)_h\times X^2_\bop$ which were blown up in the Definition~\ref{DefRCDouble} of $X^2_{\cop\semi}$; whenever the preimages of the same submanifold under two different maps intersect non-trivially, we first blow up the intersection to make the lifts of the submanifold disjoint, whence they can be blown up in either order.

The reason for passing to the extended double space is that we need to blow up $\{0\}\times\ff_{\bop,F}$ in the definition of the triple space in order for $\Id_h\times\pi_{\bop,F}$ to lift to a b-fibration down to $X^2_{\cop\semi}$; the image of the lift of $\{0\}\times\ff_{\bop,F}$ under (the lift of) $\Id_h\times\pi_{\bop,C}$ however is then the codimension $2$ submanifold $\{0\}\times\lb_\bop$. To make the latter projection lift to a b-fibration, we need to blow-up $\{0\}\times\lb_\bop$ in the target. A symmetric argument motivates the blow-up of $\{0\}\times\rb_\bop$ in~\eqref{EqRC2Ext}. The preimages of $\{0\}\times\lb_\bop$, $\{0\}\times\rb_\bop$ under $[0,1)_h\times\pi_{\bop,\bullet}$ are the (lifts of) the elements of $\cB_\bop$, hence we also blow these up in~\eqref{EqRC3Ext}.

The following statement shows that Definition~\ref{DefRCExt} is the `right' one:

\begin{lemma}
\label{LemmaRC3Proj}
  The projections $\Id_h\times\pi_{\bop,\bullet}\colon [0,1)_h\times X^3_\bop\to[0,1)_h\times X^2_\bop$ lift to b-fibrations
  \begin{equation}
  \label{EqRC3Proj}
    \pi_{\cop\semi,\bullet} \colon {}'X^3_{\cop\semi} \to {}'X^2_{\cop\semi},\quad
    \bullet=F,S,C.
  \end{equation}
\end{lemma}
\begin{proof}
  For definiteness, consider the case $\bullet=F$. Note that $\Id_h\times\pi_{\bop,F}\colon[0,1)_h\times X^3_\bop\to [0,1)_h\times X^2_\bop$ is a b-fibration. Blowing up $\{0\}\times\ff_\bop$ in the target, and its preimages $\{0\}\times\fff_\bop$ and $\{0\}\times\ff_{\bop,F}$ in the domain, this map lifts to a b-fibration
  \[
    \left[[0,1)_h\times X_\bop^3;\{0\}\times\fff_\bop;\{0\}\times\ff_{\bop,F}\right] \to \left[[0,1)_h\times X^2_\bop;\{0\}\times\ff_\bop\right].
  \]
  by \cite[Proposition~5.12.1]{MelroseDiffOnMwc}. The lift of this map upon blowing up $\{0\}\times\lb_\bop$ in the target and, in this order, $\{0\}\times\ff_{\bop,C}$, $\{0\}\times\bface_{\bop,S}$ in the domain is again a b-fibration; likewise, the lift to the subsequent blow-up of $\{0\}\times\rb_\bop$ in the target and $\{0\}\times\ff_{\bop,S}$, $\{0\}\times\bface_{\bop,C}$ in the domain is a b-fibration. Since the lift of $\{0\}\times\bface_{\bop,F}$ then gets mapped by the stretched projection diffeomorphically to the lift of $\{0\}\times X^2_\bop$, the lift of the stretched projection to the blow-up of $\{0\}\times\bface_{\bop,F}$ is still a b-fibration
  \[
    \left[ [0,1)_h\times X^3_\bop;\{0\}\times\fff_\bop;\cF_\bop;\cB_\bop \right] \to \left[ [0,1)_h\times X^2_\bop;\{0\}\times\ff_\bop;\{0\}\times\lb_\bop;\{0\}\times\rb_\bop \right]
  \]
  by \cite[Corollary~5.10.1]{MelroseDiffOnMwc}.
  
  Blowing up the lift of $\{0\}\times\Delta_\bop$ in the target and its preimage, the lift of $\{0\}\times\Delta_{\bop,F}$, in the domain produces a b-fibration
  \[
    \left[[0,1)_h\times X_\bop^3;\{0\}\times\fff_\bop;\cF_\bop;\cB_\bop;\{0\}\times\Delta_{\bop,F}\right] \to {}'X^2_{\cop\semi}.
  \]
  This is b-transversal to the lift of $\{0\}\times\Delta_{\bop,3}$ (since the latter gets mapped diffeomorphically to the codimension $1$ boundary face $\dface$), hence blowing this up, the projection map lifts again to a b-fibration. Since $\{0\}\times\Delta_{\bop,3}\subset\{0\}\times\Delta_{\bop,F}$, we can in fact blow the lifts of these two manifolds up in any order by \cite[Proposition~5.11.2]{MelroseDiffOnMwc}; thus the stretched projection
  \begin{equation}
  \label{EqRC3Proj0}
    \left[[0,1)_h\times X_\bop^3;\{0\}\times\fff_\bop;\cF_\bop;\cB_\bop;\{0\}\times\Delta_{\bop,3};\{0\}\times\Delta_{\bop,F}\right] \to X^2_{\cop\semi}
  \end{equation}
  is a b-fibration. This is b-transversal to the lifts of $\{0\}\times\Delta_{\bop,\bullet}$, $\bullet=S,C$, both of which get mapped diffeomorphically to the lift of $\{0\}\times\bface_{\bop,F}$; blowing these up, the map~\eqref{EqRC3Proj0} lifts to the map $\pi_{\cop\semi,F}$ in~\eqref{EqRC3Proj}, which is thus a b-fibration, as claimed.
\end{proof}

Composition properties of large parameter operators were proved by explicit calculations in \cite[\S\S4--5]{LoyaConicResolvent}. Compositions in our semiclassical calculus are consequences of standard pullback and pushforward theorems, see \cite[\S6]{MelroseDiffOnMwc}, \cite[Appendix~B]{EpsteinMelroseMendozaPseudoconvex}.

\begin{prop}
\label{PropRCComp}
  Let $A_j\in\Psi_{\cop\semi}^{m_j}(X)$ and $A'_j\in\Psi_{\cop\semi}^{-\infty,\cE_j}(X)$ for $j=1,2$. Then:
  \begin{enumerate}
  \item $A_1\circ A_2\in\Psi_{\cop\semi}^{m_1+m_2}(X)$ and $A_1\circ A'_2\in\Psi_{\cop\semi}^{-\infty,\cE_2}(X)$, $A_1'\circ A_2\in\Psi_{\cop\semi}^{-\infty,\cE_1}(X)$.
  \item Write $\cE_j=(\cE_{j,\lb},\cE_{j,\ff},\cE_{j,\rb},\cE_{j,\tface})$, and suppose that $\inf\Re(\cE_{1,\rb}+\cE_{2,\lb})>0$. Then $A'_1\circ A'_2\in\Psi_{\cop\semi}^{-\infty,\cF}(X)$, where $\cF=(\cF_\lb,\cF_\ff,\cF_\rb,\cF_\tface)=\cE_1\circ\cE_2$ is defined by
    \begin{equation}
    \label{EqRCCompIndex}
    \begin{aligned}
      \cF_\lb ={}& \cE_{1,\lb}\extcup{}(\cE_{1,\ff}+\cE_{2,\lb}), & \cF_\rb ={}& (\cE_{1,\rb}+\cE_{2,\ff})\extcup \cE_{2,\rb}, \\
      \cF_\ff ={}& (\cE_{1,\ff}+\cE_{2,\ff})\extcup{}(\cE_{1,\lb}+\cE_{2,\rb}), & \quad \cF_\tface ={}& \cE_{1,\tface}+\cE_{2,\tface}.
    \end{aligned}
    \end{equation}
    Here, $\extcup$ denotes the extended union of index sets, defined in~\eqref{EqBExtCup}.
  \end{enumerate}
\end{prop}
\begin{proof}
  Denote the Schwartz kernel of $A_j$ by $K_j$. Note that $K_j$ lifts to ${}'X^2_{\cop\semi}$, with index set $\emptyset$ at $\tlb,\trb$. The Schwartz kernel $K$ of $A_1 A_2$ equals
  \[
    K = (\pi_{\cop\semi,C})_*\bigl(\pi_{\cop\semi,F}^*K_1\cdot\pi_{\cop\semi,S}^*K_2\bigr).
  \]
  Now $\pi_{c\semi,F}^*K_1$ is conormal to $\Delta_{\cop\semi,\bullet}$ and vanishes to infinite order at the preimage of $\lb\cup\rb\cup\sface$ under $\pi_{\cop\semi,F}$, thus at all boundary hypersurfaces of $X^3_{\cop\semi}$ \emph{except for} $\fff$, $\tfface$, $\dfface$, $\ff_F$, $\tface_F$, $\dface_F$, $\bface_F$. The product with $\pi_{c,\semi,S}^*K_2$ is well-defined since $\Delta_{\cop\semi,F}$ and $\Delta_{\cop\semi,S}$ intersect, transversally, at $\Delta_{\cop\semi,3}$, and vanishes to infinite order at all boundary hypersurfaces except for $\fff,\tfface,\dfface$. Under pushforward by $\pi_{\cop\semi,C}$---which is transversal to these and embeds them as the submanifolds $\ff$, $\tface$, $\dface$ of ${}'X^2_{\cop\semi}$---we obtain, by \cite[Proposition~B7.20]{EpsteinMelroseMendozaPseudoconvex}, a conormal distribution on ${}'X^2_{\cop\semi}$ vanishing to infinite order at all boundary hypersurfaces except for $\ff$, $\tface$, $\dface$; in particular, we can blow down $\tlb$, $\trb$ and thus obtain an element of $\Psi_{\cop\semi}^{m_1+m_2}(X)$.

  The rest of the proposition follows by similar arguments; we give details for $A_1'\circ A_2'$. Denote the Schwartz kernel of $A_j'$ by $K_j'$. Since $K_1'$ lifts to be polyhomogeneous on ${}'X^2_{\cop\semi}$ with index set $\emptyset$ at $\tlb$, $\trb$, one finds that $\pi_{c,\semi,F}^*K_1'$ is polyhomogeneous on ${}'X^3_{\cop\semi}$, with index set $E_{1,\lb}$ at $\bface_S,\ff_C$; $E_{1,\ff}$ at $\fff$, $\ff_F$; $E_{1,\rb}$ at $\bface_C,\ff_S$; $E_{1,\tface}$ at $\tfface,\tface_F$; $\N_0$ at $\bface_F,\tbface_F$; and $\emptyset$ at the remaining boundary hypersurfaces. Similarly, $\pi_{\cop\semi,S}^*K_2'$ is polyhomogeneous with index set $E_{2,\lb}$ at $\bface_C,\ff_F$; $E_{2,\ff}$ at $\fff,\ff_S$; $E_{2,\rb}$ at $\bface_F,\ff_C$; $E_{2,\tface}$ at $\tfface,\tface_S$; $\N_0$ at $\bface_S,\tbface_S$; and $\emptyset$ otherwise. The preimage under $\pi_{\cop\semi,C}$ of $\lb$ is $\bface_S\cup\ff_F$, of $\ff$ is $\fff\cup\ff_C$, of $\rb$ is $\bface_F,\ff_S$, of $\tface$ is $\tfface\cup\tface_C$, and the preimages of $\tlb,\trb,\sface,\dface$ are boundary hypersurfaces of ${}'X^3_{\cop\semi}$ at which $\pi_{\cop\semi,F}^*K_1\cdot\pi_{\cop\semi,S}K_2$ has trivial index sets. Moreover, $(\pi_{\cop\semi,C})_*$ integrates transversally to $\bface_C$, at which the index set of $\pi_{\cop\semi,F}^*K_1'\cdot\pi_{\cop\semi,S}^*K_2'$ is $E_{1,\rb}+E_{2,\lb}$. The pushforward theorem, see \cite[Proposition~B5.6]{EpsteinMelroseMendozaPseudoconvex}, then gives the result. (For example, the index sets at $\bface_S$ and $\ff_F$ are $E_{1,\lb}$ and $E_{1,\ff}+E_{2,\lb}$; the map $\pi_{\cop\semi,C}$ maps these to $\lb$, where the pushforward of the polyhomogeneous distribution then has index set $E_{1,\lb}\extcup(E_{1,\ff}+E_{2,\lb})=\cF_\lb$.)
\end{proof}

The principal symbol map $\sigmach$ is multiplicative, $\sigmach_{m_1+m_2}(A_1 A_2)=\sigmach_{m_1}(A_1)\sigmach_{m_2}(A_2)$; this follows by continuity from $X^\circ$, $h\geq 0$, resp.\ $X$, $h>0$ from the corresponding statement for the semiclassical, resp.\ b-pseudodifferential calculus.

The normal operator maps $N_\ff$ and $N_\tface$ are homomorphisms, so $N_\ff(A_1\circ A_2)=N_\ff(A_1)\circ N_\ff(A_2)$ etc. We only need this in the case that at least one of $A_1,A_2$ is a \emph{differential} operator, in which case this is easily verified in local coordinates; the calculations are similar to \cite[\S4.15]{MelroseAPS}.

\subsection{The inverse of a fully elliptic semiclassical cone operator}
\label{SsRR}

Fix an operator
\[
  A \in \Diff_{\cop,\semi}^m(X)
\]
which is fully elliptic at weight $\alpha\in\R$. Denote the Mellin transformed normal operator by $\hat A(\sigma):=\wh{N_{\pa X}(A)}(\sigma)$, $\sigma\in\C$. We then define its boundary spectrum by
\begin{equation}
\label{EqRRSpecb}
  \specbfull(A) := \left\{ (\sigma,k) \in \C\times\N_0 \colon \hat A(\zeta)^{-1}\ \text{has a pole at}\ \zeta=\sigma\ \text{of order}\geq k+1 \right\}.
\end{equation}
Define the following index sets:
\begin{equation}
\label{EqRRElbrb}
\begin{split}
  E_\lb(\alpha) &:= \left\{ (z,k) \in \C\times\N_0 \colon (-i z,k)\in\specbfull(A),\ \Re z>\alpha \right\}, \\
  E_\rb(\alpha) &:= \left\{ (z,k) \in \C\times\N_0 \colon (i z,k)\in\specbfull(A),\ \Re z>-\alpha \right\}.
\end{split}
\end{equation}
We then set
\[
  \hat E_\lb(\alpha) := \bigextcup_{j\in\N_0} (E_\lb(\alpha)+j),\quad
  \hat E_\rb(\alpha) := \bigextcup_{j\in\N_0} (E_\rb(\alpha)+j),
\]
and finally
\begin{equation}
\label{EqRREff}
\begin{split}
  \check E_\bullet(\alpha) &:= \hat E_\bullet(\alpha)\extcup\hat E_\bullet(\alpha), \quad \bullet=\lb,\rb, \\
  \check E_\ff(\alpha) &:= \N_0 + \left((\check E_\lb(\alpha)+\check E_\rb(\alpha)) \extcup \N\right).
\end{split}
\end{equation}

\begin{thm}
\label{ThmR}
  Let $A\in\Diff_{\cop,\semi}^m(X)$ be fully elliptic at weight $\alpha$. Then there exists $h_0>0$ such that
  \[
    A_h\colon\Hb^{s,\alpha}(X)\to\Hb^{s-m,\alpha-m}(X),\quad s\in\R,
  \]
  is invertible for $0<h<h_0$. The inverse lies in the large semiclassical cone calculus,
  \begin{align}
    &A^{-1} \in \bigl(\tfrac{x}{x+h}\bigr)^m\Psi_{\cop\semi}^{-m}(X) + \Psi_{\cop\semi}^{-\infty,\cE}(X), \nonumber\\
  \label{EqRRcE}
    &\qquad \cE=(\cE_\lb,\cE_\ff,\cE_\rb,\cE_\tface)=(\check E_\lb(\alpha),\check E_\ff(\alpha)+m,\check E_\rb(\alpha)+m,\N_0).
  \end{align}
\end{thm}

\begin{rmk}
\label{RmkRGeneral}
  The conclusion holds more generally, with exactly the same proof, when $A\in\rho_\ff^{-m}\Psi_{\cop\semi}^m(X)$ is a semiclassical cusp ps.d.o.\ which is fully elliptic at weight $\alpha$ in the sense that $\sigmach_m(\rho_\ff^m A)$ is elliptic, $f N_\ff(\rho_\ff^m A)$ is $h$-independent for some $0<f\in\CI([0,1)_h\times\pa X)$, and $N_\tface(A)\colon H_{\bop,\scop}^{s,\alpha}(\ol{{}^+N\pa X})\to H_{\bop,\scop}^{s-m,\alpha-m}(\ol{{}^+N\pa X})$ is invertible for some $s\in\R$ (thus all $s\in\R$ by ellipticity).
\end{rmk}

The part of the proof which employs elliptic b-theory is very similar to that of \cite[Theorem~6.1]{LoyaConicResolvent}. The efficient analysis of the normal operator at $\tface$ is new.

\begin{proof}[Proof of Theorem~\usref{ThmR}]
  The symbolic elliptic parametrix construction, using~\eqref{EqRCSy}, produces
  \begin{equation}
  \label{EqRRB0}
    B_0\in\rho_\ff^m\Psi_{\cop\semi}^{-m}(X),\quad A B_0=I-R_0,\ R_0\in\rho_\dface^\infty\Psi_{\cop\semi}^{-\infty}(X)=\Psi_{\cop\semi}^{-\infty,(\emptyset,\N_0,\emptyset,\N_0)}(X).
  \end{equation}
  
  We next solve away the error to leading order at $\ff$ by passing to $\ff$-normal operators, see~\eqref{EqRCNff}. Write $((\tfrac{x}{x+h})^m A)(B_0 (\tfrac{x}{x+h})^{-m})=I-R_0'$, $R_0'=(\tfrac{x}{x+h})^m R_0(\tfrac{x}{x+h})^{-m}\in\rho_\dface^\infty\Psi_{\cop\semi}^{-\infty}(X)$, the point being that near $\ff$, $(\tfrac{x}{x+h})^m A\in\Psi_{\cop\semi}^m(X)$ and $B_0 (\tfrac{x}{x+h})^{-m}$ are unweighted elliptic operators. We then seek $B_1$ with $N_\ff((\tfrac{x}{x+h})^m A)N_\ff(B_1(\tfrac{x}{x+h})^{-m})=N_\ff(R_0')$, which can be accomplished using the Mellin transform along the fibers of $\ff\cap h^{-1}(h_0)$, $h_0\geq 0$ (and working on the line $\Im\sigma=-\alpha$ on the Mellin transformed side): this is standard elliptic b-theory \cite[\S5.13]{MelroseAPS}, with smooth parametric dependence on $h\in[0,1)$; recall here the assumption that the normal operator of $A$ is $h$-independent. The upshot is that, localizing near $\ff$ by means of a cutoff which vanishes identically near $\sface\cup\dface$, we can find (using $(\tfrac{x}{x+h})^m=f''\rho_\lb^m\rho_\ff^m$ and $(\tfrac{x}{x+h})^{-m}(\tfrac{x'}{x'+h})^m=f'''\rho_\lb^{-m}\rho_\rb^m$ for some smooth $f'',f'''>0$)
  \begin{align*}
    &B_1\in \Psi_{\cop\semi}^{-\infty,(E_\lb(\alpha),\N_0+m,E_\rb(\alpha)+m,\N_0)}(X), \\
    &\qquad A(B_0+B_1)=I-R_1,\quad R_1\in\Psi_{\cop\semi}^{-\infty,(E_\lb(\alpha)-m,\N,E_\rb(\alpha)+m,\N_0)}(X).
  \end{align*}
  Note that the index set of $R_1$ at $\ff$ is $\N=\{1,2,\ldots\}$. This vanishing at $\ff$ comes at the expense of polyhomogeneous error terms at $\lb$. Now, letting $A$ act on the left factor of the double space (fiberwise in $h$), its normal operator at $\lb$ is identified with $N_\ff(A)$. Thus, we can solve away the error term at $\lb$ to infinite order as in~\cite[\S5.20]{MelroseAPS} with an operator
  \begin{equation}
  \label{EqRRB2}
  \begin{split}
    &B_2 \in \Psi_{\cop\semi}^{-\infty,(E_\lb(\alpha)\extcup\hat E_\lb(\alpha),\N+m,\emptyset,\N_0)}(X), \\
    &\qquad A(B_0+B_1+B_2)=I-R_2,\quad R_2\in\Psi_{\cop\semi}^{-\infty,(\emptyset,\N,E_\rb(\alpha)+m,\N_0)}(X).
  \end{split}
  \end{equation}
  We can improve this using a Neumann series: we have $R_2^j\in\Psi_{\cop\semi}^{-\infty,(\emptyset,j+\N_0,E_\rb^j(\alpha)+m,\N_0)}(X)$ for $E_\rb^1(\alpha)=E_\rb(\alpha)$ and $E_\rb^{j+1}(\alpha)=(j+E_\rb(\alpha))\extcup E_\rb^j(\alpha)\nearrow\hat E_\rb(\alpha)$. Thus, let
  \begin{equation}
  \label{EqRRB3}
    B'_3\sim\sum_{j\geq 0}R_2^j,\quad
    B'_3-I \in \Psi_{\cop\semi}^{-\infty,(\emptyset,\N,\hat E_\rb(\alpha)+m,\N_0)}(X).
  \end{equation}
  We then obtain the analogue of \cite[Lemma~6.7]{LoyaConicResolvent} using Proposition~\ref{PropRCComp} (and slightly enlarging index sets for brevity), namely
  \begin{equation}
  \label{EqRRG3}
  \begin{split}
    &G := (B_0+B_1+B_2)B'_3 \in \rho_\ff^m\Psi_{\cop\semi}^{-m}(X) + \Psi_{\cop\semi}^{-\infty,\cE}(X) \\
    &\qquad \Rightarrow A G=I-R,\quad R\in\Psi_{\cop\semi}^{-\infty,(\emptyset,\emptyset,\check E_\rb(\alpha)+m,\N_0)}(X).
  \end{split}
  \end{equation}
  
  We need to improve the decay of the error term at $\tface$; this requires inverting the model operator $N_\tface(A)$, which is an elliptic b-scattering operator of class $\rho_{\tface,\bop}^{-m}\Diff_{\bop,\scop}^m(\ol{{}^+N\pa X})$. But by the full ellipticity of $A$ at weight $\alpha$,
  \begin{equation}
  \label{EqRRNtf}
    N_\tface(A) \colon H_{\bop,\scop}^{s,\alpha}(\ol{{}^+N\pa X}) \xra{\cong} H_{\bop,\scop}^{s-m,\alpha-m}(\ol{{}^+N\pa X}),\quad s\in\R,
  \end{equation}
  is invertible; hence we can construct the Schwartz kernel of its inverse explicitly by combining elliptic b- and scattering theory. To wit, we first construct a parametrix $B_{\tface,0}\in\rho_{\tface,\bop}^m\Psi_{\bop,\scop}^{-m}(\ol{{}^+N\pa X})$ with $N_\tface(A)B_{\tface,0}=I-R_{\tface,0}$, $R_{\tface,0}\in\rho_{\tface,\scop}^\infty\Psi_{\bop,\scop}^{-\infty}(\ol{{}^+N\pa X})$ using the symbol calculus. By elliptic b-theory as above, one then obtains a right parametrix, and by similar means a left parametrix,
  \begin{align}
  \label{EqRRtfSp}
    &B_{\tface,1},B_{\tface,1}' \in \rho_{\tface,\bop}^m\Psi_{\bop,\scop}^{-m}(\ol{{}^+N\pa X}) + \Psi_{\bop,\scop}^{-\infty,(\cE_\lb,\cE_\ff,\cE_\rb)}(\ol{{}^+N\pa X}), \\
    &\qquad N_\tface(A)B_{\tface,1} = I-R_{\tface,1},\quad R_{\tface,1}\in\Psi_{\bop,\scop}^{-\infty,(\emptyset,\emptyset,\check E_\rb(\alpha)+m)}(\ol{{}^+N\pa X}), \nonumber\\
    &\qquad B'_{\tface,1}N_\tface(A) = I-R'_{\tface,1},\quad R'_{\tface,1}\in\Psi_{\bop,\scop}^{-\infty,(\check E_\lb(\alpha),\emptyset,\emptyset)}(\ol{{}^+N\pa X}); \nonumber
  \end{align}
  the index sets arising here are the same as above since the normal operator of $N_\tface(\rho_\ff^m A)$ at $\pa_0\ol{{}^+N\pa X}$ is equal to $N_\ff(\rho_\ff^m A)$ restricted to $h=0$, hence their boundary spectra are the same. We then have
  \[
    N_\tface(A)^{-1} = B_{\tface,1} + N_\tface(A)^{-1}R_{\tface,1} = B_{\tface,1} + B'_{\tface,1}R_{\tface,1} + R'_{\tface,1}N_\tface(A)^{-1}R_{\tface,1},
  \]
  with the first two terms lying in the space~\eqref{EqRRtfSp}, while the last term, in view of~\eqref{EqRRNtf}, has Schwartz kernel in $\cA_\phg^{(\check E_\lb(\alpha),\check E_\rb(\alpha)+m)}((\ol{{}^+N\pa X})^2)$ (with trivial index sets at $\pa_\infty\ol{{}^+N\pa X}\times\ol{{}^+N\pa X}$ and $\ol{{}^+N\pa X}\times\pa_\infty\ol{{}^+N\pa X}$), which is a subspace of $\Psi_{\bop,\scop}^{-\infty,(\cE_\lb,\cE_\ff,\cE_\rb)}(X)$. Therefore,
  \begin{equation}
  \label{EqRRNtfAInv}
    N_\tface(A)^{-1} \in \rho_{\tface,\bop}^m\Psi_{\bop,\scop}^{-m}(\ol{{}^+N\pa X}) + \Psi_{\bop,\scop}^{-\infty,(\cE_\lb,\cE_\ff,\cE_\rb)}(\ol{{}^+N\pa X}).
  \end{equation}

  Equipped with this structure of $N_\tface(A)^{-1}$, we now return to~\eqref{EqRRG3}. Choose an operator $B'_4\in\rho_\ff^m\Psi_{\cop\semi}^{-m}(X)+\Psi_{\cop\semi}^{-\infty,\cE}(X)$ with $N_\tface(B'_4)=N_\tface(A)^{-1}$. For
  \[
    B_4 := B'_4\circ R \in \Psi_{\cop\semi}^{-\infty,\cE}(X),
  \]
  we then have
  \[
    A(G+B_4)=I-R_4,\quad R_4 \in \Psi_{\cop\semi}^{-\infty,(\emptyset,\emptyset,\check E_\rb(\alpha)+m,\N)}(X),
  \]
  which is an improvement at $\tface$. We can now solve away the error by a Neumann series: taking $B_5'\in\Psi_{\cop\semi}^{-\infty,(\emptyset,\emptyset,\check E_\rb(\alpha)+m,\N)}(X)$ with $B_5'\sim\sum_{j\geq 0}R_4^j$, we put
  \[
    G_5=(G+B_4)B_5' \in \rho_\ff^m\Psi_{\cop\semi}^{-m}(X) + \Psi_{\cop\semi}^{-\infty,\cE}(X),
  \]
  which satisfies
  \[
    A G_5=I-R_5,\quad R_5\in\Psi_{\cop\semi}^{-\infty,(\emptyset,\emptyset,\check E_\rb(\alpha)+m,\emptyset)}(X)=h^\infty\cA_\phg^{(\emptyset,\check E_\rb(\alpha)+m)}([0,1)_h\times X^2),
  \]
  with the index sets in the last space referring to the left and right boundary, $[0,1)\times\pa X\times X$ and $[0,1)\times X\times\pa X$, respectively. For sufficiently small $h>0$, $I-R_5$ is an invertible map on $\Hb^{s-m,\alpha-m}(X)$, with inverse given by a convergent Neumann series, and thus lying in $I+h^\infty\cA_\phg^{(\emptyset,\check E_\rb(\alpha)+m)}([0,1)\times X^2)$. Therefore,
  \[
    A B = I,\quad B:=G_5(I-R_5)^{-1} \in \rho_\ff^m\Psi_{\cop\semi}^{-m}(X)+\Psi_{\cop\semi}^{-\infty,\cE}(X),
  \]
  completing the construction of a right inverse. The construction of a left inverse is analogous. A standard argument shows that the right and left inverses agree, finishing the proof of the theorem.
\end{proof}

\subsection{Mapping properties}
\label{SsRM}

Recall from~\eqref{EqRCSingle} the space $X_{\cop\semi}=[[0,1)_h\times X;\{0\}\times\pa X]$, which is a natural manifold for the definition of function spaces compatible with the semiclassical cone algebra:

\begin{definition}
\label{DefRMX1}
  The boundary hypersurfaces of $X_{\cop\semi}$ are denoted $\pa_\cop X$ (`cone face', lift of $[0,1)_h\times\pa X$), $\pa_{\rm t}X$ (`transition face', lift of $\{0\}\times\pa X$), and $\pa_\semi X$ (`semiclassical face', lift of $\{0\}\times X$). See Figure~\ref{FigRMX1}. Given index sets $\cF_\cop,\cF_{\rm t},\cF_\semi\subset\C\times\N_0$, we define $\cA_\phg^{(\cF_\cop,\cF_{\rm t},\cF_\semi)}(X_{\cop\semi})$ to consist of polyhomogeneous distributions with specified index sets at the respective boundary hypersurfaces.
\end{definition}

\begin{figure}[!ht]
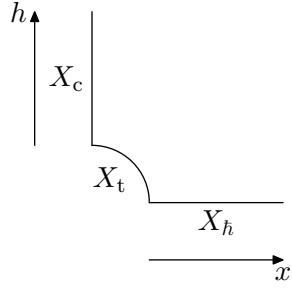

\inclfig{FigRMX1}
\caption{The semiclassical cone single space $X_{\cop\semi}$. Also indicated are axes for the semiclassical parameter $h$ and a defining function $x$ of $\pa X\subset X$; the single space arises by blowing up the corner $h=x=0$ in the product space $[0,1)_h\times X$.}
\label{FigRMX1}
\end{figure}

For notational simplicity, let us fix a trivialization of ${}^\bop\Omega^\mhalf X$.

\begin{lemma}
\label{LemmaRML2}
  Let $A\in\Psi_{\cop\semi}^0(X)$. Then $A$ is uniformly bounded on $L^2_\bop(X)$.
\end{lemma}
\begin{proof}
  By H\"ormander's square root trick, it suffices to prove this for operators which are residual in the symbolic sense, i.e.\ $A\in\rho_\dface^\infty\Psi_{\cop\semi}^{-\infty}(X)=\Psi_{\cop\semi}^{-\infty,(\emptyset,\N_0,\emptyset,\N_0)}(X)$. Writing the Schwartz kernel of such $A$ as $K(h,z,z')$, $z,z'\in X$, it then suffices, by Schur's lemma, to prove the uniform boundedness of $\int_X K(h,z,z')\,\dd\mu_\bop(z)$ and $\int_X K(h,z,z')\,\dd\mu_\bop(z')$, where $\mu_\bop$ is a b-density on $X$. Note then that the projection $(h,z,z')\mapsto(h,z)$ lifts to a b-fibration ${}'X^2_{\cop\semi}\to X_{\cop\semi}$ under which $K$ pushes forward to an element of $\cA_\phg^{(\N_0,\N_0,\emptyset)}(X_{\cop\semi})$. This is thus a bounded function; similarly for $\int_X K(h,z,z')\,\dd\mu_\bop(z')$. This concludes the proof.
\end{proof}

Note that $\frac{x}{x+h}$, $x+h$, and $\frac{h}{x+h}$ are defining functions of $\pa_\cop X$, $\pa_{\rm t}X$, and $\pa_\semi X$, respectively. Since $h$ is merely a commutative parameter, this suggests the introduction of the following weighted function spaces:

\begin{definition}
\label{DefRMSob}
  Let $s,\alpha,\tau\in\R$. Then we define the weighted semiclassical cone Sobolev space by
  \[
    H_{\cop,h}^{s,\alpha,\tau}(X)=\bigl(\tfrac{x}{x+h}\bigr)^\alpha(x+h)^\tau H_{\cop,h}^s(X) = \{ x^\alpha(x+h)^\tau u \colon u\in H_{\cop,h}^s(X) \},\quad h > 0,
  \]
  where $H_{\cop,h}^s(X)$ is the completion of $\CIc(X^\circ)$ with respect to the norm
  \begin{equation}
  \label{EqRMSobNorm}
    \| u \|_{H_{\cop,h}^s(X)}^2 := \|u\|_{L^2_\bop(X)}^2 + \|A u\|_{L^2_\bop(X)}^2,\quad s\geq 0,
  \end{equation}
  where $A\in\Psi_{\cop,h}^s(X)$ is a fixed elliptic operator (i.e.\ has elliptic principal symbol). For $s<0$, we define $H_{\cop,h}^s(X)$ as the $L^2_\bop(X)$-dual of $H_{\cop,h}^{-s}(X)$.\footnote{Equivalently, $H_{\cop,h}^s(X)$ for $s<0$ consists of all extendible distributions \cite[Appendix~B]{HormanderAnalysisPDE3} $u$ on $X^\circ$ which are of the form $u=u_0+A u_1$ with $u_0,u_1\in L^2_\bop(X)$ and $A\in\Psi_{\cop\semi}^{-s}(X)$ a fixed elliptic operator. As the norm of $u$, one can then take $\inf\|u_0\|+\|u_1\|$, the infimum taken over all such decompositions.}
\end{definition}

For $s\in\N_0$, an equivalent definition is that $u\in H_{\cop,h}^{s,\alpha,\tau}(X)$ if and only if
\[
  \bigl(\tfrac{h}{x+h}V_1\bigr)\cdots\bigl(\tfrac{h}{x+h}V_j\bigr)\bigl(\bigl(\tfrac{x}{x+h}\bigr)^{-\alpha}(x+h)^{-\tau}u\bigr)\in L^2_\bop(X)
\]
for all $j=0,\ldots,k$ and all $V_j\in\Vb(X)$.

\begin{rmk}
\label{RmkRMHbh}
  For $h$ bounded away from $0$, $H_{\cop,h}^s(X)$ is the standard b-Sobolev space $\Hb^s(X)$, while for $x$ bounded away from $0$, this is the standard semiclassical Sobolev space $H_h^s(X^\circ)$. More generally, we have $H_{\cop,h}^{s,\alpha,\tau}(X)=\Hb^{s,\alpha}(X)$ as sets, but with inequivalent norms as $h\to 0$; we describe the relationship more precisely in Proposition~\ref{PropRMHbh} below.
\end{rmk}

We record the boundedness of semiclassical cone ps.d.o.s on weighted Sobolev spaces:

\begin{prop}
\label{PropRMSob}
  Let $s,\alpha,\tau\in\R$, and let $\cE=(\cE_\lb,\cE_\ff,\cE_\rb,\cE_\tface)$ be a collection of index sets. Then:
  \begin{enumerate}
  \item\label{ItRMSob1} Any $A\in\Psi_{\cop\semi}^m(X)$ defines a bounded map $A\colon H_{\cop,h}^{s,\alpha,\tau}(X)\to H_{\cop,h}^{s-m,\alpha,\tau}(X)$.
  \item\label{ItRMSob2} Let $A\in\Psi_{\cop\semi}^{-\infty,\cE}(X)$. If $\alpha+\inf\Re\cE_\rb>0$, then $A$ is a bounded map
    \begin{equation}
    \label{EqRMSob}
      A\colon H_{\cop,h}^{s,\alpha,\tau}(X)\to H_{\cop,h}^{\infty,\beta,\mu}(X)
    \end{equation}
    for $\beta<\min(\inf\Re\cE_\lb,\alpha+\inf\Re\cE_\ff)$, $\mu<\tau+\inf\Re\cE_\tface$.
  \item\label{ItRMSob3} If the leading order terms of $\cE_\ff$ and $\cE_\tface$ do not have logarithmic factors (i.e.\ $(z,k)\in\cE_\ff$, $\Re z=\inf\Re\cE_\ff$ implies $k=0$; similarly for $\cE_\tface$), \eqref{EqRMSob} holds provided that, still, $\alpha+\inf\Re\cE_\rb>0$, $\beta<\inf\Re\cE_\lb$, but only $\beta\leq\alpha+\inf\Re\cE_\ff$, $\mu\leq\tau+\inf\Re\cE_\tface$.
  \end{enumerate}
\end{prop}
\begin{proof}
  For $\alpha=\tau=0$, statement~\eqref{ItRMSob1} is Lemma~\ref{LemmaRML2} for $m=0$, and for general $m$ follows by a standard argument (using a symbolic elliptic parametrix for the operator $A$ in~\eqref{EqRMSobNorm}) from Lemma~\ref{LemmaRML2}. For general weights, we only need to observe that $\Psi_{\cop\semi}^m(X)$ is invariant under conjugation by $x^\alpha(x+h)^{\tau-\alpha}$. For $\tau-\alpha=0$, this follows from $(x'/x)^\alpha=\rho_\lb^{-\alpha}\rho_\rb^\alpha$ and the infinite order vanishing of Schwartz kernels of elements of $\Psi_{\cop\semi}(X)$ at $\lb\cup\rb$. For $\alpha=0$ on the other hand, we have
  \[
    \frac{x'+h}{x+h} = \frac{\rho_\rb\rho_\ff\rho_\tface+\rho_\tface\rho_\sface\rho_\dface}{\rho_\lb\rho_\ff\rho_\tface+\rho_\tface\rho_\sface\rho_\dface} = \frac{\rho_\rb\rho_\ff+\rho_\sface\rho_\dface}{\rho_\lb\rho_\ff+\rho_\sface\rho_\dface},
  \]
  which is a smooth function on $X^2_{\cop\semi}$ except at $\sface\cap\lb$, where it is equivalent to $1/(\rho_\lb+\rho_\sface)$ and thus merely conormal. The infinite order vanishing of Schwartz kernels at $\lb$ means that multiplication by $(\frac{x'+h}{x+h})^{\tau-\alpha}$, $\tau-\alpha>0$, preserves $\Psi_{\cop\semi}(X)$. The argument for $\tau-\alpha<0$ is analogous.

  To prove~\eqref{ItRMSob2}, we define
  \[
    A' = (x+h)^{-\mu}\bigl(\tfrac{x}{x+h}\bigr)^{-\beta} A \bigl(\tfrac{x'}{x'+h}\bigr)^\alpha(x'+h)^\tau \in \Psi_{\cop\semi}^{-\infty,\cE'}(X),
  \]
  where $\cE'=(\cE'_\lb,\cE'_\ff,\cE'_\rb,\cE'_\tface)$ with
  \[
    \cE'_\lb = \cE_\lb-\beta,\quad
    \cE'_\ff = \cE_\ff+\alpha-\beta,\quad
    \cE'_\rb = \cE_\rb+\alpha,\quad
    \cE'_\tface = \cE_\tface + (\tau-\mu).
  \]
  Since $A$ has differential order $-\infty$, it suffices to prove that $A'$ is uniformly bounded on $L^2_\bop(X)$. Note then that $\inf\Re\cE'_\bullet>0$ for $\bullet=\lb,\ff,\rb,\tface$, which implies the desired boundedness by Schur's lemma.
  
  Part~\eqref{ItRMSob3} follows from these considerations as well, since now polyhomogeneous distributions with index sets $\cE'_\ff$ and $\cE'_\tface$ at $\ff$ and $\tface$ are \emph{bounded} there, hence Schur's lemma applies again.
\end{proof}

We further record mapping properties on polyhomogeneous distributions:

\begin{prop}
\label{PropRMPhg}
  Let $\cE=(\cE_\lb,\cE_\ff,\cE_\rb,\cE_\tface)$ and $\cF=(\cF_\cop,\cF_{\rm t},\cF_\semi)$ denote two collections of index sets. Then:
  \begin{enumerate}
  \item Any $A\in\Psi_{\cop\semi}^m(X)$ is a bounded map on $\cA_\phg^\cF(X_{\cop\semi})$.
  \item Let $A\in\Psi_{\cop\semi}^{-\infty,\cE}(X)$, and suppose $\inf\Re(\cF_\cop+\cE_\rb)>0$. Then $A$ is a bounded map $A \colon \cA_\phg^\cF(X_{\cop\semi}) \to \cA_\phg^\cG(X_{\cop\semi})$, where
  \[
    \cG_\cop = (\cF_\cop+\cE_\ff) \extcup \cE_\lb, \quad
    \cG_{\rm t} = \cF_t+\cE_\tface, \quad
    \cG_\semi = \emptyset.
  \]
  \end{enumerate}
\end{prop}
\begin{proof}
  This follows easily from pullback and pushforward theorems by writing $A u=\pi_*(K_A\cdot(\pi')^*u)$, where $K_A$ is the Schwartz kernel of $A$, and $\pi$, resp.\ $\pi'\colon{}'X^2_{\cop\semi}\to X_{\cop\semi}$ are the lifts of the projection maps $[0,1)_h\times X^2\to[0,1)_h\times X$ to the left, resp.\ right factor of $X^2$.
\end{proof}

Finally, we relate the spaces $H_{\cop,h}(X)$ to the more standard semiclassical b-Sobolev spaces $\Hbh^{s,\alpha}(X)=x^\alpha\Hbh^s(X)$, see Appendix~\ref{SB}.

\begin{prop}
\label{PropRMHbh}
  For $s,\alpha,\tau\in\R$, we have
  \begin{alignat*}{3}
    H_{\cop,h}^{s,\alpha,\tau}(X)&\ &\subset&\ h^{-(\tau-\alpha)_- -s_-}\Hbh^{s,\alpha}(X), \\
    \Hbh^{s,\alpha}(X)&\ &\subset&\ h^{-(\tau-\alpha)_+ - s_+}H_{\cop,h}^{s,\alpha,\tau}(X),
  \end{alignat*}
  in the sense that the inclusion maps are uniformly bounded. That is, there exists $C>0$ such that $\|h^{(\tau-\alpha)_- +s_-}u\|_{\Hbh^{s,\alpha}(X)}\leq C\|u\|_{H_{\cop,h}^{s,\alpha,\tau}(X)}$; similarly for the second inclusion.
\end{prop}
\begin{proof}
  We have $H_{\cop,h}^0(X)=L^2_\bop(X)=\Hbh^0(X)$. Moreover, given $V=\tfrac{h}{x+h}W$, $W\in\Vb(X)$, we have $h W u=(x+h)V u$, which implies that for $s\in\N_0$,
  \[
    H_{\cop,h}^s(X)\subset H_{\bop,h}^s(X),\quad
    H_{b,h}^s(X)\subset(x+h)^{-s}H_{\cop,h}^s(X).
  \]
  Multiplication by weights gives
  \begin{alignat*}{2}
    H_{\cop,h}^{s,\alpha,\tau}(X) &\subset x^\alpha(x+h)^{\tau-\alpha}H_{\bop,h}^s(X) &&\subset x^\alpha h^{-(\tau-\alpha)_-}H_{\bop,h}^s(X), \\
    x^\alpha H_{b,h}^s(X) &\subset (x+h)^{-s+\alpha-\tau}\bigl(\tfrac{x}{x+h}\bigr)^\alpha (x+h)^\tau H_{\cop,h}^s(X) &&\subset h^{-s}h^{-(\tau-\alpha)_+}H_{\cop,h}^{s,\alpha,\tau}(X).
  \end{alignat*}
  These inclusions in fact hold for real $s\geq 0$, as can be shown using complex interpolation. (Indeed, using Theorem~\ref{ThmPC} below---whose proof is independent of this proposition---applied to $A_{\digamma,z}:=(x^{-2}h^2 A_{\bop,\digamma}+1)_z$, where $A_{\bop,\digamma}=x^2\Delta_g+\digamma^2$ for any fixed smooth conic metric $g$ and $\digamma\gg 1$, one finds that the complex interpolation space $[x^k H_{\cop,h}^k(X),x^{k'}H_{\cop,h}^{k'}(X)]_\theta$ is $x^s H_{\cop,h}^s(X)$ for $s=\theta k'+(1-\theta)k$; here, taking $\digamma$ large ensures that the Schwartz kernel of $A_{\digamma,z}$ for $\Re z$ in any fixed interval vanishes to any pre-specified order at $\lb$ and $\rb$. Similarly, interpolation spaces of weighted (semiclassical) b-Sobolev spaces are weighted (semiclassical) b-Sobolev spaces as well.) For general real orders $s\in\R$, the claim follows by duality.
\end{proof}

\section{Doubly semiclassical cone calculus and doubly semiclassical resolvent}
\label{SC}

For a fully elliptic $m$-th order semiclassical cone operator $A\in\Diff_{\cop,\semi}^m(X)$, the ultimate goal is to describe the structure of complex powers $A^z$, which using~\eqref{EqIContour} we will accomplish by analyzing the resolvent $(A_h-\tilde\lambda)^{-1}$ for $\tilde\lambda$ on a suitable unbounded contour. Denoting $\tilde h=|\tilde\lambda|^{-1/m}$, the resolvent is a scalar multiple of $(\tilde h^m A_h-\tilde\omega)^{-1}$, $\tilde\omega=\tilde\lambda/|\tilde\lambda|$. The operator we are inverting here is thus of the form
\begin{equation}
\label{EqCIntroOp}
  \tilde A_{h,\tilde h,\tilde\omega} := \tilde h^m A_h-\tilde\omega, \qquad A\in\Diff_{\cop,\semi}^m(X),\ h,\tilde h>0,\ \tilde\omega\in\Omega\subset\C.
\end{equation}
It thus has two semiclassical parameters $h,\tilde h$. For our purposes, we think of $\tilde h$ as the smaller one: for any fixed $h>0$, a suitable integral over $(\tilde h,\tilde\omega)$ (using a description of the resolvent for $\tilde h$ near $0$) will produce complex powers, whose uniform behavior as $h\to 0$ we wish to describe.

We construct the \emph{d}oubly \emph{s}emiclassical \emph{c}one (or `dsc') calculus $\Psidsc(X)$ and associated normal operators in~\S\ref{SsCC}, after a warm-up on manifolds without boundary in~\S\ref{SsCP}. The composition properties of the dsc-calculus are proved in~\S\ref{SsCCo}. The subsequent construction of the resolvent of $\tilde A_{h,\tilde h,\tilde\omega}$ in~\S\ref{SsCR} will then be a simple application of the calculus.

\subsection{Prologue: doubly semiclassical calculus on closed manifolds}
\label{SsCP}

Suppose, only in this section, that $X$ is a closed manifold (compact without boundary). The semiclassical double space is $X^2_\semi=[[0,1)_h\times X^2;\{0\}\times\Delta_X]$, where $\Delta_X\subset X^2$ is the diagonal, and elements of $\Psi_\semi(X)$ are characterized via their Schwartz kernels: they vanish to infinite order on the lift, $\sface_0$, of $h=0$, and are conormal to the lift $\Delta_\semi$ of $[0,1)_h\times\Delta_X$. The principal symbol takes values in $(S^m/h S^{m-1})(T^*X)$. We denote the front face of $X^2_\semi$ by $\dface_0$. 

We define the doubly semiclassical double space as
\begin{equation}
\label{EqCPDouble}
  X^2_{\semi\tilde\semi} = [[0,1)_{\tilde h}\times X^2_\semi;\{0\}\times\Delta_\semi],
\end{equation}
see Figure~\ref{FigCP2}. The space $X^2_{\semi\tilde\semi}$ has four boundary hypersurfaces: the lifts $\sface$, $\dface$ of $[0,1)_{\tilde h}\times\sface_0$, $[0,1)_{\tilde h}\times\dface_0$, as well as the lift of $\tilde h=0$, denoted $\wt\sface$, and the front face of the blow-up in~\eqref{EqCPDouble}, denoted $\wt\dface$. The lift of $[0,1)_{\tilde h}\times\Delta_\semi$ to $X^2_{\semi\tilde\semi}$ is denoted $\Delta_{\semi\tilde\semi}$. Then $\Psi_{\semi\tilde\semi}^m(X)$ consists of distributions conormal to $\Delta_{\semi\tilde\semi}$, of order $m-\half$ (cf.\ Definition~\ref{DefRCPsdo}), which vanish to infinite order at $\sface\cup\wt\sface$.

\begin{figure}[!ht]
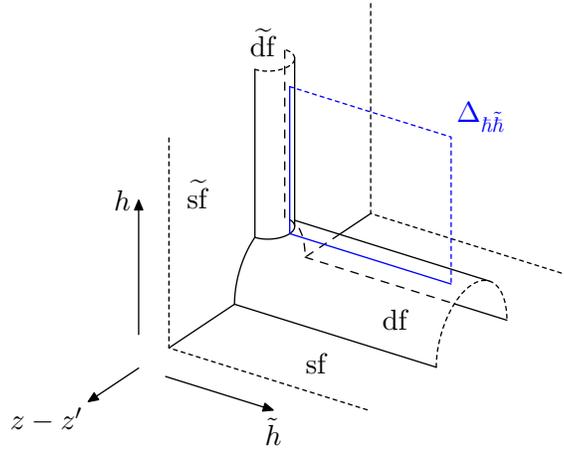

\inclfig{FigCP2}
\caption{The doubly semiclassical double space $X^2_{\semi\tilde\semi}$, its boundary hypersurfaces, and the `diagonal' $\Delta_{\semi\tilde\semi}$. We denote by $z$, resp.\ $z'$ the pullback of local coordinates on $X$ along the left, resp.\ right projection $X\times X\to X$.}
\label{FigCP2}
\end{figure}

This definition is \emph{not symmetric} in $h$ and $\tilde h$. However, if one switches the roles of $h$ and $\tilde h$, a neighborhood of the lift of $[0,1)_h\times[0,1)_{\tilde h}\times\Delta_X$ to the resulting space is naturally diffeomorphic to a neighborhood of $\Delta_{\semi\tilde\semi}$ as defined above, and moreover the spaces of $\CI$ functions vanishing to infinite order at the lifts of $h=0$ and $\tilde h=0$ are naturally identified via the pointwise identity map on the interior $(0,1)_h\times(0,1)_{\tilde h}\times X^2$. Therefore, the order of the blow-ups in~\eqref{EqCPDouble} is irrelevant in practice; our present choice of ordering will however be more convenient for the discussion of complex powers in~\S\ref{SP}. (Any definition which is symmetric in $h,\tilde h$ is necessarily more complicated and thus less convenient here.)

Let us make this concrete in local coordinates, and denote by $z\in\R^n$ local coordinates on $X$ pulled back to $X^2$ along the left projection, and by $z'$ their pullback along the right projection. Local coordinates on $[0,1)_{\tilde h}\times X^2_\semi$ near $[0,1)_{\tilde h}\times(\dface\setminus\wt\sface)$ are then $(\tilde h,h,z,Z)$, $Z=(z-z')/h$. Upon blowing up $\tilde h=Z=0$, local coordinates near $\wt\dface\setminus\wt\sface$ are $(\tilde h,h,z,Z')$, $\tilde Z=Z/\tilde h=(z-z')/(h\tilde h)$, with $\Delta_{\semi\tilde\semi}=\{\tilde Z=0\}$. Semiclassical quantizations
\begin{equation}
\label{EqCPQu}
  (2\pi h\tilde h)^{-n}\int e^{i(z-z')\zeta/(h\tilde h)}a(h,\tilde h,z,\zeta)\,\dd\zeta
\end{equation}
of symbols $a$ thus lift to conormal distributions $\Psi_{\semi\tilde\semi}(X)$.

The principal symbol map is
\begin{equation}
\label{EqCPSy}
  0 \to h\tilde h\Psi^{m-1}(X) \to \Psi^m(X) \xra{\sigmahth_m} (S^m/h\tilde h S^{m-1})(T^*X) \to 0.
\end{equation}
It is defined by observing that $\wt\dface\cong[0,1)_h\times\ol{T}X$ via continuous extension of $(0,1)_h\times T X\ni(h,z,v)\mapsto\lim_{\tilde h\to 0}(h,z,z+h\tilde h v)\in\wt\dface$, so the restriction of $\kappa\in\Psi_{\semi\tilde\semi}^m(X)$ to $\wt\dface$ can be Fourier-transformed along the fibers of $T X$, giving an element of $\CI([0,1)_h;S^m(T^*X))$ capturing $\kappa$ modulo $\tilde h\Psi_{\semi\tilde\semi}^m(X)$, which is `half' of~\eqref{EqCPSy}. The other `half' arises by restricting to $\dface$. Now $\dface^\circ$ is naturally identified with $(0,1)_{\tilde h}\times T X$ via $(0,1)_{\tilde h}\times T X\ni(\tilde h,z,v)\mapsto\lim_{h\to 0}(\tilde h,z,z+h v)\in\dface^\circ$; the semiclassical quantization~\eqref{EqCPQu} restricts to $(2\pi\tilde h)^{-n}\int e^{i Z\zeta/\tilde h}a(0,\tilde h,z,\zeta)\,\dd\zeta$, hence its \emph{semiclassical} (in $\tilde h$) Fourier transform (which resolves the degeneration as $\tilde h\to 0$) gives the desired principal symbol $a(0,\tilde h,z,\zeta)$ at $\dface$, capturing $\kappa$ modulo $h\Psi_{\semi\tilde\semi}^m(X)$.

Thus, we can construct parametrices elliptic operators $A\in\Psi_{\semi\tilde\semi}^m(X)$ by means of the usual symbolic parametrix construction, producing $B\in\Psi_{\semi\tilde\semi}^{-m}(X)$ such that $R=A B-I$, $R'=B A-I\in h^\infty\tilde h^\infty\Psi_{\semi\tilde\semi}^{-\infty}(X)=\CIdot([0,1)^2\times X^2)$. For $h,\tilde h$ small, $I+R$ and $I+R'$ can then be inverted by a Neumann series, and one obtains $A^{-1}\in\Psi_{\semi\tilde\semi}^{-m}(X)$.

\subsection{Definition of the dsc-calculus}
\label{SsCC}

From now on, $X$ is again a connected compact manifold with connected embedded boundary $\pa X$. Let us denote the boundary hypersurfaces of $X^2_{\cop\semi}$ by adding a subscript `$0$', so $\ff_0$, $\lb_0$, $\rb_0$, $\tface_0$, $\dface_0$, $\sface_0$. Recall that $\Delta_{\cop\semi}\subset X^2_{\cop\semi}$ denotes the lifted diagonal.

\begin{definition}
\label{DefCC2}
  The \emph{dsc-double space} is
  \begin{equation}
  \label{EqCC2}
    X^2_\dscop := \left[ [0,1)_{\tilde h}\times X^2_{\cop\semi}; \{0\}\times\ff_0; \{0\}\times\Delta_{\cop\semi} \right].
  \end{equation}
  The boundary hypersurfaces are:
  \begin{itemize}
  \item $\lb,\ff,\rb,\tface,\dface,\sface$ are the lifts of $[0,1)_{\tilde h}\times\lb_0$, $[0,1)_{\tilde h}\times\ff_0$, etc;
  \item $\wt\tface$, $\wt\dface$, $\wt\sface$ are the lifts of $\{0\}\times\ff_0$, $\{0\}\times\Delta_{\cop\semi}$, $\tilde h^{-1}(0)$.
  \end{itemize}
  The \emph{dsc-diagonal} $\Delta_\dscop$ is the lift of $[0,1)_{\tilde h}\times\Delta_{\cop\semi}$.
\end{definition}

Note the similarity with how $X^2_{\cop\semi}$ is constructed from $X^2_\bop$ in Definition~\ref{DefRCDouble}. Denote by $\rho_H$ a defining function of the boundary hypersurface $H=\ff,\lb$, etc. Let $\pi_\bop\colon X^2_{\cop\semi\tilde\semi}\to X^2_\bop$ denote the lift of the projection $[0,1)_{\tilde h}\times[0,1)_h\times X_\bop^2\to X_\bop^2$.
\begin{definition}
\label{DefCCPsdo}
  Let $m\in\R$. Then
  \[
    \Psidsc^m(X) = \left\{ \kappa\in I^{m-\mhalf}\left(X^2_\dscop;\Delta_\dscop;\rho_\dface^{-n}\rho_{\wt\dface}^{-n}\pi_\bop^*\bigl(\Omegab^\mhalf(X^2_\bop)\bigr)\right) \colon \kappa\equiv 0\ \text{at}\ \lb\cup\rb\cup\sface\cup\wt\sface \right\},
  \]
  with `$\equiv$' denoting equality of Taylor series. If $\cE=(\cE_\lb,\cE_\ff,\cE_\rb,\cE_\tface,\cE_{\wt\tface})$ is a collection of index sets, we set
  \[
    \Psidsc^{-\infty,\cE}(X) := \cA_\phg^\cE(X^2_\dscop),
  \]
  with index sets $\emptyset$ at $\dface$, $\sface$, $\wt\dface$, $\wt\sface$ (not made explicit in the notation).
\end{definition}

Let us verify that this is a sensible definition by considering an operator
\begin{subequations}
\begin{equation}
\label{EqCEx1}
  \tilde A_{h,\tilde h}=\tilde h^m h^m x^{-m}(x D_x)^j P_{m-j}(h,x,y,D_y).
\end{equation}
{Note that a constant term $\tilde\omega$ is of this type, too, for $m=0$, so the calculations here cover all operators of the form~\eqref{EqCIntroOp}.} Identified with its Schwartz kernel, \eqref{EqCEx1} is, in $(\tilde h,h,s=\frac{x}{x'},x',y,y')$ coordinates and dropping b-half density factors for brevity,
\begin{equation}
\label{EqCEx2}
  \tilde A_{h,\tilde h} = \tilde h^m h^m(x')^{-m}s^{-m}(s D_s)^j P_{m-j}(h,s x',y,D_y)\bigl(\delta(s-1)\delta(y-y')\bigr),
\end{equation}
cf.\ \eqref{EqRCEx1}. We lift this to $[0,1)_{\tilde h}\times X^2_{\cop\semi}$ and compute its form in various regimes:

\begin{enumerate}[leftmargin=\enummargin]
  \item In a neighborhood of $(0,1)_{\tilde h}\times(\ff_0\setminus\rb_0)$, we can use the coordinates $(\tilde h,h,s,\hat x'=\frac{x'}{h},y,y')$, in which
    \begin{equation}
    \label{EqCEx3}
      \tilde A_{h,\tilde h} = \tilde h^m(\hat x')^{-m} s^{-m}(s D_s)^j P_{m-j}(h,h s\hat x',y,D_y)\bigl(\delta(s-1)\delta(y-y')\bigr).
    \end{equation}
    Thus, in these coordinates, $\tilde A_{h,\tilde h}$ is a semiclassical cone operator in $\tilde h$, with parametric dependence on $h$. (This is the uniform, down to $h=0$, version of the obvious fact that $\tilde A_{h,\tilde h}$ for fixed $h=h_0>0$ is a semiclassical cone operator in $\tilde h$.) In these coordinates, $\tface$ is defined by $h=0$, so from~\eqref{EqCEx3} one expects the restriction $N_\tface(\tilde A)$ of $\tilde A$ to $\tface$, defined in detail in~\eqref{EqCCNtf} below, to be an $\tilde h$-semiclassical cone operator. The first blow-up in~\eqref{EqCC2} thus resolves its structure at $\tilde h=\hat x'=0$. Near $\wt\tface$, we consider two projective coordinate systems.

  \begin{enumerate}[leftmargin=\enummargin]
  \item The first is a neighborhood of the interior of $\wt\tface\cap\ff$ where we can use coordinates $(\tilde h,h,s,\tilde x'=\frac{\hat x'}{\tilde h},y,y')$; in these,
    \begin{equation}
    \label{EqCEx4}
      \tilde A_{h,\tilde h} = (\tilde x')^{-m} s^{-m}(s D_s)^j P_{m-j}(h,h\tilde h s\tilde x',y,D_y)\bigl(\delta(s-1)\delta(y-y')\bigr)
    \end{equation}
    is a weighted operator of the form~\eqref{EqRCEx2}. Restriction to $\wt\tface$, with defining function $\tilde h$, gives a family of operators with smooth dependence on $h$,
    \begin{equation}
    \label{EqCEx5Ntf2}
      N_{\wt\tface}(\tilde A) = (\tilde x')^{-m}s^{-m}(s D_s)^j P_{m-j}(h,0,y,D_y)\bigl(\delta(s-1)\delta(y-y')\bigr),
    \end{equation}
    cf.\ \eqref{EqCCNtf2} below.

  \item The second projective coordinate system near $\wt\tface$ we consider is a neighborhood of the interior of $\wt\tface\cap\wt\dface$, where we can use $(\hat{\tilde h}=\frac{\tilde h}{\hat x'},h,s_{\tilde h}=(\log s)/\hat{\tilde h},\hat x',y,Y_{\tilde h}=(y-y')/\hat{\tilde h})$, so
    \begin{equation}
    \label{EqCEx6sf}
      \tilde A_{h,\tilde h} = e^{-m\hat{\tilde h} s_{\tilde h}} D_{s_{\tilde h}}^j \hat{\tilde h}^{m-j} P_{m-j}(h,h e^{\hat{\tilde h}s_{\tilde h}}\hat x',y,\hat{\tilde h}{}^{-1} D_{Y_{\tilde h}})\bigl(\delta(s_{\tilde h})\delta(Y_{\tilde h})\bigr)\cdot\hat{\tilde h}{}^{-n}.
    \end{equation}
    Since this is smooth and nondegenerate down to the boundary hypersurfaces ($\hat{\tilde h}=0$ and $\hat x'=0$) in these coordinates, this justifies the blow-up of $\{0\}\times\Delta_{\cop\semi}$ in~\eqref{EqCC2} (reflected in the definition of $s_{\tilde h}$ and $Y_{\tilde h}$ here) as well as the factor $\rho_{\wt\dface}^{-n}$ in Definition~\ref{DefCCPsdo}.
  \end{enumerate}

  \item Near $(0,1)_{\tilde h}\times(\tface_0\cap(\dface_0\setminus\sface_0))$, we use $(\tilde h,\hat h=\frac{h}{x'},s_h=(\log s)/\hat h,x',y,Y_h=(y-y')/\hat h)$, so
  \begin{equation}
  \label{EqCEx7tf}
    \tilde A_{h,\tilde h} = \tilde h^m e^{-m\hat h s_h}D_{s_h}^j \hat h^{m-j}P_{m-j}(\hat h x',e^{\hat h s_h}x',y,\hat h^{-1}D_{Y_h})\bigl(\delta(s_h)\delta(Y_h)\bigr)\cdot \hat h^{-n}.
  \end{equation}
  The blow-up of $\{0\}\times\Delta_{\cop\semi}$ (defined locally by $\tilde h=s_h=Y_h=0$) resolves the $\tilde h$-semiclassical degeneration: near $\tface\cap(\dface\setminus\sface)\cap(\wt\dface\setminus\wt\sface)$ (where $\tilde h\gtrsim|s_h|,|Y_h|$), we can use the local coordinates $(\tilde h,\hat h,s_{\tilde h}=\frac{s_h}{\tilde h},x',y,Y_{\tilde h}=\frac{Y_h}{\tilde h})$, so
  \begin{equation}
  \label{EqCEx8ttf}
    \tilde A_{h,\tilde h} = e^{-m\hat h\tilde h s_{\tilde h}}D_{s_{\tilde h}}^j (\hat h\tilde h)^{m-j}P_{m-j}(\hat h x',e^{\hat h\tilde h s_{\tilde h}}x',y,(\hat h\tilde h)^{-1}D_{Y_{\tilde h}})\bigl(\delta(s_{\tilde h})\delta(Y_{\tilde h})\bigr)\cdot (\hat h\tilde h)^{-n}.
  \end{equation}
  In these coordinates, defining functions of $\tface$, $\dface$, $\wt\dface$ are $x'$, $\hat h$, $\tilde h$, respectively; this fully justifies the density factor in Definition~\ref{DefCCPsdo}.
\end{enumerate}
\end{subequations}

\begin{rmk}
\label{RmkCCNondeg}
  A further consequence of these calculations is that $\tilde A_{h,\tilde h}\in\rho_\ff^{-m}\Psidsc^m(X)$, and it is non-degenerate as such, in the sense that the principal symbol of the diagonal singularity is a nonzero multiple of the b-principal symbol of $(x D_x)^j P_{m-j}(h,x,y,D_y)$.
\end{rmk}

The principal symbol map is denoted $\sigmadsc_m$ and fits into the short exact sequence
\begin{equation}
\label{EqCCSy}
  0 \to \rho_\dface\rho_{\wt\dface}\Psidsc^{m-1}(X) \to \Psidsc^m(X) \xra{\sigmadsc_m} (S^m/\rho_\dface\rho_{\wt\dface}S^{m-1})(N^*\Delta_\dscop) \to 0,
\end{equation}
cf.\ \eqref{EqRCSy}. At $\ff\cong[0,1)_h\times[0,1)_{\tilde h}\times\ff_\bop$ (where we recall that $\ff_\bop$ is the front face of $X_\bop^2$, as discussed after~\eqref{EqRCX2b}), we have a 2-parameter family of b-normal operators, giving a map $N_\ff$ with
\begin{equation}
\label{EqCCNff}
  0 \to \rho_\ff\Psidsc^m(X) \to \Psidsc^m(X) \xra{N_\ff} \CI\left([0,1)_h\times[0,1)_{\tilde h};\Psi_{\bop,I}^m({}^+N\pa X)\right) \to 0,
\end{equation}
cf.\ \eqref{EqRCNff}. Similarly, restriction to the transition face for the second semiclassical parameter, $\wt\tface$, which is fibered over $[0,1)_h$ gives a normal operator map
\begin{equation}
\label{EqCCNtf2}
  0 \to \rho_{\wt\tface}\Psidsc^m(X) \to \Psidsc^m(X) \xra{N_{\wt\tface}} \CI\left([0,1)_h;\Psi_{\bop,\scop}^m(\ol{{}^+N\pa X})\right) \to 0,
\end{equation}

The transition face $\tface$ is \emph{not} a smooth fibration over $[0,1)_{\tilde h}$, see Figure~\ref{FigCCtf}; rather, by restricting the blow-up procedure~\eqref{EqCC2} to $[0,1)_{\tilde h}\times\tface_0$, we find that
\[
  \tface = \left[ [0,1)_{\tilde h}\times\tface_0; \{0\}\times(\tface_0\cap\ff_0); \{0\}\times(\tface_0\cap\Delta_{\cop\semi}) \right].
\]
Recalling from~\eqref{EqRCtfModel} that $\tface_0\cong(\ol{{}^+N\pa X})^2_{\bop,\scop}$, we thus conclude that $\tface$ is the double space for operators which are semiclassical (with respect to $\tilde h$) cone operators near $\pa_0\ol{{}^+N\pa X}$ and semiclassical scattering operators near $\pa_\infty\ol{{}^+N\pa X}$. Near the lift of $[0,1)_{\tilde h}\times(\tface_0\setminus(\sface_0\cap\dface_0))$, the boundary hypersurface $\tface$ is naturally identified with $({}^+N\pa X)^2_{\cop\tilde\semi}$ (with the subscript `$\tilde\semi$' indicating that the semiclassical parameter is $\tilde h$); the following definition (with $\bfX=\ol{{}^+N\pa X}$) captures the global structure of $\tface$:

\begin{definition}
\label{DefCCNtf2}
  Let $\bfX$ be a manifold whose boundary is the disjoint union $\pa\bfX=\pa_0\bfX\sqcup\pa_\infty\bfX$ of two connected embedded hypersurfaces. Denote the b-left boundary, b-front face, b-right boundary, scattering front face, and lifted diagonal in $\bfX^2_{\bop,\scop}$ by $\mathbf{lb}_\bop$, $\mathbf{ff}_\bop$, $\mathbf{rb}_\bop$, $\mathbf{ff}_{0,\scop}$, and $\mathbf{\Delta}_{\bop,\scop}$. (See Figure~\ref{FigCCNtf2}.) Then the \emph{semiclassical cone-scattering double space} is
  \[
    \bfX^2_{\cop,\scop,\tilde\semi} := \left[ [0,1)_{\tilde h} \times \bfX^2_{\bop,\scop}; \{0\}\times\mathbf{ff}_\bop; \{0\}\times\mathbf{\Delta}_{\bop,\scop} \right].
  \]
  Denote by $\mathbf{ff}_\scop$ the lift of $[0,1)_{\tilde h}\times\mathbf{ff}_{0,\scop}$, and by $\mathbf{lb}$, $\mathbf{ff}$, $\mathbf{rb}$, $\wt{\mathbf{tf}}$, $\wt{\mathbf{df}}$, $\wt{\mathbf{sf}}$ the closures of the corresponding non-bold-face hypersurfaces of $(\bfX\setminus\pa_\infty\bfX)^2_{\cop\tilde\semi}$ inside of $\bfX^2_{\cop,\scop,\tilde\semi}$. The lift of $[0,1)_{\tilde h}\times\mathbf{\Delta}_{\bop,\scop}$ is denoted $\mathbf{\Delta}_{\cop,\scop,\tilde\semi}$.
\end{definition}

\begin{figure}[!ht]
\centering
\includegraphics{FigCCNtf2}
\caption{Illustration of the double space $\bfX_{\bop,\scop}^2$.}
\label{FigCCNtf2}
\end{figure}

Let $\rho_H$ denote a boundary defining function of $H$ for $H=\mathbf{lb},\mathbf{ff}$, etc.
\begin{definition}
\label{DefCCNtfOp}
  For $\bfX$ as in Definition~\ref{DefCCNtf2}, we define
  \[
    \Psi_{\cop,\scop,\tilde\semi}^m(\bfX) := \left\{ \kappa\in I^{m-\mfrac14}\left(\bfX^2_{\cop,\scop,\tilde\semi};\mathbf{\Delta}_{\cop,\scop,\tilde\semi}; \rho_{\wt{\mathbf{df}}}^{-n}\rho_{\mathbf{ff}_\scop}^{-n}\cdot\pi_\bop^*\left(\Omegab^\mhalf(\bfX^2_\bop)\right)\right) \right\},
  \]
  where $\pi_\bop$ is the lift of the projection $[0,1)_{\tilde h}\times\bfX^2_\bop\to\bfX^2_\bop$, so $\pi_\bop\colon\bfX^2_{\cop,\scop,\tilde\semi}\to\bfX^2_\bop$. Given a collection $\cE=(\cE_{\mathbf{lb}},\cE_{\mathbf{ff}},\cE_{\mathbf{rb}},\cE_{\wt{\mathbf{tf}}})$ of index sets, we denote by
  \[
    \Psi_{\cop,\scop,\tilde\semi}^{-\infty,\cE}(\bfX) = \cA_\phg^\cE(\bfX^2_{\cop,\scop,\tilde\semi}),
  \]
  with index set $\emptyset$ at all boundary hypersurfaces except $\mathbf{lb}$, $\mathbf{ff}$, $\mathbf{rb}$, $\wt{\mathbf{tf}}$.
\end{definition}

For $\bfX=\ol{{}^+N\pa X}$, we have
\begin{equation}
\label{EqCCtftf}
  \tface = (\ol{{}^+N\pa X})^2_{\cop,\scop,\tilde\semi}
\end{equation}
and $\mathbf{lb}=\tface\cap\lb$, $\mathbf{ff}=\tface\cap\ff$, $\mathbf{rb}=\tface\cap\rb$, $\wt{\mathbf{tf}}=\tface\cap\wt\tface$, $\mathbf{ff}_\scop=\tface\cap\dface$, $\wt{\mathbf{sf}}=\tface\cap\wt\sface$, $\wt{\mathbf{df}}=\tface\cap\wt\dface$. See Figure~\ref{FigCCtf}, where we also indicate $\mathbf{sf}=\tface\cap\sface$.

\begin{figure}[!ht]
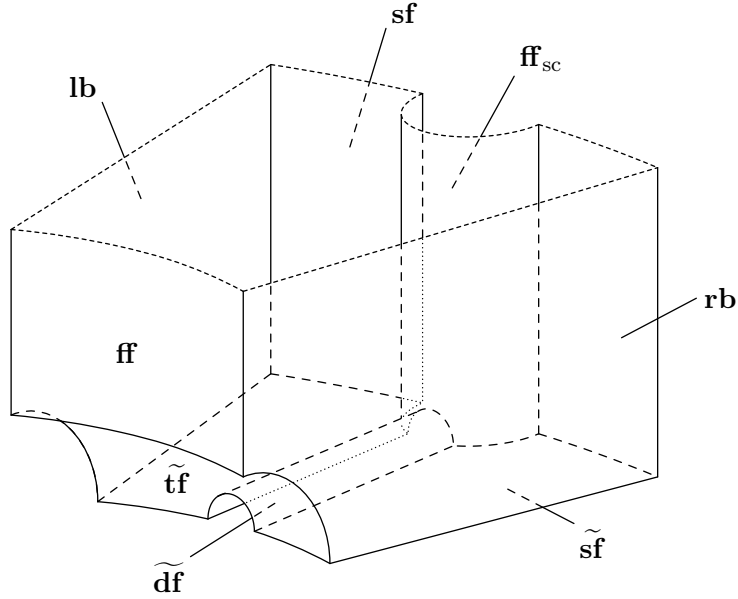

\inclfig{FigCCtf}
\caption{Structure of $\tface$ and its boundary hypersurfaces as stated after~\eqref{EqCCtftf}. The local coordinates used in~\eqref{EqCEx4} and~\eqref{EqCEx6sf} (without $h$, which defines $\tface$) are valid near the interiors of $\mathbf{ff}\cap\wt{\mathbf{tf}}$ and $\wt{\mathbf{tf}}\cap\wt{\mathbf{df}}$, respectively; \eqref{EqCEx7tf} and \eqref{EqCEx8ttf} (without $\hat h$, which defines $\tface$) are valid near the interiors of $\mathbf{ff}_\scop\setminus\wt{\mathbf{df}}$ and $\mathbf{ff}_\scop\cap\wt{\mathbf{df}}$, respectively.}
\label{FigCCtf}
\end{figure}

With these preparations, the normal operator at $\tface$ is given by restriction to $\tface$, giving a short exact sequence
\begin{equation}
\label{EqCCNtf}
  0 \to \rho_\tface\Psidsc^m(X) \to \Psidsc^m(X) \xra{N_\tface} \Psi_{\cop,\scop,\tilde\semi}^m(\ol{{}^+N\pa X}) \to 0.
\end{equation}
Since $\tface\cap\wt\tface$ is equal to the transition face of $({}^+N\pa X)^2_{\cop\tilde\semi}$, the $\tface$- and $\wt\tface$-normal operators are related by
\begin{equation}
\label{EqCCNtfRel}
  N_\tface(N_\tface(A))=N_{\wt\tface}(A)|_{h=0}.
\end{equation}
Here, the leftmost $N_\tface$ is the $\tface$-normal operator map for semiclassical cone operators with the single semiclassical parameter $\tilde h$.

\subsection{Composition}
\label{SsCCo}

As in~\S\ref{SsRC}, it is convenient to introduce an extended dsc-double space, based on the extended semiclassical cone double space ${}'X^2_{\cop\semi}$ from Definition~\ref{DefRCExt}. Let us adorn the names of the boundary hypersurfaces of ${}'X^2_{\cop\semi}$ with a subscript `$0$' as before Definition~\ref{DefCC2}.

\begin{definition}
\label{DefCCoExt2}
  The \emph{extended dsc-double space} is
  \begin{equation}
  \label{EqCCoExt2}
    {}'X^2_\dscop := \left[ [0,1)_{\tilde h}\times{}'X^2_{\cop\semi}; \{0\}\times\ff_0; \{0\}\times\Delta_{\cop\semi}; \{0\}\times\lb_0; \{0\}\times\rb_0 \right].
  \end{equation}
  We label its boundary hypersurfaces as follows:
  \begin{itemize}
  \item $\lb$, $\ff$, $\rb$, $\tface$, $\dface$, $\sface$, $\tlb$, $\trb$ are the lifts of $[0,1)_{\tilde h}\times\lb_0$, $[0,1)_{\tilde h}\times\ff_0$ etc;
  \item $\wt\tface$, $\wt\dface$, $\wt\sface$ are the lifts of $\{0\}\times\ff_0$, $\{0\}\times\Delta_{\cop\semi}$, $\tilde h^{-1}(0)$;
  \item $\wt\tlb$, $\wt\trb$ are the lifts of $\{0\}\times\lb_0$, $\{0\}\times\rb_0$.
  \end{itemize}
\end{definition}

The lift of $\Delta_\dscop$ to ${}'X^2_\dscop$ will be denoted by the same letter. Schwartz kernels of operators in $\Psidsc^m(X)$ and $\Psidsc^{-\infty,\cE}(X)$ lift to ${}'X^2_\dscop$ and vanish to infinite order at $\tlb$, $\trb$, $\wt\tlb$, $\wt\trb$.

The extended dsc-triple space is constructed on the basis of the extended semiclassical cone triple space ${}'X^3_{\cop\semi}$ from Definition~\ref{DefRCExt}. Again, we adorn the names of hypersurfaces of ${}'X^3_{\cop\semi}$ with a subscript `$0$', prior to any subscripts they may already have. The triple and lifted diagonals in ${}'X^3_{\cop\semi}$ are still denoted $\Delta_{\cop\semi,3}$, $\Delta_{\cop\semi,\bullet}$, $\bullet=F,S,C$.
\begin{definition}
\label{DefCCoExt3}
  Define the collections of submanifolds of $[0,1)_{\tilde h}\times{}'X^3_{\cop\semi}$
  \[
    \cF_{\cop\semi} = \{ \{0\}\times\ff_{0,\bullet}\}, \quad
    \cB_{\cop\semi} = \{ \{0\}\times\bface_{0,\bullet}\}, \quad
    \cD_{\cop\semi} = \{ \{0\}\times\Delta_{\cop\semi,\bullet}\}, \qquad
    \bullet=F,S,C.
  \]
  The \emph{extended dsc-triple space} is then
  \begin{align*}
    {}'X^3_\dscop := \left[ [0,1)_{\tilde h}\times{}'X^3_{\cop\semi}; \{0\}\times\fff_0; \cF_{\cop\semi}; \cB_{\cop\semi}; \{0\}\times\Delta_{\cop\semi,3}; \cD_{\cop\semi} \right].
  \end{align*}
  We label its boundary hypersurfaces as follows:
  \begin{itemize}
  \item $\fff,\ff_\bullet,\bface_\bullet,\tfface,\tface_\bullet,\tbface_\bullet$, $\dfface,\dface_\bullet,\sfface$ are the lifts of $[0,1)_{\tilde h}\times\fff_0$, $[0,1)_{\tilde h}\times\ff_{0,\bullet}$, etc;
  \item $\wt\tfface$, $\wt\tface_\bullet$ are the lifts of $\{0\}\times\fff_0$, $\{0\}\times\ff_{0,\bullet}$;
  \item $\wt\tbface_\bullet$ is the lift of $\{0\}\times\bface_{0,\bullet}$;
  \item $\wt\dfface$, $\wt\dface_\bullet$ are the lifts of $\{0\}\times\Delta_{\cop\semi,3}$, $\{0\}\times\Delta_{\cop\semi,\bullet}$;
  \item $\wt\sfface$ is the lift of $\{0\}\times{}'X^3_{\cop\semi}$.
  \end{itemize}
  Lastly, the p-submanifolds $\Delta_{\dscop,3}$, $\Delta_{\dscop,\bullet}$ are the lifts of $[0,1)_{\tilde h}\times\Delta_{\cop\semi,3}$, $[0,1)_{\tilde h}\times\Delta_{\cop\semi,\bullet}$.
\end{definition}

Recalling the stretched projections $\pi_{\cop\semi,\bullet}\colon{}'X^3_{\cop\semi}\to{}'X^2_{\cop\semi}$ from~\eqref{EqRC3Proj}, we have:
\begin{lemma}
\label{LemmaCCoProj}
  The stretched projections $\Id\times\pi_{\cop\semi,\bullet}\colon[0,1)_{\tilde h}\times{}'X^3_{\cop\semi}\to[0,1)_{\tilde h}\times{}'X^2_{\cop\semi}$ lift to b-fibrations
  \begin{equation}
  \label{EqCCoProj}
    \pi_{\dscop,\bullet} \colon {}'X^3_\dscop \to {}'X^2_\dscop,\quad \bullet=F,S,C.
  \end{equation}
\end{lemma}
\begin{proof}
  We consider the case of $\pi_{\dscop,F}$; the proof is very similar to that of Lemma~\ref{LemmaRC3Proj}, so we shall be brief. The starting point is that $\Id\times\pi_{\cop\semi,F}$ is a b-fibration. The lift of this to the blow-up $\{0\}\times\ff_0$ in the target and its preimages $\{0\}\times\fff_0$, $\{0\}\times\ff_{0,F}$ in the domain is the b-fibration
  \[
    \left[ [0,1)_{\tilde h}\times{}'X^3_{\cop\semi}; \{0\}\times\fff_0; \{0\}\times\ff_{0,F}\right] \to \left[ [0,1)_{\tilde h}\times{}'X^2_{\cop\semi}; \{0\}\times\ff_0 \right].
  \]
  We next blow up the lift of $\{0\}\times\lb_0$, resp.\ $\{0\}\times\rb_0$ in the range and its preimages---the lifts of $\{0\}\times\ff_{0,C}$, $\{0\}\times\bface_{0,S}$, resp.\ $\{0\}\times\ff_{0,S}$, $\{0\}\times\bface_{0,C}$---in the domain, with the projection map lifting to a b-fibration. The lift of $\{0\}\times\bface_{0,F}$ gets mapped diffeomorphically to the lift of $\{0\}\times{}'X^2_{\cop\semi}$; thus, upon blowing up $\{0\}\times\bface_{0,F}$, we get a b-fibration
  \[
    \left[ [0,1)_{\tilde h}\times{}'X^3_{\cop\semi}; \{0\}\times\fff_0; \cF_{\cop\semi}; \cB_{\cop\semi} \right] \to \left[ [0,1)_{\tilde h}\times{}'X^2_{\cop\semi}; \{0\}\times\ff_0; \{0\}\times\lb_0; \{0\}\times\rb_0 \right].
  \]
  Blowing up the lift of the diagonal $\{0\}\times\Delta_{\cop\semi}$ in the range and the lift of $\{0\}\times\Delta_{\cop\semi,F}$ in the domain, this map lifts to a b-fibration, and remains such upon lifting to the blow-up of the triple diagonal at $\tilde h=0$, $\{0\}\times\Delta_{\cop\semi,3}$. The lift of this map to subsequent blow-ups of the other diagonals at $\tilde h=0$, $\{0\}\times\Delta_{\cop\semi,\bullet}$, $\bullet=S,C$---which lift to be disjoint from the lift of $\{0\}\times\Delta_{\cop\semi,F}$---is still a b-fibration, as the latter submanifolds map diffeomorphically onto the lift of $\{0\}\times\bface_{0,F}$.
\end{proof}

This allows for a simple proof of composition properties of the dsc-calculus:

\begin{prop}
\label{PropCCo}
  Let $A_j\in\Psidsc^{m_j}(X)$, $A_j'\in\Psidsc^{-\infty,\cE_j}$ for $j=1,2$. Then:
  \begin{enumerate}
  \item $A_1\circ A_2\in\Psidsc^{m_1+m_2}(X)$, $A_1\circ A'_2\in\Psidsc^{-\infty,\cE_2}(X)$, $A'_1\circ A_2\in\Psidsc^{-\infty,\cE_1}(X)$.
  \item Write $\cE_j=(\cE_{j,\lb},\cE_{j,\ff},\cE_{j,\rb},\cE_{j,\tface},\cE_{j,\wt\tface})$, and suppose that $\inf\Re(\cE_{1,\rb}+\cE_{2,\lb})>0$. Then $A'_1\circ A'_2\in\Psidsc^{-\infty,\cF}(X)$, where $\cF=(\cF_\lb,\cF_\ff,\cF_\rb,\cF_\tface,\cF_{\wt\tface})$ is defined by~\eqref{EqRCCompIndex} and $\cF_{\wt\tface}=\cE_{1,\wt\tface}+\cE_{2,\wt\tface}$.
  \end{enumerate}
\end{prop}
\begin{proof}
  The proof is similar to that of Proposition~\ref{PropRCComp}, but notationally more complex. We shall only discuss the composition of remainder terms with polyhomogeneous expansions. Denoting by $K'_j$ the Schwartz kernel of $A'_j$, we write the Schwartz kernel $K$ of $A'_1\circ A'_2$ as
  \[
    K = (\pi_{\dscop,C})_*\left(\pi_{\dscop,F}^*K'_1 \cdot \pi_{\dscop,S}^*K'_2\right).
  \]
  $K'_1$ is polyhomogeneous on $X^2_\dscop$. We first need to lift it to ${}'X^2_\dscop$; this is not quite automatic since ${}'X^2_\dscop$ is defined as a blow-up not of $X^2_\dscop$ but of
  \begin{equation}
  \label{EqCCoIssue}
    [0,1)_{\tilde h}\times{}'X^2_{\cop\semi} = \left[ [0,1)_{\tilde h}\times X^2_{\cop\semi}; [0,1)_{\tilde h}\times(\lb\cap\sface); [0,1)_{\tilde h}\times(\rb\cap\sface) \right].
  \end{equation}
  Note however that the blow-up of $\{0\}\times\ff_0$ in the definition of $X^2_\dscop$ in~\eqref{DefCC2} commutes with the first blow-up in~\eqref{EqCCoIssue} since $[0,1)_{\tilde h}\times(\lb\cap\sface)$ and $\{0\}\times\ff_0$ are transversal; similarly for the second blow-up. Thus, we can lift $K'_1$ from $X^2_\dscop=[[0,1)_{\tilde h}\times X^2_{\cop\semi};\{0\}\times\ff_0;\{0\}\times\Delta_{\cop\semi}]$ to the blow-up of~\eqref{EqCCoIssue} along $\{0\}\times\ff_0$ and $\{0\}\times\Delta_{\cop\semi}$ (the latter being disjoint from the submanifolds involving $\lb$ and $\rb$), and thus to ${}'X^2_\dscop$ by lifting to the final two blow-ups in~\eqref{EqCCoExt2}. This lift (which we denote by $K'_1$ still) has non-trivial index sets only at $\lb,\ff,\rb,\tface,\wt\tface$, while at the remaining hypersurfaces, in particular at those only present in the \emph{extended} double space, $\tlb,\trb,\wt\tlb,\wt\trb$, it vanishes to infinite order.

  We now observe that under the map $\pi_{\cop\semi\tilde\semi,F}$, preimages of boundary hypersurfaces of ${}'X^2_\dscop$ are unions of boundary hypersurfaces of ${}'X^3_\dscop$ as follows:\footnote{We write `$A$: $B_1$, $B_2$, \ldots' for $\pi_{\dscop,F}^{-1}(A)=B_1\cup B_2\cup\cdots$.}
  \begin{alignat*}{8}
    &\lb:&\ & \ff_C,\bface_S; &\qquad& \ff:&\ & \fff,\ff_F; &\qquad& \rb:&\ & \ff_S,\bface_C; &\qquad& \tface:&\ & \tfface,\tface_F; \\
    &\wt\tface:&& \wt\tfface,\wt\tface_F; &\qquad& \sface:&& \sfface,\dface_C,\dface_S; &\qquad& \wt\sface:&& \wt\sfface,\wt\dface_C,\wt\dface_S; &\qquad& \dface:&& \dfface,\dface_F; \\
    &\wt\dface:&& \wt\dfface,\wt\dface_F; &\qquad& \tlb:&& \tbface_S,\tface_C; &\qquad& \trb:&& \tbface_C,\tface_S; &\qquad& \wt\tlb:&& \wt\tbface_S,\wt\tface_C; \\
    &\wt\trb:&& \wt\tbface_C,\wt\tface_S.
  \end{alignat*}
  The only boundary hypersurfaces of ${}'X^3_\dscop$ not contained in this list are $\bface_F$, $\tbface_F$, $\wt\tbface_F$; the pullback $\pi_{\dscop,F}^*K'_1$ has index set $\N_0$ at (i.e.\ is smooth down to) these. The corresponding statement for $\pi_{\dscop,C}$ is obtained from this by cyclically permuting indices according to $F\to S\to C\to F$, while the statement for $\pi_{\dscop,S}$ is obtained by interchanging $C$ and $F$, but leaving $S$ fixed.
  
  Therefore, the index sets of $K'_{1 2}:=\pi_{\dscop,F}^*K'_1\cdot\pi_{\dscop,S}^*K'_2$ at $\bface_C$, $\tbface_C$, $\wt\tbface_C$ are $\cE_{1,\rb}+\cE_{2,\lb}$, $\emptyset$, $\emptyset$, so the pushforward $K=(\pi_{\dscop,C})_*K'_{1 2}$ is well-defined when $\inf\Re(\cE_{1,\rb}+\cE_{2,\lb})>0$. The index set of $K$ at $\lb$ is then the extended union of the index sets of $K'_{1 2}$ at $\ff_F$ and $\bface_S$, which are $\cE_{1,\ff}+\cE_{2,\lb}$ and $\cE_{1,\lb}$. The arguments for $\ff,\rb,\tface$, and $\wt\tface$ are similar; in the latter case (which did not arise in~\S\ref{SR}), we note that the index sets of $K'_{1 2}$ at $\wt\tfface$ and $\wt\tface_C$ are $\cE_{1,\wt\tface}+\cE_{2,\wt\tface}$ and $\emptyset$, whose extended union is indeed $\cF_{\wt\tface}$ as stated.

  Finally, the index sets of $K$ at $\tlb,\trb,\wt\tlb,\wt\trb,\sface$ are trivial. Following the arguments after~\eqref{EqCCoIssue}, this implies that $K$ is the lift of a polyhomogeneous distribution on $X^2_\dscop$ to ${}'X^2_\dscop$, finishing the proof.
\end{proof}

\subsection{Construction of resolvents of semiclassical cone operators}
\label{SsCR}

Since the doubly semiclassical calculus is exclusively used as a tool (rather than as an interesting class of operators in its own right), we restrict attention to the following class of operators (cf.\ \eqref{EqCIntroOp}): fix $A\in\Diff_{\cop,\semi}^m(X)$, $m\geq 1$, and pick an $h$-dependent operator $A_{\bop,h}\in\CI([0,1)_h;\Diffb^m(X))$ such that
\[
  A - h^m x^{-m}A_\bop \in \Diff_{\cop,\semi}^{m-1}(X).
\]
Furthermore, let $\tilde\omega\in\C\setminus\{0\}$. Define $\tilde A\in\rho_\ff^{-m}\Psidsc^m(X)$ by
\begin{equation}
\label{EqCROp}
  \tilde A_{h,\tilde h} = \tilde h^m A_h-\tilde\omega
\end{equation}

To motivate the correct definition of full ellipticity for such operators, let us write
\begin{equation}
\label{EqCROpLocal}
  A_h = \sum_{k+|\alpha|\leq j\leq m} \Bigl(\frac{h}{x}\Bigr)^j a_{j k \alpha}(h,x,y)(x D_x)^k D_y^\alpha,\quad a_{j k \alpha}(h,x,y)\in\CI([0,1)_h\times[0,\eps)_x\times\R^{n-1}_y),
\end{equation}
in local coordinates on $[0,1)_h\times X$; we can then take
\[
  A_{\bop,h}=\sum_{k+|\alpha|\leq m} a_{m k\alpha}(h,x,y)(x D_x)^k D_y^\alpha.
\]
Observe that in the definition of $\tilde A_{h,\tilde h}$ in~\eqref{EqCROp}, all terms of $A_h$ with $j\leq m-1$ get multiplied by $\tilde h^{m-j}\cdot\tilde h^j$ with $m-j\geq 1$; the calculations starting with equation~\eqref{EqCEx1} (which include the constant term in the special case $m=0$) thus imply, as noted in Remark~\ref{RmkCCNondeg}, that the ellipticity of the dsc-principal symbol of $\rho_\ff^m\tilde A$ is equivalent to the semiclassical ellipticity of $\sum_{k+|\alpha|=m}a_{m k\alpha}(h,x,y)\xi^k\eta^\alpha-\tilde\omega$ for $(\xi,\eta)\in\R\times\R^{n-1}$, meaning the absolute value of this expression is bounded from below by $c_0(1+|\xi|+|\eta|)^m$, $c_0>0$, for all $(\xi,\eta)$. Since this is an open condition in the parameter $h$, it holds automatically for all sufficiently small $h>0$ provided it holds at $h=0$, where it can be expressed as the semiclassical ellipticity of $\sigmab_m(A_{\bop,0})-\tilde\omega$ on $\Tb^*X$.

Next, consider the normal operator $N_\ff(\rho_\ff^m\tilde A)$; in $h,\tilde h>0$, this is a smooth nonzero multiple of $N_{\pa X}(x^m\tilde A_{h,\tilde h})=\tilde h^m h^m N_{\pa X}(A_{\bop,h})$. \emph{Assuming} that the b-normal operator of $A$ is $h$-independent---equivalently, $N_{\pa X}(A_{\bop,h})$ is $h$-independent---this implies that, for some $0<f\in\CI([0,1)_h\times[0,1)_{\tilde h}\times\pa X)$ depending on the choice of $\rho_\ff$, the $\ff$-normal operator $f^{-1}N_\ff(\rho_\ff^m\tilde A)=N_{\pa X}(A_{\bop,h})\in\Diff_\bop({}^+N\pa X)$ is $h$-independent.

The normal operator $N_{\wt\tface}(\tilde A)$ can be computed by passing to the coordinates $(\tilde h,h,\tilde x=\frac{x}{h\tilde h},y)$ and restricting to $\tilde h=0$; since $\tilde h^m h^j x^{-j}=\tilde h^{m-j}\tilde x^{-j}$, only the leading order part $h^m x^{-m}A_\bop$ of $A$ and the constant term $\tilde\omega$ contribute:
\[
  N_{\wt\tface}(\tilde A) = \sum_{k+|\alpha|\leq m} \tilde x^{-m}a_{m k\alpha}(h,0,y)(\tilde x D_{\tilde x})^k D_y^\alpha - \tilde\omega.
\]
But we are assuming that $a_{m k\alpha}(h,0,y)$ is $h$-independent; hence
\begin{equation}
\label{EqCRNwttf}
  N_{\wt\tface}(\tilde A) = N_\tface(h^m x^{-m}A_{\bop,0}-\tilde\omega)
\end{equation}
is an $h$-independent (weighted) b-scattering operator on $\ol{{}^+N\pa X}$. We shall need to assume its invertibility as a map $H_{\bop,\scop}^{s,\alpha}(\ol{{}^+N\pa X})\to H_{\bop,\scop}^{s-m,\alpha-m}(\ol{{}^+N\pa X})$.

Finally, the invertibility of the $\tface$-normal operator of $\tilde A$---as in Theorem~\ref{ThmR} but also taking the semiclassical scattering behavior at $\pa_\infty\ol{{}^+N\pa X}$ into account (see the proof of Theorem~\ref{ThmCR} below for details)---requires a full ellipticity assumption. But the $\tface$-normal operator of $N_\tface(\tilde A)$ is equal to~\eqref{EqCRNwttf} in view of the relationship with $N_{\wt\tface}(\tilde A)$, see~\eqref{EqCCNtfRel}; and the semiclassical cone-scattering principal symbol of $N_\tface(\tilde A)$ is the restriction to $N^*(\Delta_{\cop\semi\tilde\semi}\cap\tface)$ of the dsc-principal symbol of $\tilde A$, the ellipticity condition for which was already discussed above.

In summary, the correct notion of ellipticity is the following.

\begin{definition}
\label{DefCREll}
  Let $m\geq 1$, $A\in\Diff_{\cop,\semi}^m(X)$, $\tilde\omega\in\C\setminus\{0\}$. Let $A_\bop\in\CI([0,1)_h;\Diffb^m(X))$ be such that $A=h^m x^{-m}A_\bop+A'$ with $A'\in\Diff_{\cop,\semi}^{m-1}(X)$. Define the operator
  \[
    \tilde A_{h,\tilde h} = \tilde h^m A_h-\tilde\omega.
  \]
  Then $\tilde A$ is \emph{fully elliptic at weight $\alpha\in\R$} if the following conditions hold.
  \begin{enumerate}
  \item\label{ItCREllSymb} We have $|\sigmab_m(A_{\bop,0})(z,\zeta)-\tilde\omega|\geq c_0(1+|\zeta|)^m$, $c_0>0$, for all $(z,\zeta)\in\Tb^*X$.
  \item\label{ItCREllFull} The operator $h^m x^{-m}A_{\bop,h}-\tilde\omega$ is fully elliptic at weight $\alpha$ in the sense of Definition~\ref{DefOpEll}.
  \end{enumerate}
\end{definition}

The second condition here entails the $h$-independence of the normal operator of $A_{\bop,h}$ as well as the invertibility of the model problem~\eqref{EqCRNwttf}.

\begin{thm}
\label{ThmCR}
  Suppose $\tilde A$ as in Definition~\usref{DefCREll} is fully elliptic at weight $\alpha\in\R$. Then there exists $c_0>0$ such that
  \[
    \tilde A_{h,\tilde h} \colon \Hb^{s,\alpha}(X) \to \Hb^{s-m,\alpha-m}(X),\quad s\in\R,
  \]
  is invertible for $0<h,\tilde h<c_0$. The inverse lies in the large dsc-calculus,
  \[
    \tilde A^{-1} \in \rho_\ff^m\Psidsc^{-m}(X) + \Psi_\dscop^{-\infty,\cE}(X),\quad
    \cE=(\check E_\lb(\alpha),\check E_\ff(\alpha)+m,\check E_\rb(\alpha)+m,\N_0,\N_0),
  \]
  where the index sets are defined by~\eqref{EqRRElbrb}--\eqref{EqRREff}.
\end{thm}
\begin{proof}
  The proof is similar to that of Theorem~\ref{ThmR}, except we now also need to invert the $\tface$-normal operator. The first step of the inversion of $\tilde A$ is the usual symbolic elliptic parametrix construction. Restricting henceforth to $h,\tilde h<c_0$ for some small $c_0>0$, where $\tilde A$ is dsc-elliptic as discussed above, we obtain
  \[
    B_0\in\rho_\ff^m\Psidsc^{-m}(X),\quad
    A B_0=I-R_0,\ R_0\in\rho_\dface^\infty\rho_{\wt\dface}^\infty\Psidsc^{-\infty}(X) = \Psi_\dscop^{-\infty,(\emptyset,\N_0,\emptyset,\N_0,\N_0)}(X).
  \]
  The next step solves away the error at the front face $\ff\cong[0,1)_{\tilde h}\times[0,1)_h\times\ff_\bop$ by employing b-normal operator analysis with smooth parametric dependence on $(\tilde h,h)$ and using the independence of $N_\ff(\tilde A)$ on $(h,\tilde h)$; this gives
  \begin{align*}
    &B_1\in\Psidsc^{-\infty,(E_\lb(\alpha),\N_0+m,E_\rb(\alpha)+m,\N_0,\N_0)}(X), \\
    &\qquad A(B_0+B_1)=I-R_1,\ R_1\in\Psidsc^{-\infty(E_\lb(\alpha)-m,\N,E_\rb(\alpha)+m,\N_0,\N_0)}(X).
  \end{align*}
  As in~\eqref{EqRRB2}, we solve away the error at $\lb$ to infinite order with an operator
  \[
    B_2\in\Psidsc^{-\infty,(E_\lb(\alpha)\extcup\hat E_\lb(\alpha),\N+m,\emptyset,\N_0,\N_0)}(X);
  \]
  subsequently, the error at $\ff$ (where the index set of the error term is $\N$) is removed using a Neumann series argument as in~\eqref{EqRRB3}. The upshot is that we can construct a right parametrix $G$ as in~\eqref{EqRRG3}, and similarly a left parametrix $G'$, with
  \begin{align}
    &G,G' \in \rho_\ff^m\Psidsc^{-m}(X) + \Psidsc^{-\infty,\cE}(X), \nonumber\\
  \label{EqCRG}
    &\qquad A G=I-R,\ R\in\Psidsc^{-\infty,(\emptyset,\emptyset,\hat E_\rb(\alpha)+m,\N_0,\N_0)}(X), \\
  \label{EqCRGLeft}
    &\qquad G'A=I-R',\ R'\in\Psidsc^{-\infty,(\hat E_\lb(\alpha),\emptyset,\emptyset,\N_0,\N_0)}(X).
  \end{align}

  We next improve the error term at $\wt\tface$ using the normal operator $N_{\wt\tface}$ from~\eqref{EqCCNtf2}; see also~\eqref{EqCEx5Ntf2}. By~\eqref{EqRRNtfAInv}, and recalling from~\eqref{EqCRNwttf} that the $h$-dependence is trivial, we have
  \[
    N_{\wt\tface}(A)^{-1} \in \rho_{\tface,\bop}^m\Psi_{\bop,\scop}^{-m}(\ol{{}^+N\pa X})+\Psi_{\bop,\scop}^{-\infty,(\cE_\lb,\cE_\ff,\cE_\rb)}(\ol{{}^+N\pa X})\quad \text{(independent of $h$)}.
  \]
  Take thus $B_{\tilde\tface}\in\rho_\ff^m\Psidsc^{-m}(X)+\Psidsc^{-\infty,\cE}(X)$ to be an operator with $N_{\wt\tface}(B_{\tilde\tface})=N_{\wt\tface}(A)^{-1}$; at $\tface$, whose intersection with $\wt\tface$ is defined by $h=0$ within $\wt\tface$, we can, indeed, take as the index set for $B_{\tilde\tface}$ the set $\cE_\tface=\N_0$ in view of the smooth $h$-dependence in~\eqref{EqCRNwttf}. Then
  \begin{align*}
    &G_4 := G+B_{\tilde\tface} R\in\rho_\ff^m\Psidsc^{-m}(X)+\Psidsc^{-\infty,\cE}(X) \\
    &\qquad \Rightarrow A G_4 = I - R_4,\ R_4\in\Psidsc^{-\infty,(\emptyset,\emptyset,\check E_\rb(\alpha)+m,\N_0,\N)}(X).
  \end{align*}
  We solve away the error at $\wt\tface$ to infinite order by a Neumann series argument; thus, taking $B'_5\sim\sum_{j=0}^\infty R_4^j$, we obtain
  \[
    \tilde G:=G_4 B'_5,\quad
    A\tilde G=I-\tilde R,\ \tilde R\in\Psidsc^{-\infty,(\emptyset,\emptyset,\check E_\rb(\alpha)+m,\N_0,\emptyset)}(X).
  \]
  Analogous arguments starting with~\eqref{EqCRGLeft} produce
  \[
    \tilde G'\in\rho_\ff^m\Psidsc^{-m}(X)+\Psidsc^{-\infty,\cE}(X), \quad
    \tilde G' A=I-\tilde R',\ \tilde R'\in\Psidsc^{-\infty,(\check E_\lb(\alpha),\emptyset,\emptyset,\N_0,\emptyset)}(X).
  \]

  Finally, we remove the error at $\tface$. We claim that there exists $\tilde h_0>0$ such that $N_\tface(A)$---which in $\tilde h>0$ is an $\tilde h$-dependent family of b-scattering operators on $\ol{{}^+N\pa X}$---is invertible for $0<\tilde h<\tilde h_0$, with
  \begin{equation}
  \label{EqCRNtfInv}
    N_\tface(A)^{-1} \in \rho_{\mathbf{ff}}^m\Psi_{\cop,\scop,\tilde\semi}^{-m}(\ol{{}^+N\pa X}) + \Psi_{\cop,\scop,\tilde\semi}^{-\infty,(\cE_\lb,\cE_\ff,\cE_\rb,\cE_\tface)}(\ol{{}^+N\pa X}).
  \end{equation}
  This \emph{almost} follows from Theorem~\ref{ThmR} upon restricting $N_\tface(A)$ to a (weighted) element of $\Psi_{\cop\tilde\semi}({}^+N\pa X)$; however, ${}^+N\pa X$ is not compact, hence the theorem does not apply. However, $N_\tface(A)$ has a very simple structure at the scattering end, i.e.\ near $\mathbf{ff}_\scop$: it is a semiclassically elliptic scattering operator. Thus, symbolic arguments produce a parametrix of $N_\tface(A)$ of the form
  \[
    B_{\tface,0}\in\rho_{\mathbf{ff}}^m\Psi_{\cop,\scop,\tilde\semi}^{-m}(\ol{{}^+N\pa X}), \quad
    N_\tface(A)B_{\tface,0}=I-R_{\tface,0},\ R_{\tface,0}\in\rho_{\wt{\mathbf{df}}}^\infty\rho_{\mathbf{ff}_\scop}^\infty\Psi_{\cop,\scop,\tilde\semi}^{-\infty}(\ol{{}^+N\pa X}),
  \]
  i.e.\ the error $R_\tface$ is trivial at the scattering end. Thus,~\eqref{EqCRNtfInv} follows by subsequently improving the parametrix $B_{\tface,0}$ using the parameterized b-arguments and the $\wt{\mathbf{tf}}$-normal operator analysis analogous to the $\tface$-normal operator analysis in the proof of Theorem~\ref{ThmR}; note that the infinite order vanishing of error terms at the scattering end is preserved by these steps, as they are local near $\mathbf{ff}$, $\mathbf{lb}$, $\mathbf{rb}$, and $\wt{\mathbf{tf}}$.

  Thus, if $B_\tface\in\rho_\ff^m\Psidsc^{-m}(X)+\Psidsc^{-\infty,\cE}(X)$ has $\tface$-normal operator $N_\tface(B_\tface)=N_\tface(A)^{-1}$ (in $\tilde h<\tilde h_0$), then
  \begin{align*}
    &G_5 := \tilde G+B_\tface\tilde R \in \rho_\ff^m\Psidsc^{-m}(X) + \Psidsc^{-\infty,\cE}(X) \\
    &\qquad \Rightarrow A G_5=I-R_5,\ R_5\in\Psidsc^{-\infty,(\emptyset,\emptyset,\check E_\rb(\alpha)+m,\N,\infty)}(X).
  \end{align*}
  Multiplying $G_5$ from the right by an operator given as an asymptotic sum $\sum_{j=0}^\infty R_5^j$ gives $G_6\in\rho_\ff^m\Psidsc^{-m}(X) + \Psidsc^{-\infty,\cE}(X)$ with
  \[
    A G_6 = I-R_6,\quad R_6\in\Psidsc^{-\infty,(\emptyset,\emptyset,\check E_\rb(\alpha)+m,\emptyset,\emptyset)}(X).
  \]
  Now $R_6\in\cA_\phg^{(\emptyset,\check E_\rb(\alpha)+m,\emptyset,\emptyset)}([0,1)_{\tilde h}\times[0,1)_h\times X^2)$, with the index sets referring to $[0,1)^2\times\pa X\times X$, $[0,1)^2\times X\times\pa X$, $\tilde h^{-1}(0)$, $h^{-1}(0)$, in this order. Thus, for small $h,\tilde h$, $I-R_6$ is invertible on $\Hb^{s-m,\alpha-m}(X)$ by a convergent Neumann series, with inverse of the same form.
  
  Setting $G=G_6(I-R_6)^{-1}$ proves the existence of $G\in\rho_\ff^m\Psidsc^{-m}(X)+\Psidsc^{-\infty,\cE}(X)$ with $A G=I$. Similar arguments produce a left inverse $G'$ lying in the same operator class, and standard arguments give $G=G'$, finishing the proof of the theorem.
\end{proof}

\section{Structure of complex powers}
\label{SP}

Using the resolvent constructed in~\S\ref{SsCR}, we shall now deduce the precise structure of complex powers of semiclassical cone operators. Seeley's original approach to the definition of complex powers of elliptic operators on closed manifolds uses a large parameter calculus; we explain in~\S\ref{SsPP} in this simple setting how one can use semiclassical operators instead. In~\S\ref{SsPC}, we prove a slight extension of Loya's result \cite{LoyaConicPower} about the structure of complex powers of fully elliptic cone operators using the resolvent rather than the heat kernel \cite{LoyaConicHeat}; this uses a special case of the results of~\S\ref{SsRR}. In~\S\ref{SsPh} finally, we generalize this to the semiclassical setting.

\subsection{Seeley's approach revisited}
\label{SsPP}

Let $X$ be a compact manifold without boundary. Let $A\in\Diff^m(X)$, $m>0$, be elliptic, and $\sigma_m(A)+1\neq 0$ on $T^*X$ (thus $\sigma_m(A)$ misses $(-\infty,0)$). Assume that $A$ is invertible. Then by \cite{SeeleyPowers}, see also \cite[\S\S10--11]{ShubinSpectralTheory}, we can define complex powers of $A$ by
\begin{equation}
\label{EqPP}
  A_w = \frac{i}{2\pi}\int_{\gamma_\eps} \lambda^w(A-\lambda)^{-1}\,\dd\lambda \in \Psi^w(X),
\end{equation}
initially for $\Re w<0$ and then for all $w\in\C$ by analytic continuation. Here, we integrate over the contour
\begin{equation}
\label{EqPPContour}
\begin{split}
  &\gamma_\eps=\gamma_3\circ\gamma_2\circ\gamma_1, \\
  &\gamma_1(s)=s+i\eps,\ s\in(-\infty,0],
  \quad \gamma_2(\theta)=\eps e^{-i\theta},\ \theta\in\left[-\tfrac{\pi}{2},\tfrac{\pi}{2}\right],
  \quad \gamma_3(s)=-s-i\eps,\ s\in[0,\infty),
\end{split}
\end{equation}
where $\eps>0$ is so small that $(A-\lambda)^{-1}$ exists when $\lambda$ is $\eps$-close to $(-\infty,0]$; and we define $\lambda^w=\exp(w\log\lambda)$ using the logarithm with branch cut along $(-\infty,0]$. See Figure~\ref{FigPPContour}.

\begin{figure}[!ht]
\centering
\includegraphics{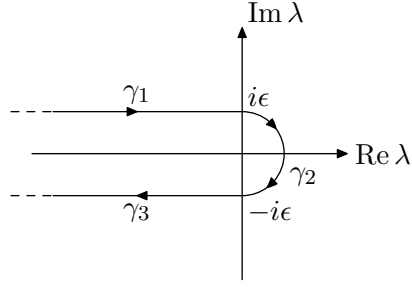}
\caption{The integration contour $\gamma_\eps$ defined in~\eqref{EqPPContour}.}
\label{FigPPContour}
\end{figure}

We rewrite the integrand for large $|\lambda|$ in terms of the semiclassical resolvent
\[
  R(h,\omega)=(h^m A-\omega)^{-1}\in\Psih^{-m}(X).
\]
This exists, and depends smoothly on $\omega$ near $-1$, when $h>0$ is small enough. Moreover, $R(h,\omega)$ can be computed modulo $h^\infty\Psi^{-\infty}(X)=h^\infty\CI(X^2)$ in local coordinates by a symbolic computation. More precisely, given $k\in\N$, there exists $N\in\N$ such that we can write $R(h,\omega)$ in a local coordinate chart as
\[
  R(h,\omega)=\Oph(r_N(h,\omega))+E^{(N)}(h,\omega),
\]
where $E^{(N)}\in h^k\cC^k(X^2)$ and $r^{(N)}(h,\omega)=\sum_{j=0}^{N-1}b_j(h,\omega)$, $b_j\in h^j\Shom^{-m-j}(T^*X\setminus o)$ (with $\CI$ dependence on $h,\omega$ omitted from the notation); this uses that Schwartz kernels of elements of $\Oph(h^j S^{-m-j})$ become arbitrarily smooth on $[0,1)_h\times X^2$ as $j\to\infty$. Rewriting the integrand in~\eqref{EqPP} as
\[
  (A-\lambda)^{-1} = |\lambda|^{-1}R(h,\omega),\quad h=|\lambda|^{-1/m},\ \omega=\lambda/|\lambda|,
\]
we can replace $R^{(N)}$ in local coordinates on $X^2$ by $\Op_h(r^{(N)})$, modulo an error
\begin{equation}
\label{EqPPErr}
  \frac{i}{2\pi}\int_{\gamma_\eps}\lambda^w|\lambda|^{-1} E^{(N)}(|\lambda|^{-1/m},\omega,x,x')\,\dd\lambda \in \cC^k(X^2).
\end{equation}

On the other hand, since the symbol of an elliptic parametrix of an elliptic semiclassical operator in local coordinates is unique modulo $h^\infty S^{-\infty}$, the expansion of $r^{(N)}$ \emph{must agree} with the expansion into large parameter symbols upon making the appropriate identifications; that is, we have
\begin{equation}
\label{EqPPSymb}
  |\lambda|^{-1}\Op_{|\lambda|^{-1/m}}\bigl(b_j(|\lambda|^{-1/m},\lambda/|\lambda|)\bigr) = \Op(b'_j),\quad b'_j(\lambda)\in\Shom^{-m-j,m}((T^*X\times\Lambda)\setminus o),
\end{equation}
where $\Lambda=\{\theta\omega\colon\theta\in[0,\infty),\ \omega\in\Omega\}$, $\Omega$ a small neighborhood of $-1$, and the $b'_j$ are the terms of the symbolic parametrix construction of $(A-\lambda)^{-1}$ in the large parameter calculus, see~\cite[\S11.1]{ShubinSpectralTheory}. (Note that we use the \emph{standard} quantization on the right!) Here, we recall that $\Shom^{-m-j,m}((T^*X\times\Lambda)\setminus o)$ (with $o=\{(z,\zeta,\lambda)\colon\zeta\in T^*_z X,\ \zeta=\lambda=0\}$ the zero section) consists of all $a\in\CI((T^*X\times\Lambda)\setminus o)$ such that $a(z,\theta\zeta,\theta^m\lambda)=\theta^{-m-j} a(z,\zeta,\lambda)$ for all $\theta>0$. In particular, $\int\lambda^w b'_j(\lambda)\,\dd\lambda\in\Shom^{m w-j}(T^*X\setminus o)$ and so
\[
  \frac{i}{2\pi}\int_{\gamma_\eps}\lambda^w \Op(b'_j)\,\dd\lambda \in \Psi^{m w-j}(X).
\]
Together with~\eqref{EqPPErr}, and recalling that $k$ is arbitrary, this proves~\eqref{EqPP}.

\subsection{Complex powers of cone operators}
\label{SsPC}

From now on, $X$ again denotes a compact $n$-dimensional manifold with connected embedded boundary $\pa X$, and $x\in\CI(X)$ denotes a defining function of $\pa X$. For a (non-semiclassical) cone operator $A=x^{-m}A_\bop\in x^{-m}\Diffb^m(X)$ satisfying a suitable ellipticity condition, we wish to describe the complex powers $A_w$; this is a warm-up for~\S\ref{SsPh}.

\begin{thm}
\label{ThmPC}
  Suppose that $h^m A+1=h^m x^{-m}A_\bop+1$ is fully elliptic with weight $\alpha\in\R$, and that $A-\lambda\colon\Hb^{s,\alpha}(X)\to\Hb^{s-m,\alpha-m}(X)$ is invertible for $\lambda\in(-\infty,0]$.\footnote{The full ellipticity assumption implies invertibility for sufficiently large negative $\lambda$ in view of Theorem~\ref{ThmR} with $\omega=-\lambda/|\lambda|$ and $h=|\lambda|^{-1/m}$, hence this merely serves to exclude the possibility of spectrum in a large but finite interval $[-C,0]$.} Let
  \begin{equation}
  \label{EqPC}
    A_w := \frac{i}{2\pi}\int_\gamma\lambda^w(A-\lambda)^{-1}\,\dd\lambda,\quad w\in\C,
  \end{equation}
  where $\gamma_\eps$ is as in~\eqref{EqPPContour}, with $\eps>0$ chosen so that $(A-\lambda)^{-1}$ exists when $\lambda$ is $\eps$-close to $(-\infty,0]$. Then $A_w$, $\Re w<0$, is well-defined as an operator on $\Hb^{s,\alpha-m}(X)$. Its Schwartz kernel admits a holomorphic extension to $w\in\C$, and we have
  \[
    A_w \in x^{-m w}\Psib^{m w}(X) + \Psib^{-\infty,\cE_\bop(w)}(X), \quad
    \cE_\bop(w) = \bigl(\cE_\lb,\cE_\ff \extcup{}(\N_0-m w),\cE_\rb \bigr),
  \]
  where $\cE_\lb,\cE_\ff,\cE_\rb$ are defined by~\eqref{EqRRElbrb}, \eqref{EqRREff}, \eqref{EqRRcE}.
\end{thm}

Loya \cite{LoyaConicPower} proves this when $A$ has no spectrum inside a conic sector whose interior includes the closed left half plane.

\begin{proof}[Proof of Theorem~\usref{ThmPC}]
  It suffices to consider the case $\Re w<0$; the extension to the complex plane is accomplished by setting $A_w:=A^k A_{w-k}$ when $w\in\C$, $k\in\N_0$, and $\Re(w-k)<0$. By assumption, $(A-\lambda)^{-1}$ exists for all $\lambda\in\gamma_\eps$, and we have $(A-\lambda)^{-1}\in x^m\Psib^{-m}(X)+\Psib^{-\infty,\cE_\bop}(X)$ by \cite{MelroseAPS} or as the special case $h=1$ of Theorem~\ref{ThmR} (using the invertibility assumption on $A-\lambda$). For large $|\lambda|$, we rewrite~\eqref{EqPC} semiclassically in terms of $A_{h,\omega}=h^m A-\omega$:
  \begin{equation}
  \label{EqPC2}
    A_w = \frac{i}{2\pi} \int_{\gamma_\eps} h^{-m w}\omega^w A_{h,\omega}^{-1}\cdot |\lambda|^{-1}\,\dd\lambda,\quad h=|\lambda|^{-1/m},\ \omega=\frac{\lambda}{|\lambda|}.
  \end{equation}
  Confining $\omega$ to a small neighborhood $\Omega$ of $-1$, we have
  \[
    A_{h,\omega}^{-1} = R_{h,\omega} + E_{h,\omega},\quad
    R_{h,\omega}\in\rho_\ff^m\Psi_{\cop\semi}^{-m}(X),\ E_{h,\omega}\in\Psi_{\cop\semi}^{-\infty,\cE}(X)
  \]
  with $\cE=\left(\check E_\lb(\alpha),\check E_\ff(\alpha)+m,\check E_\rb(\alpha)+m,\N_0\right)$ by Theorem~\ref{ThmR}, with smooth dependence on $\omega\in\Omega$.
  
  For convenience, we may deform the contour $\gamma_\eps$ to a new contour $\gamma_\eps'$ which is exactly radial, at a small angle $\pm\eps$ with $(-\infty,0]$, for $|\lambda|\geq 1$. Upon inserting a smooth cutoff $\chi\in\CIc((1,\infty))$, $\chi(|\lambda|)\equiv 1$ for $|\lambda|\geq 2$, the integral~\eqref{EqPC2} is a sum of two integrals for $\omega=-e^{\pm i\eps}$, each of which is a push-forward along the map $X^2_{\cop\semi}\to X^2_\bop$ which is the lift of the projection $[0,1)_h\times X^2_\bop\to X^2_\bop$. Recall that $E_{h,\omega}$ vanishes to infinite order at $\sface\cup\dface$, hence is polyhomogeneous on the simpler space $X^2_{\cop\semi,\infty}$, see~\eqref{EqRCPsdoRes}; the stretched projection $X^2_{\cop\semi,\infty}\to X^2_\bop$ is a b-fibration. Since $\ff$ and $\tface$ both map to $\ff_\bop\subset X^2_\bop$, we obtain
  \begin{equation}
  \label{EqPCErr}
    \int_{\gamma_\eps} \chi(|\lambda|) h^{-m w}\omega^w E_{h,\omega}\cdot \omega\frac{\dd h}{h} \in \Psib^{-\infty,\cE_\bop(w)}(X)
  \end{equation}
  since the index set of the integrand is $\cE_\ff$ at $\ff$ and $\N_0-m w$ (from $h^{-m w}$) at $\tface$.

  It thus remains to analyze the contribution of $R_{h,\omega}$, which, in view of~\eqref{EqPCErr}, we may in addition localize arbitrarily closely to the preimage $\Delta_{\cop\semi}\cup\dface$ of the diagonal under the blow-down of $\dface$. But then $R_{h,\omega}$ can be expressed in local coordinates near the diagonal $\Delta_{\cop\semi,\infty}$ in $X^2_{\cop\semi,\infty}$, see~\eqref{EqRCPsdoResDiag}. Its full symbol, modulo $\rho_\dface^\infty S^{-\infty}(N^*\Delta_{\cop\semi})$, is uniquely determined by the requirement that $R_{h,\omega}$ be an elliptic parametrix of $A_{h,\omega}$ in $X^\circ$. As in the discussion leading to~\eqref{EqPPSymb}, we can thus compute $|\lambda|^{-1}R_{h,\omega}$ as a quantization of a sum of $N$ large parameter symbols, modulo error terms
  \begin{equation}
  \label{EqPCErr2}
    E^{(N)}_{h,\omega}\in|\lambda|^{-1}\rho_\ff^m\rho_\dface^N\Psi_{\cop\semi}^{-m-N}(X)
  \end{equation}
  with Schwartz kernels localized near $\Delta_{\cop\semi}$. Any fixed continuous $|\lambda|^{-1}\rho_\ff^m\rho_\dface^\infty\Psi_{\cop\semi}^{-\infty}(X)$-semi\-norm is continuous on the space in~\eqref{EqPCErr2} when $N$ is large enough; hence, by the arguments leading to~\eqref{EqPCErr}, $E^{(N)}_{h,\omega}$ contributes to~\eqref{EqPC2} by a Schwartz kernel $E^{(N,w)}_{h,\omega}$ in such a way that any fixed $\Psi_\bop^{-\infty,(\emptyset,(m+\N_0)\extcup{}(\N_0-m w),\emptyset)}(X)$-seminorm becomes finite for $N$ large enough.

  It therefore remains to describe the contribution from the quantization of a \emph{single} large parameter b-symbol arising in the symbolic parametrix construction for $(A-\lambda)^{-1}$. Writing $a_m(z,\zeta)=\sigmab_m(A_\bop)(z,\zeta)$, $(z,\zeta)\in\Tb^*X$, the first term is $b_{-m}(z,\zeta,\lambda)=x^m(a_m-x^m\lambda)^{-1}$; thus, $b'_{-m-j}(z,\zeta,\lambda'):=b_{-m-j}(z,\zeta,x^{-m}\lambda')$ (with $j=0$ for now) has the property that $x^{-m}b'_{-m-j}(z,\zeta,\lambda')$ is homogeneous of degree $(-m-j,m)$ in $(\zeta,\lambda')$ (that is, $b'_{-m-j}(\theta z,\theta^m\zeta,\lambda')=\theta^{-m-j}b'_{-m}(z,\zeta,\lambda')$ for $\theta>0$) and smooth down to $x=0$. Subsequent terms in the symbol expansion of $(A-\lambda)^{-1}$, with $j=1,2,\ldots$, have the same property. Similarly to~\cite[\S11.12]{ShubinSpectralTheory}, it follows that
  \begin{align*}
    b^{(w)}_{m w-j}(z,\zeta)&:=\frac{i}{2\pi}\int_{\gamma_\eps}\lambda^w b_{-m-j}(z,\zeta,\lambda)\,\dd\lambda \\
      &=\frac{i}{2\pi} \int x^{-m w}(\lambda')^w\cdot x^{-m}b'_{-m-j}(z,\zeta,\lambda')\,\dd\lambda'
  \end{align*}
  is a homogeneous symbol of degree $m w-j$, and smooth as such down to $x=0$ when multiplied by $x^{m w}$. If $\psi\in\CIc((-1,1))$ is identically $1$ near $0$, serving to localize to a small neighborhood of the b-diagonal, the quantization of $b^{(w)}_{m w-j}$ is therefore
  \begin{align*}
    \Op_\bop(b_{m w-j}^{(w)}) &= (2\pi)^{-n}\psi\Bigl(\frac{x-x'}{x'}\Bigr)\psi(|y-y'|) \\ 
     &\quad\qquad \times \iint \left(\frac{x}{x'}\right)^{i\sigma}e^{i(y-y')\eta}b_{m w-j}^{(w)}(x,y,\sigma,\eta)\,\dd\sigma\,\dd\eta \in x^{m w}\Psib^{m w-j}(X).
  \end{align*}
  This completes the proof.
\end{proof}

\subsection{Complex powers of semiclassical cone operators}
\label{SsPh}

Let $X$ be as in~\S\ref{SsPC}.
\[
  A_h=h^m x^{-m}A_\bop-\omega_0,\quad A_\bop\in\Diffb^m(X),
\]
with $0\neq\omega_0\in\C$ fixed. The following is the main result of this paper:

\begin{thm}
\label{ThmPh}
  Let $m\geq 1$, $A\in\Diff_{\cop,\semi}^m(X)$. Suppose that $A_h-\tilde\lambda$ is fully elliptic at weight $\alpha$ for all $\tilde\lambda\in(-\infty,0]$ in the sense of Definition~\usref{DefOpEll}, and suppose that $\tilde h^m A_h+1$ is fully elliptic at weight $\alpha$ in the sense of Definition~\usref{DefCREll}. Let
  \begin{equation}
  \label{EqPh}
    (A_h)_w := \frac{i}{2\pi}\int_{\gamma_\eps} \tilde\lambda^w (A_h-\tilde\lambda)^{-1}\,\dd\tilde\lambda,\quad 0<h<h_0,
  \end{equation}
  with $\gamma_\eps$ as in~\eqref{EqPPContour}, where $\eps>0$ is so small that $(A_h-\tilde\lambda)^{-1}$ exists for all $\tilde\lambda\in\gamma_\eps$ and all $0<h<h_0$ for some small $h_0>0$.\footnote{Since full ellipticity of $A_h-\tilde\lambda$ is an open condition in $\tilde\lambda\in\C$, there exists $\eps>0$ small guaranteeing the invertibility of $A_h-\tilde\lambda$ for $\tilde\lambda\in\gamma_\eps$ with $|\tilde\lambda|$ bounded by any fixed constant. For large (independent of $\eps\in(0,1]$) $\tilde\lambda\in\gamma_\eps$ on the other hand, invertibility follows from Theorem~\ref{ThmCR}.} Then $(A_h)_w$, $\Re w<0$, is well-defined as an operator on $\Hb^{s,\alpha-m}(X)$. Its Schwartz kernel admits a holomorphic extension to $w\in\C$, and we have
  \[
    (A_h)_w \in \bigl(\tfrac{x}{x+h}\bigr)^{-m w}\Psi_{\cop\semi}^{m w}(X) + \Psi_{\cop\semi}^{-\infty,\cE(w)}(X),\quad
    \cE(w)=(\cE_\lb,\cE_\ff\extcup{}(\N_0-m w),\cE_\rb,\N_0),
  \]
  where the index sets are defined by~\eqref{EqRRElbrb}, \eqref{EqRREff}, \eqref{EqRRcE}.
\end{thm}
\begin{proof}
  The proof is parallel to that of Theorem~\ref{ThmPC}. Thus, for finite $\tilde\lambda\in\gamma_\eps$, Theorem~\ref{ThmR} implies $(A_h-\tilde\lambda)^{-1}\in\rho_\ff^m\Psi_{\cop\semi}^{-m}(X)+\Psi_{\cop\semi}^{-\infty,\cE}(X)$, while for large $|\tilde\lambda|$, we write the resolvent semiclassically in terms of the operator
  \[
    A_{h,\tilde h,\tilde\omega} = \tilde h^m A_h-\tilde\omega,\quad \tilde h=|\tilde\lambda|^{-1/m},\ \tilde\omega=\frac{\tilde\lambda}{|\tilde\lambda|},
  \]
  with $\tilde\omega$ lying in a small neighborhood $\tilde\Omega$ of $-1$. Note that this operator is fully elliptic at weight $\alpha$ in the sense of Definition~\ref{DefCREll} for $\tilde\omega$ sufficiently close to $-1$, since full ellipticity is an open condition. Hence, by Theorem~\ref{ThmCR},
  \[
    A_{h,\tilde h,\tilde\omega}^{-1} = R_{h,\tilde h,\tilde\omega}+E_{h,\tilde h,\tilde\omega},\quad R_{h,\tilde h,\tilde\omega}\in\rho_\ff^m\Psidsc^{-m}(X),\ E_{h,\tilde h,\tilde\omega}\in\Psidsc^{-\infty,\cE}(X),
  \]
  with smooth dependence on $\tilde\omega$.
  
  Since the Schwartz kernel of $E_{h,\tilde h,\tilde\omega}$ is the lift from $X^2_{\dscop,\infty}:=[[0,1)_{\tilde h}\times X^2_{\cop\semi,\infty};\{0\}\times\ff_0]$ (see~\eqref{EqRCPsdoRes}) of a polyhomogeneous distribution with trivial index set at the lifts of $h=0$ and $\tilde h=0$, the contribution of $E_{h,\tilde h,\tilde\omega}$ to $(A_h)_w$ can be described by the push-forward theorem similarly to~\eqref{EqPCErr}; note that the map $X^2_{\dscop,\infty}\to X^2_{\cop\semi,\infty}$ is a b-fibration. Thus, with notation as in~\eqref{EqPCErr},
  \[
    \int_{\gamma_\eps} \chi(|\tilde\lambda|)\tilde h^{-m w}\tilde\omega^w E_{h,\tilde h,\tilde\omega}\cdot \tilde\omega\frac{\dd\tilde h}{\tilde h} \in \Psi_{\cop\semi}^{-\infty,\cE(w)}(X),
  \]
  recalling that Schwartz kernels of elements of this space can be described equivalently as living on $X^2_{\cop\semi}$ or $X^2_{\cop\semi,\infty}$, as discussed around~\eqref{EqRCPsdoRes}.

  $R_{h,\tilde h,\tilde\omega}$ is given, modulo error terms in $\rho_\ff^m\rho_\dface^N\rho_{\wt\dface}^N\Psidsc^{-m-N}(X)$, as the quantization of a symbolic expansion to order $N$. As $N\to\infty$, such error terms contribute elements of $\Psi_{\cop\semi}^{-\infty,(\emptyset,(m+\N_0)\extcup{}(\N_0-m w),\emptyset,\N_0)}(X)$ to $(A_h)_w$, in the sense explained after~\eqref{EqPCErr2}.
  
  It thus suffices to consider the contribution to $(A_h)_w$ of the quantization of a single term in the large parameter ($\tilde\lambda$) semiclassical cone symbol expansion of $(A_h-\tilde\lambda)^{-1}$; we can localize these quantizations in an arbitrarily small neighborhood of $\Delta_{\cop\semi}$. Concretely, in the interior $\ff^\circ\subset X^2_{\cop\semi}$, i.e.\ for $h$ bounded away from $0$, the large parameter expansion is the same (with parametric dependence on $h$) as that discussed in the proof of Theorem~\ref{ThmPC}; the quantization of the $j$-th term contributes an element of $\CI((0,1)_h;x^{-m w}\Psib^{m w-j}(X))$ to $(A_h)_w$.
  
  Next, to obtain a uniform description down to $h$, we first work near $\ff\cap\tface\cap\Delta_{\cop\semi}$. There, we rescale the local boundary defining function $x'$ of $\ff$ to $\hat x'=\frac{x'}{h}$ as in~\eqref{EqRCEx2}, and in the rescaled coordinates, $A_h-\tilde\lambda$ is a large parameter cone operator, non-degenerately down to $h=0$ (a local defining function of $\tface$). After this rescaling, the large parameter symbolic expansion for $(A_h-\tilde\lambda)^{-1}$ is thus again of the form discussed in the proof of Theorem~\ref{ThmPC}, and the quantization of the $j$-th term contributes an element of $\rho_\ff^{-m w}\Psi_{\cop\semi}^{m w-j}(X)$ (with infinite order of vanishing at $\sface\cup\dface$, as we are localizing near $\ff\cap\tface$).

  It remains to analyze $A_w$ near $\tface\cap\dface\cap\Delta_{\cop\semi}$, cf.\ \eqref{EqRCEx3}. But there, the expansion of $(A_h-\tilde\lambda)^{-1}$ is into large parameter ($\tilde\lambda$) semiclassical (in $\hat h=h/x'$) symbols, i.e.\ successive terms in symbolic expansions gain powers of $\hat h$ as well as reduce in symbolic order in the joint fiber variable $(\zeta,\tilde\lambda)\in(N^*\Delta_{\cop\semi}\times\tilde\Lambda)\setminus o$, where $\tilde\Lambda=\{\theta\tilde\omega\colon\theta\in[0,\infty),\ \tilde\omega\in\tilde\Omega\}$. Thus, the $j$-th term of this expansion contributes a term $\rho_\dface^j\Psi_{\cop\semi}^{m w-j}(X)$ (with Schwartz kernel vanishing near $\lb\cup\ff\cup\rb$). The proof is complete.
\end{proof}

\section{Application to semiclassical conic Laplacians}
\label{SL}

We now consider a connected $n$-dimensional compact conic manifold, $n\geq 3$. For notational simplicity, we assume that there is only one cone point.

Upon blowing up the cone point, we arrive at the following setup: $X$ is a compact, $n$-dimensional manifold whose boundary $\pa X\neq\emptyset$ is a connected embedded hypersurface, and $g$ is a smooth Riemannian metric over the interior $X^\circ$. We assume that there exists a boundary defining function $x\in\CI(X)$ such that in a collar neighborhood $[0,x_0)_x\times\pa X$ of $\pa X$, the metric $g$ takes the form
\begin{equation}
\label{EqLMet}
  g = \dd x^2 + x^2 k(x,y,\dd y),
\end{equation}
where $k$ is a smooth metric on $\pa X$ with smooth parametric dependence on $x\in[0,x_0)$. Thus, using local coordinates on $\pa X$,
\begin{equation}
\label{EqLOp}
  \Delta_g = D_x^2 - i\bigl(n-1+\half x\pa_x\log|\det k|\bigr)x^{-1} D_x + x^{-2}\Delta_k.
\end{equation}

In Theorem~\ref{ThmL}, we describe the structure of complex powers of $h^2\Delta_g+1$; in Theorem~\ref{ThmLF}, we show that these powers are equal to those defined by means of the functional calculus for $\Delta_g$ when acting between suitable weighted spaces. Theorem~\ref{ThmLProp} is an application to semiclassical propagation estimates at conic points.

\subsection{Complex powers and domains}

To begin, we define the index sets for the Schwartz kernel of the relevant resolvent: denote by $0=\lambda_0^2<\lambda_1^2\leq\cdots$ the eigenvalues of $\Delta_{k(0)}$, and put
\[
  \nu_{j\pm}=i\left(\frac{n-2}{2}\pm\sqrt{\left(\frac{n-2}{2}\right)^2+\lambda_j^2}\,\right).
\]
Then, by a simple calculation using~\eqref{EqLOp}, the boundary spectrum (see~\eqref{EqRRSpecb}) is given by
\begin{equation}
\label{EqLSpecb}
  \specbfull(x^2\Delta_g)=\bigcup_{j,\pm}\{(\nu_{j\pm},0)\}.
\end{equation}
We fix a weight\footnote{The degeneration of the spectral gap for $n=2$ is the reason why we assume $n\geq 3$ here. Indeed, part~\eqref{ItOpEll3} in Definition~\ref{DefOpEll} of full ellipticity cannot be satisfied for $n=2$.}
\begin{equation}
\label{EqLWeight}
  \alpha \in (-n+2,0);
\end{equation}
these are the only choices of weights for which $h^2\Delta_g+1$ is fully elliptic at weight $\alpha$, as we will see in the proof of Proposition~\ref{ThmL} below. The index sets $E_\lb(\alpha),E_\rb(\alpha)$ in~\eqref{EqRRElbrb} are independent of $\alpha$ in this range, and given by
\[
  E_\lb = \{ (i\nu_{j-},0) \colon j\in\N_0 \},\quad
  E_\rb = \{ (-i\nu_{j+},0) \colon j\in\N_0 \}.
\]
See Figure~\ref{FigLIndex}.

\begin{figure}[!ht]
\centering
\includegraphics{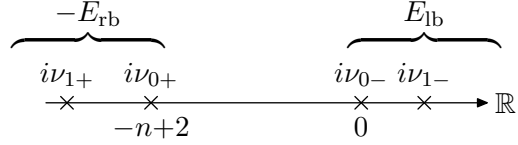}
\caption{The index sets $E_\lb$ and $-E_\rb$.}
\label{FigLIndex}
\end{figure}

Define
\begin{gather*}
  \hat E_\lb=\bigextcup_{j\in\N_0}(E_\lb+j),\quad
  \hat E_\rb=\bigextcup_{j\in\N_0}(E_\rb+j), \\
  \check E_\bullet=\hat E_\bullet\extcup\hat E_\bullet,\ \bullet=\lb,\rb,\quad
  \check E_\ff = \N_0 + ((\check E_\lb+\check E_\rb)\extcup\N), \\
  \cE_\bop = (\cE_\lb,\cE_\ff,\cE_\rb) = \bigl(\check E_\lb,\check E_\ff+2,\check E_\rb+2\bigr);
\end{gather*}
the summand $2$ here arises from the second order nature of $\Delta_g$, cf.\ Theorem~\ref{ThmR}. We remark that $\inf\Re\cE_\lb=0$ and $\inf\Re\cE_\rb=n$. Let
\[
  \cE'_\ff(w) := \cE_\ff \extcup{}(\N_0-2 w).
\]
We will make the index set of complex powers more precise at $\ff$ than in Theorem~\ref{ThmCR}; thus, let
\begin{align*}
  \cE_\ff(w) :=
    \begin{cases}
      \cE_\ff, & w=-1, \\
      \{ (z,k)\in\cE'_\ff(w) \colon \Re z\geq n \}, & \Re w\leq-\frac{n}{2}, \\
      \{ (z,k)\in\cE'_\ff(w) \colon \Re z\geq -2\Re w\ \text{(`$>$' when $k>0$)} \}, & \Re w>-\frac{n}{2},\ w\neq -1,
    \end{cases}
\end{align*}
and finally
\[
  \cE(w) = \bigl(\cE_\lb,\cE_\ff(w),\cE_\rb,\N_0).
\]

\begin{thm}
\label{ThmL}
  Let $(X,g)$ be as above, with $X$ of dimension $n\geq 3$. Let $\alpha\in(-n+2,0)$. Then $h^2\Delta_g+1$ is fully elliptic at weight $\alpha$. Moreover, for $w\in\C$, the complex powers defined in~\eqref{EqPh} satisfy
  \begin{equation}
  \label{EqL}
    (h^2\Delta_g+1)_w \in \bigl(\tfrac{x}{x+h}\bigr)^{-2 w}\Psi_{\cop\semi}^{2 w}(X) + \Psi_{\cop\semi}^{-\infty,\cE(w)}(X).
  \end{equation}
\end{thm}
\begin{proof}
  The main task is the verification of full ellipticity. The symbolic conditions~\eqref{ItOpEll1} and \eqref{ItOpEll2} in Definition~\ref{DefOpEll} are immediate since $\sigmab_2(x^2\Delta_g)$ is a positive definite quadratic form on $\Tb^*X$, hence it remains to check~\eqref{ItOpEll3}: the invertibility of the $\tface$-normal operator.
  
  Let $A_h=h^2\Delta_g+1$ and $A_\bop=x^2\Delta_g$. Note that there is a spectral gap,
  \begin{equation}
  \label{EqLSpecGap}
    \bigl(-\Im\specb(A_\bop)\bigr) \cap (-n+2,0)=\emptyset,
  \end{equation}
  so $\alpha\notin-\Im\specb(A_\bop)$. Now, the model operator $N_\tface(A_h)$, see~\eqref{EqOpNtf}, is the shifted (by $1$) Laplace operator of an \emph{exact} conic metric on the model space ${}^+N\pa X\cong[0,\infty)_x\times\pa X$; indeed, we have
  \[
    N_\tface(A_h)=\Delta_{g_0}+1,\quad g_0 := \dd x^2 + x^2 k(y,\dd y),
  \]
  We need to prove its invertibility as a map
  \begin{equation}
  \label{EqLNtf}
    N_\tface(A_h) \colon H_{\bop,\scop}^{s,\alpha}(\ol{{}^+N\pa X}) \to H_{\bop,\scl}^{s-2,\alpha-2}(\ol{{}^+N\pa X}).
  \end{equation}
  Passing to metric densities for convenience, we need to show that
  \begin{equation}
  \label{EqLNtf2}
    N_\tface(A_h) \colon H_{\bop,\scop}^{s,\alpha+\mfrac{n}{2}}(\ol{{}^+N\pa X},|\dd g_0|) \to H_{\bop,\scop}^{s-2,\alpha+\mfrac{n}{2}-2}(\ol{{}^+N\pa X},|\dd g_0|)
  \end{equation}
  is invertible. By elliptic regularity in the b- and scattering calculi, any element $u\in\ker(N_\tface(A_h))$ automatically lies in $x^{\alpha+\mfrac{n}{2}}\Hb^\infty(|\dd g_0|)$ near $x=0$ and is Schwartz as $x\to\infty$. Near $x=0$, we in fact have $u\in x^{\mfrac{n}{2}-\eps}\Hb^\infty(|\dd g_0|)$ in view of~\eqref{EqLSpecGap}. This suffices to justify integrations by parts in
  \[
    0 = \la\la N_\tface(A_h)u,u\ra_{L^2(|\dd g_0|)} = \|\dd_{g_0}u\|^2+\|u\|^2,
  \]
  which implies that $u=0$. This proves the injectivity of~\eqref{EqLNtf2}.

  To prove surjectivity, we note that~\eqref{EqLNtf2} is a Fredholm operator, and symmetric on $L^2(\ol{{}^+N\pa X},|\dd g_0|)$; thus its cokernel is equal to the kernel of $N_\tface(A_h)$ on $H_{\bop,\scop}^{-s+2,-\alpha-\mfrac{n}{2}+2}$, which is the dual of the target space in~\eqref{EqLNtf2}. But $-\alpha-\tfrac{n}{2}+2=\tfrac{n}{2}+(-(n-2)-\alpha)$, and $-(n-2)-\alpha\in(-(n-2),0)$; thus, the cokernel is trivial by the same argument as for injectivity.

  We wish to apply Theorem~\ref{ThmPh}. First of all, for $\tilde\lambda\leq 0$, the operator $(h^2\Delta_g+1)-\tilde\lambda$ is fully elliptic at weight $\alpha$, as follows either by repeating the above arguments (using $1-\tilde\lambda>0$), or by noting that $(h^2\Delta_g+1)-\tilde\lambda=(1-\tilde\lambda)((h')^2\Delta_g+1)$ is a nonzero multiple of the operator just studied, with a rescaled semiclassical parameter $h'=h(1-\tilde\lambda)^{1/2}$.
  
  The second assumption of Theorem~\ref{ThmPh} to be verified is the full ellipticity of $\tilde h^2 A_h+1$; in the notation of Definition~\ref{DefCREll}, we have $A_\bop=x^2\Delta_g$ (consistent with the definition before~\eqref{EqLSpecGap}). But then the symbolic ellipticity condition~\eqref{ItCREllSymb} of Definition~\ref{DefCREll} follows from the positive definiteness of $\sigmab_2(A_\bop)$, while condition~\eqref{ItCREllFull} is the full ellipticity of $h^2 x^{-2}A_\bop+1=h^2\Delta_g+1$, which we just verified. By Theorem~\ref{ThmPh}, we now get
  \begin{equation}
  \label{EqLAw1}
    (A_h)_w \in \rho_\ff^{-2 w}\Psi_{\cop\semi}^{2 w}(X) + \Psi_{\cop\semi}^{-\infty,(\cE_\lb,\cE'_\ff(w),\cE_\rb,\N_0)}(X).
  \end{equation}
  For $w=-1$, $(A_h)_w=A_h^{-1}$ has index set $\cE_\ff$ at $\ff$ by Theorem~\ref{ThmR}. For $\Re w=-1$ with $w\neq -1$ (thus $\Im w\neq 0$), we have $\cE'_\ff(w)=\cE_\ff\cup(\N_0-2 w)$ simply, hence $\cE_\ff(w)=\cE'_\ff(w)$. For $\Re w>-1$, we have $\inf\Re\cE_\ff=2>-2\Re w$, hence $\cE_\ff(w)=\cE'_\ff(w)$ in this case as well.

  Consider now $w\in\C$ with $\Re w<-1$, in which case $\cE_\ff(w)$ is a proper subset of $\cE'_\ff(w)$. Let $k\in\N$ with $\Re(w/k)\in(-1,0)$; then
  \begin{equation}
  \label{EqLAwk}
    (A_h)_w = ((A_h)_{w'})^k,\quad w':=w/k.
  \end{equation}
  The index sets of $(A_h)_{w'}$ having lower bounds $0$, $-2\Re w'$, $n$, $0$ at $\lb,\ff,\rb,\tface$, we conclude using Proposition~\ref{PropRCComp} that for $-2\Re w<n$, the index set of $(A_h)_w$ at $\ff$ has leading coefficient $-2 w$ (without logarithms), while for any $w$ with $-2\Re w\geq n$, exponents $(z,k)$ with $\Re z=n$ may be present. To determine the index set of $(A_h)_w$ at $\lb$ and $\rb$, we do \emph{not} use formula~\eqref{EqLAwk}, but recall that these index sets are necessarily subsets of $\cE_\lb$ and $\cE_\rb$ by~\eqref{EqLAw1}. This proves that the index set of $(A_h)_w$ at $\ff$ is $\cE_\ff(w)$ as stated.
\end{proof}

Less precise mapping properties require only conormal bounds; hence we record:
\begin{cor}
\label{CorLCon}
  For $\Re w>-\tfrac{n}{2}$, we have
  \begin{subequations}
  \begin{equation}
  \label{EqLCon}
    (h^2\Delta_g+1)_w \in \bigl(\tfrac{x}{x+h}\bigr)^{-2 w}\Psi_{\cop\semi}^{2 w}(X) + \cA^{-\eps,-2\Re w,n-\eps,0}(X^2_{\cop\semi})
  \end{equation}
  for all $\eps>0$, where $\cA$ denotes the space of $L^\infty$-conormal functions vanishing rapidly at $\sface,\dface$, with the exponents denoting the decay rates at $\lb$, $\ff$, $\rb$, $\tface$. For $\Re w\leq-\tfrac{n}{2}$, we have
  \begin{equation}
  \label{EqLCon2}
    (h^2\Delta_g+1)_w \in \bigl(\tfrac{x}{x+h}\bigr)^{-2 w}\Psi_{\cop\semi}^{2 w}(X) + \cA^{-\eps,n-\eps,n-\eps,0}(X^2_{\cop\semi}).
  \end{equation}
  \end{subequations}
\end{cor}

We now relate this to the complex powers defined via the functional calculus for self-adjoint operators:

\begin{thm}
\label{ThmLF}
  Let $(X,g)$ be a compact conic manifold of dimension $n\geq 3$.
  \begin{enumerate}
  \item\label{ItLF1} Denote by $\cD_h^1$ the quadratic form domain of $A_h=h^2\Delta_g+1$: this is the completion of $\CIc(X^\circ)$ with respect to the norm
  \[
    \|u\|_{\cD_h^1}^2 := \|u\|_{L^2(X)}^2 + \| h \dd u \|_{L^2(X)}^2,
  \]
  where $L^2(X)=L^2(X,|\dd g|)$ is the metric $L^2$ space. Then $\cD_h^1=H_{\cop,h}^{1,1-\mfrac{n}{2},-\mfrac{n}{2}}(X)$ with uniformly equivalent norms, i.e.\ there exists $C>1$ with $C^{-1}\|u\|_{H_{\cop,h}^{1,1-\mfrac{n}{2},-\mfrac{n}{2}}(X)}\leq\|u\|_{\cD_h^1}\leq C\|u\|_{H_{\cop,h}^{1,1-\mfrac{n}{2},-\mfrac{n}{2}}(X)}$.\footnote{The shift of $\tfrac{n}{2}$ in the weight is due to the different volume densities: the metric density for $L^2(X)$, and a b-density for the $H_{\cop,h}(X)$-spaces.}
  \item\label{ItLF2} Define $A_h^{w/2}$, $w\in\C$, via the functional calculus for the Friedrichs extension of $\Delta_g$. Denote its domain for $\Re w\geq 0$ by $\cD_h^w=\cD(A_h^{w/2})$ with norm $\|u\|_{\cD_h^w}=\|A_h^{w/2}u\|_{L^2(X)}$; for $\Re w<0$, set $\cD_h^w=(\cD_h^{-\bar w})^*$.\footnote{Thus, for $w<0$, $A_h^{w/2}$ acts on $\cD_h^w$ by duality: for $u\in\cD_h^w$, we define $A_h^{w/2}u\in L^2(X)$ by $\la A_h^{w/2}u,v\ra:=\la u,A_h^{\bar w/2}v\ra$ for all $v\in L^2(X)$; note that $A_h^{\bar w/2}v\in\cD_h^{-\bar w}$.} Then
    \begin{equation}
    \label{EqLF2}
      \cD_h^w = H_{\cop,h}^{\Re w,\Re w-\mfrac{n}{2},-\mfrac{n}{2}}(X), \quad \Re w\in(-\tfrac{n}{2},\tfrac{n}{2}),
    \end{equation}
    with uniformly equivalent norms. Moreover, for $\Re w\geq\tfrac{n}{2}$, the inclusion map
    \[
      \cD_h^w \hra \bigcap_{\eps>0} H_{\cop,h}^{\Re w,-\eps,-\mfrac{n}{2}}(X)
    \]
    is uniformly bounded as $h\to 0$.
  \item\label{ItLF3} Let $w,z\in\C$ be such that $\Re z,\Re(z-w)>-\tfrac{n}{2}$. Then $A_h^{w/2}\colon\cD_h^z\to\cD_h^{z-w}$ has a distributional Schwartz kernel which is a semiclassical cone ps.d.o.\ of the form~\eqref{EqL} and \eqref{EqLCon}--\eqref{EqLCon2}.
  \end{enumerate}
\end{thm}

\begin{rmk}
\label{RmkLFDomSet}
  The domains are $\cD_h^w=x^{\Re w-\mfrac{n}{2}}\Hb^{\Re w}(X)$ as sets, but with $h$-dependent norms; cf.\ Remark~\ref{RmkRMHbh}.
\end{rmk}

\begin{rmk}
\label{RmkLFDom}
  For $w>\tfrac{n}{2}$, and restricting to the non-semiclassical case $h=1$, the domain $\cD_1^w$ contains $x^{w-\mfrac{n}{2}}\Hb^w(X)$ as a subspace of elements with finite codimension, a complement of which consists of finite polyhomogeneous expansions at $\pa X$ (arising from the indicial roots~\eqref{EqLSpecb}); see \cite[Lemma~3.2]{MelroseWunschConic} for the case $w=\tfrac{n}{2}$. In particular, $\CIc(X^\circ)\subset\cD_1^w$ is not a dense subspace. By duality, for $w\leq-\tfrac{n}{2}$, the identification~\eqref{EqLF2} is not valid either.
\end{rmk}

\begin{proof}
  We shall write `$A\lesssim B$' to mean `$A\leq C B$ for an $h$-independent constant $C$', and `$A\sim B$' for $A\lesssim B$, $B\lesssim A$. We first note that
  \[
    \|u\|_{L^2(X,|\dd g|)} \sim \|x^{\mfrac{n}{2}}u\|_{L^2_\bop(X)} \sim \|u\|_{H_{\cop,h}^{0,-\mfrac{n}{2},-\mfrac{n}{2}}(X)}.
  \]
  Since $h\dd\in(\tfrac{x}{x+h})^{-1}\Psi_{\cop\semi}^1(X)$, we thus have
  \[
    \|u\|_{\cD_h^1}^2 = \|u\|_{L^2(X,|\dd g|)}^2 + \|h\dd u\|_{L^2(X,|\dd g|)}^2 \lesssim \|u\|_{H_{\cop,h}^{0,-\mfrac{n}{2},-\mfrac{n}{2}}(X)}^2 + \|u\|_{H_{\cop,h}^{1,1-\mfrac{n}{2},-\mfrac{n}{2}}(X)}^2.
  \]
  so $\|u\|_{\cD_h^1}\lesssim\|u\|_{H_{\cop,h}^{1,1-\mfrac{n}{2},-\mfrac{n}{2}}}$.
  
  For the converse, we split $u=\chi u+(1-\chi)u$, with $\chi\in\CI(X)$ supported in the collar neighborhood of $\pa X$. Then Hardy's inequality $\|x^{-1}\chi u\|_{L^2}\lesssim\|\dd(\chi u)\|_{L^2}$ (which uses that $n\geq 3$) gives
  \begin{align*}
    \bigl\|\bigl(\tfrac{x}{x+h}\bigr)^{-1}u\bigr\|_{L^2} &\leq \|u\|_{L^2} + h\|x^{-1}\chi u\|_{L^2} + h\|x^{-1}(1-\chi)u\|_{L^2} \\
      &\lesssim \|u\|_{L^2}+h\|\dd(\chi u)\|_{L^2}+h\|(1-\chi)u\|_{L^2} \\
      &\lesssim \|u\|_{L^2} + h\|\dd u\|_{L^2},
  \end{align*}
  using the bound $\|[\dd,\chi]u\|_{L^2}\lesssim\|u\|_{L^2}$. For $V\in\Vb(X)$, we further have the trivial estimate
  \[
    \bigl\|\bigl(\tfrac{x}{x+h}\bigr)^{-1}\tfrac{h}{h+x}V u\bigr\|_{L^2} \lesssim \|h\dd u\|_{L^2}.
  \]
  Together, this proves $\|u\|_{H_{\cop,h}^{1,1-\mfrac{n}{2},-\mfrac{n}{2}}}\lesssim\|u\|_{\cD_h^1}$, and thus completes the proof of part~\eqref{ItLF1}.

  Turning to part~\eqref{ItLF2}, we note that $\cD_h^{-1}=(\cD_h^1)^*=H_{\cop,h}^{-1,-1-\mfrac{n}{2},-\mfrac{n}{2}}(X)$, where we take adjoints with respect to the inner product of $L^2(X,|\dd g|)=H_{\cop,h}^{0,-\mfrac{n}{2},-\mfrac{n}{2}}(X)$. By interpolation, this proves~\eqref{EqLF2} for $w\in[-1,1]$. In particular, since $\CIc(X^\circ)\subset\cD_h^{\pm 1}$ is dense, we conclude that $A_h|_{\cD_h^1}$ is given by the distributional action of $h^2\Delta_g+1$.

  Next, we prove~\eqref{EqLF2} for $w\in[1,\tfrac{n}{2})$; the result for $w\in(-\tfrac{n}{2},-1]$ then follows by duality. Now, since $A_h$ is positive, $\spec A_h\subset[1,\infty)$, we have $\cD_h^w\subset\cD_h^{w'}$ for $w\geq w'$; hence we can identify, for $w\geq 1$,
  \[
    \cD_h^w = \{ u\in \cD_h^1 \colon (h^2\Delta_g+1)u\in\cD_h^{w-2} \},
  \]
  also in the sense of equivalence of norms, namely
  \[
    \|u\|_{\cD_h^w} := \|A_h^{w/2}u\|_{L^2} \sim \|u\|_{\cD_h^1} + \|(h^2\Delta_g+1)u\|_{\cD_h^{w-2}}.
  \]
  Assume inductively that~\eqref{EqLF2} holds for $w-2$ in place of $w$. Then $u\in\cD_h^w$ means $u\in H_{\cop,h}^{1,1-\mfrac{n}{2},-1}(X)$ and
  \[
    f := (h^2\Delta_g+1)u \in H_{\cop,h}^{w-2,w-2-\mfrac{n}{2},-\mfrac{n}{2}}(X).
  \]
  Since the weight of $u$ at $x=0$, i.e.\ $1-\tfrac{n}{2}$, lies in the spectral gap $(-n+2,0)$, we can compute $u$ by inverting the operator $h^2\Delta_g+1$ at weight $\alpha=1-\tfrac{n}{2}$ by means of Theorem~\ref{ThmL} or \ref{ThmR}. (Note here that $h^2\Delta_g+1$ is invertible on $\Hb^{s,\alpha}(X)$ for \emph{all} $h>0$.) Thus,
  \[
    (h^2\Delta_g+1)^{-1} \in \bigl(\tfrac{x}{x+h}\bigr)^2\Psi_{\cop\semi}^{-2}(X) + \Psi_{\cop\semi}^{-\infty,(\check E_\lb,\check E_\ff+2,\check E_\rb+2,\N_0)}(X),
  \]
  which maps $f$ into $H_{\cop,h}^{w,w-\mfrac{n}{2},-\mfrac{n}{2}}(X)$ by parts~\eqref{ItRMSob1} and \eqref{ItRMSob3} of Proposition~\ref{PropRMSob}. This establishes~\eqref{EqLF2} for the power $w$.

  For $w\in[\tfrac{n}{2},\tfrac{n}{2}+2)$, $(h^2\Delta_g+1)^{-1}$ maps $\cD_h^{w-2}$ into
  \[
    H_{\cop,h}^{w,w-\mfrac{n}{2},-\mfrac{n}{2}}(X)+H_{\cop,h}^{\infty,-\eps,-\mfrac{n}{2}}(X) \subset H_{\cop,h}^{w,-\eps,-\mfrac{n}{2}}(X)
  \]
  for any $\eps>0$ by Proposition~\ref{PropRMSob}\eqref{ItRMSob3}, as claimed. Assuming inductively that $\cD_h^{w-2}\subset H_{\cop,h}^{w-2,-\eps,-\mfrac{n}{2}}(X)$ with $w-2\geq\tfrac{n}{2}$, we obtain $\cD_h^w\subset H_{\cop,h}^{w,-\eps,-\mfrac{n}{2}}(X)$ by an analogous argument. Since $\cD_h^w=\cD_h^{\Re w}$ for any $w\in\C$, the proof of part~\eqref{ItLF2} is now complete.

  Part~\eqref{ItLF3} is an immediate consequence of part~\eqref{ItLF2}.
\end{proof}

\subsection{Semiclassical propagation through cone points}
\label{SsLProp}

In this section, we shall work locally near a conic point. Thus, with $\pa X$ denoting a compact $(n-1)$-dimensional manifold without boundary, $n\geq 3$, let
\[
  X' := [0,x_0)_x \times \pa X,\quad
  g = \dd x^2 + x^2 k(x,y,\dd y),
\]
where $k$ is a smooth Riemannian metric on $\pa X$ depending smoothly on $x\in[0,x_0]$. We are interested in the propagation of semiclassical singularities through the cone point $x_0=0$ for solutions of the equation
\begin{equation}
\label{EqLPropEq}
  (h^2\Delta_g - 1)u = f,
\end{equation}
where $u,f$ lie in a suitable semiclassical cone Sobolev spaces. We may isometrically embed $(X',g)$ into a compact Riemannian manifold $(X,g)$ with boundary $\pa X'=\pa X$; we then have semiclassical cone Sobolev spaces on $X$ at our disposal, and we define
\[
  H_{\cop,h,\loc}^{s,\alpha,\tau}(X'),\quad \text{resp.}\quad \cD^w_{h,\loc}(X')
\]
to consist of all $u$ so that $\phi u\in H_{\cop,h}^{s,\alpha,\tau}(X)$ (with respect to a smooth b-density on $X$ as in Definition~\ref{DefRMSob}), resp.\ $\phi u\in\cD^w_h(X)$ for all $\phi\in\CIc(X')$. Here, we set $\cD^w_h(X)=\cD((h^2\Delta_g+1)^{w/2})$ (using the functional calculus for the Friedrichs extension of $\Delta_g$) as in Theorem~\ref{ThmLF}\eqref{ItLF2}.

We need to assume the absence of semiclassical wave front set along null-bicharacteristics over the interior $((X')^\circ,g)$ which strike the cone point. Now, any \emph{outgoing} unit speed geodesic $\gamma(s)$ on $(X',g)$, meaning $\gamma$ is defined on $(0,\eps)$ for some $\eps>0$ and has the property that $x(\gamma(s))\to 0$ as $s\to 0$, is necessarily radial, i.e.\ of the form $\gamma(s)=(s,y_0)$ for some fixed $y_0\in\pa X$. Thus, writing covectors over a point in the interior $(X')^\circ$ as
\[
  \sigma\,\dd x + \eta,\quad \eta\in T^*\pa X,
\]
the set
\[
  \cU := \bigl\{ (z,\zeta) \in T^*(X')^\circ \colon |\zeta|_{g^{-1}}^2=1,\ \zeta\neq\sigma\,\dd x\ \text{with}\ \sigma>0 \bigr\}
\]
contains all covectors in the characteristic set of $h^2\Delta_g-1$ which do not correspond (via $g$) to the tangent vector of an outgoing unit speed geodesic.

We recall that a point $(z,\zeta)\in T^*(X')^\circ$ does \emph{not} lie in the semiclassical wave front set of order $0$, denoted here by $\WFh^0(u)$, of an $h$-tempered distribution $u$ on $(X')^\circ$ (i.e.\ $u\in h^{-N}H_{h,\loc}^{-N}((X')^\circ)$ for some $N$) if and only if there exists $A\in\Psih((X')^\circ)$ with compactly supported Schwartz kernel so that $A u\in L^2(X')$ is uniformly bounded in $h$. Equivalently, in local coordinates, there exist cutoff functions $\phi\in\CIc((X')^\circ)$ identically $1$ near $z$, and a cutoff $\psi\in\CIc(\R^n)$ identically $1$ near $\zeta$ so that $\psi\cF_h(\phi u)\in L^2(\R^n)$ uniformly as $h\to 0$, where $\cF_h u(\zeta)=\int_{\R^n} e^{-i z\cdot\zeta/h}u(z)\,\dd z$ is the semiclassical Fourier transform.

\begin{thm}
\label{ThmLProp}
  Let $l\in(-\tfrac{n-2}{2},\tfrac{n-2}{2})$, and suppose that $u\in h^{-N}\cD_{h,\loc}^{1+l}(X')$ solves equation~\eqref{EqLPropEq} with $f\in h\cD^{-1+l}_{h,\loc}(X')$. (That is, $\phi h^{-1}f\in\cD_{h,\loc}^{-1+l}(X')$ is uniformly bounded as $h\to 0$ for all $\phi\in\CIc(X')$.) If $\WFh^0(u|_{(X')^\circ})\cap\cU=\emptyset$, then $u\in\cD^{1+l}_{h,\loc}(X')$.
\end{thm}

The crucial point of this theorem is the fact it is a lossless propagation estimate: $u$ loses precisely one power of $h$ relative to the right hand side $f$.

\begin{proof}[Proof of Theorem~\usref{ThmLProp}]
  Using Theorem~\ref{ThmLF}, we may replace $u,f$ by
  \[
    u'=(h^2\Delta_g+1)^{l/2}u \in h^{-N}\cD^1_{h,\loc}(X'), \quad
    f'=(h^2\Delta_g+1)^{l/2}f \in \cD^{-1}_{h,\loc}(X').
  \]
  We make the following two important observations:
  \begin{enumerate}
  \item $(h^2\Delta_g+1)^{l/2}$ and $h^2\Delta_g-1$ commute, hence $(h^2\Delta_g-1)u'=f'$.
  \item Since for any $\chi\in\CI(X^\circ)$, the operator $\chi\cdot(h^2\Delta_g+1)^{l/2}$ is a semiclassical b-pseudodifferential operator of order $l$ by~\eqref{EqRCOtherB}, we have $\WFh^0(u'|_{(X')^\circ})\cap\cU=\emptyset$, and thus $u'$ lies microlocally in $H_h^1$ on $\cU$.
  \end{enumerate}

  We can then apply the semiclassical propagation estimate \cite[Proposition~7.2]{BaskinMarzuolaCone}, combined with the elliptic estimate \cite[Proposition~7.1]{BaskinMarzuolaCone}; in the notation of the reference, we have $\cD_h=\cD_h^1$, $\cD_h'=\cD_h^{-1}$, and by using the elliptic estimate, we may take in \cite[Proposition~7.2]{BaskinMarzuolaCone} the operators $\cQ,G$ to be cutoff functions on the base $X'$, identically $1$ near $\pa X'$, while $\cQ_1$ is a semiclassical ps.d.o.\ on $(X')^\circ$ with wave front set contained in $\cU$, and elliptic on a sufficiently large subset of $\cU$, namely so that all backwards null-bicharacteristics starting over a point in $\supp\cQ\cap(X')^\circ$ either strike the cone point or enter the elliptic set of $\cQ_1$. The a priori assumption that $u'$ is $h$-tempered provides control of the $\cO(h^\infty)$ error term in the estimate in \cite[Proposition~7.2]{BaskinMarzuolaCone}.
\end{proof}

\begin{rmk}
  Theorem~\ref{ThmLF} and Proposition~\ref{PropRMHbh} give the inclusions
  \[
    h^{|w|}H_{\bop,h}^{w,w-\mfrac{n}{2}}(X) \subset \cD_h^w(X) \subset h^{-|w|}H_{\bop,h}^{w,w-\mfrac{n}{2}}(X)
  \]
  for $w\in(-\tfrac{n}{2},\tfrac{n}{2})$, which allows for a translation of Theorem~\ref{ThmLProp} into a propagation result on semiclassical b-Sobolev spaces allowing for a flexible choice of the weight at the conic point; this is however rather lossy in terms of powers of $h$. One can in fact deduce almost sharp propagation estimates (losing $1+\eps$ powers of $h$ for any $\eps>0$ for the entire range of weights allowed in Theorem~\ref{ThmLProp}) from real principal type and radial point estimates in a semiclassical cone calculus with variable orders, thus without relying on the delicate arguments of \cite{MelroseVasyWunschEdge,BaskinMarzuolaCone}. Details will be given in future work.
\end{rmk}

\appendix
\section{Review of b-analysis}
\label{SB}

We refer the reader to \cite{MelroseDiffOnMwc,MelrosePushFwd,EpsteinMelroseMendozaPseudoconvex,MelroseAPS,MazzeoEdge} for details on various aspects of b-analysis relevant to the present paper; here, we only briefly collect some of the key notions. Let $X$ denote an $n$-dimensional manifold with corners, which is locally modelled on
\[
  \R^n_k=[0,\infty)_x^k\times\R_y^{n-k}
\]
for various $0\leq k\leq n$. We assume that $\pa X$ is the union of embedded hypersurfaces $M_1(X)=\{H_1,\ldots,H_N\}$. Denote by $\rho_{H_i}\in\CI(X)$ a defining function, so $\rho_{H_i}\geq 0$, $H_i=\rho_{H_i}^{-1}(0)$, and $\dd\rho_{H_i}\neq 0$ on $H_i$. The space $\Vb(X)\subset\cV(X)=\CI(X;T X)$ of \emph{b-vector fields} consists of all smooth vector fields which are tangent to $\pa X$; it can be identified with the space of smooth sections of the \emph{b-cotangent bundle} $\Tb M$, with local frame $x_i\pa_{x_i}$, $\pa_{y_j}$, $1\leq i\leq k$, $1\leq j\leq n-k$. The dual bundle is the \emph{b-cotangent bundle} $\Tb^*X$.

For a \emph{b-differential operator} $A\in\Diffb^m(X)$, i.e.\ a locally finite sum of up to $m$-fold products of b-vector fields, we define its \emph{normal operator} $N_H(A)\in\Diffb^m({}^+N_H X)$ at $H\in M_1(X)$ by freezing coefficients (here ${}^+N_H X$ is the inward pointing normal bundle, defined to contain the zero section). In local coordinates as above, in which $H$ is given by $x_1=0$, and $x'=(x_2,\ldots,x_k)$, we can write $A=\sum a_{\alpha\beta}(x_1,x',y) (x D_x)^\alpha D_y^\beta$, where $a_{\alpha\beta}\in\CI(X)$, and $(x D_x)^\alpha:=\prod_{i=1}^k(x_i D_{x_i})^{\alpha_i}$; then $N_H(A)=\sum a_{\alpha\beta}(0,x',y)(x D_x)^\alpha D_y^\beta$. Fixing a defining function $\rho_H$ of $H$, and working in coordinates in which $x_1=\rho_H$, the \emph{Mellin transformed normal operator} $\wh{N_H(A)}(\sigma)$ is defined by replacing $x_1 D_{x_1}$ by $\sigma$ in the expression for $N_H(A)$. Furthermore, we say that $A$ is said to be \emph{elliptic} if the polynomial $\sum_{|\alpha|+|\beta|=m} a_{\alpha\beta}(x,y)\xi^\alpha\eta^\beta$ does not vanish for $(0,0)\neq(\xi,\eta)\in\R^k\times\R^{n-k}$.

For $\alpha\in\R^{M_1(X)}$, we define the \emph{space of distributions conormal to the boundary}, $\cA^\alpha(X)$, to consist of all distributions on $X^\circ$ which lie in $\rho^\alpha L^\infty(X)$ together with all their derivatives along any finite number of b-vector fields; here $\rho^\alpha:=\prod_{H\in M_1(X)}\rho_H^{\alpha_H}$. We furthermore define $L^2_\bop(X)$ as the $L^2$-space for a smooth b-density on $X$, which is locally a smooth positive multiple of $|\frac{\dd x_1}{x_1}\ldots\frac{\dd x_k}{x_k}\dd y_1\ldots \dd y_{n-k}|$; b-Sobolev spaces $\Hb^{s,\alpha}(X)=\rho^\alpha\Hb^s(X)$ are then defined for $s\in\N_0$ by membership of all b-derivatives up to order $s$ in $\rho^\alpha L^2_\bop(X)$, and for general $s\in\R$ by duality and interpolation. Semiclassical b-Sobolev spaces $\Hbh^{s,\alpha}(X)=\rho^\alpha\Hbh^s(X)$ are equal to $\Hb^{s,\alpha}(X)$ as sets, but with $h$-dependent norms: in local coordinates and for $s\in\N_0$,
\[
  \|u\|_{\Hbh^{s,\alpha}(X)}^2 = \sum_{|\beta|+|\gamma|\leq s} \|x^{-\alpha}(h x D_x)^\beta(h D_y)^\gamma u\|_{L^2(X)}^2.
\]

An \emph{index set} $\cE\subset\C\times\N_0$ satisfies, by definition, that $(\sigma,k)\in\cE$ implies $(\sigma+j,\ell)\in\cE$ for all $j\in\N_0$, $0\leq\ell\leq k$, and that $\cE$ contains only finitely many elements $(\sigma,k)$ for which $\Re\sigma$ is smaller than any fixed constant. \emph{Polyhomogeneity}, or \emph{$\cE$-smoothness}, of a conormal distribution $u$ on a manifold $X$ with boundary $H=\pa X$ then means that there exist $u_{\sigma,k}\in\CI(\pa X)$, $(\sigma,k)\in\cE$, such that $u-\sum_{\Re\sigma\leq N}u_{\sigma,k}\rho_H^\sigma|\log\rho_H|^k\in\cA^N(X)$ for all $N\in\R$. More generally, for a collection $\cE=\{\cE_H\colon H\in M_1(X)\}$ of index sets on a manifold with corners, we say that $u$ is $\cE$-smooth if at each $H\in M_1(X)$ it is $\cE_H$-smooth with coefficients which themselves are polyhomogeneous distributions on $H$, with index set at all non-empty boundary hypersurfaces $H\cap H'$, $H'\in M_1(X)$, of $H$ given by $\cE_{H'}$. If this holds, we write $u\in\cA_\phg^\cE(X)$. We also recall the definition
\begin{equation}
\label{EqBExtCup}
  \cE \extcup \cF := \cE \cup \cF \cup \{ (z,j+k+1) \colon (z,j)\in\cE,\ (z,k)\in\cF \}
\end{equation}
of the \emph{extended union} of two index sets $\cE,\cF$.

Let $S\subset X$ be a \emph{p-submanifold}, i.e.\ given in suitable local coordinates by the vanishing of $x'$ and $y'$ for some splittings $(x_1,\ldots,x_k)=(x',x'')$ and $(y_1,\ldots,y_{n-k})=(y',y'')$. If $x'=x$, we call $S$ a \emph{boundary p-submanifold}, otherwise it is a \emph{interior p-submanifold}. The blow-up $[X;S]$ is then defined as $(X\setminus S)\sqcup S N^+S$, where $S N^+S$ is the spherical (i.e.\ quotient by the fiber-wise dilation action on the) inward pointing normal bundle, with smooth structure defined by declaring polar coordinates around $S$ to be smooth down to the origin; one then calls $S N^+S$ the \emph{front face} of $[X;S]$. The natural map $\beta\colon[X;S]\to X$ (given on $S N^+S$ by projection to the base $S$) is called the \emph{blow-down map}. If $T\subset X$ is another p-submanifold which near $T\cap S$ can be expressed in the same local coordinates by the vanishing of another subset of these coordinates, we define the \emph{lift of $T$ to $[X;S]$} by $\beta^{-1}(T)$ if $T\subset S$, and by $\ol{\beta^{-1}(T\setminus S)}$ otherwise.

If $X,X'$ are two manifolds with corners, with boundary defining functions $\rho_H$ ($H\in M_1(X)$) and $\rho'_{H'}$ $(H'\in M_1(X')$), we say that a smooth map $F\colon X\to X'$ is a \emph{b-map} if for all $H'\in M_1(X')$, either $F^*\rho'_{H'}\equiv 0$, or
\[
  F^*\rho'_{H'} = a_{H'} \prod_{H\in M_1(X)} \rho_H^{e(H,H')}
\]
where $a_{H'}\in\CI(X)$ is non-vanishing, and $e(H,H')\in\N_0$. If the first possibility does not occur for all $H'$, $F$ is an \emph{interior b-map}; we henceforth only consider such $F$. Push-forward $F_*\colon T_x X^\circ\to T_{F(x)}(X')^\circ$ extends by continuity to the \emph{b-differential} ${}^\bop F_*\colon\Tb_x X\to\Tb_{F(x)}X'$, $x\in M$. If this map is surjective for all $x\in M$, we say that $F$ is a \emph{b-submersion}; if in addition the exponents $e(H,H')$ have the property that for all $H\in M_1(X)$ there exists at most one $H'\in M_1(X')$ such that $e(H,H')\neq 0$ (equivalently: $F$ does not map a boundary hypersurface of $X$ into a corner of $X'$ of codimension $2$ or higher), we call $F$ a \emph{b-fibration}.

If $S\subset X$ is a p-submanifold and $p\in S$, denote by $\Tb_p S\subset\Tb_p X$ the space of all b-tangent vectors $V|_p$ where $V\in\Vb(X)$ is tangent to $S$. We then say that an interior b-map $F$ is transversal to $S$ if at each $p\in S$, the b-nullspace ${}^\bop\rm{null}(F)=\ker({}^\bop F_*)$ and $\Tb_p S$ are transversal as subspaces of $\Tb_p X$. If for $q\in X$ we denote by $\rm{Fa}(q)$ the largest codimension boundary face of $X$ containing $q$ (so $q$ lies in the interior $\rm{Fa}(q)^\circ$), this condition is equivalent to $F|_{\rm{Fa}(q)^\circ}\colon\rm{Fa}(q)^\circ\to X'$ being transversal (in the standard sense) to $S\cap\rm{Fa}(q)$ for all $q\in S$.

\bibliographystyle{alpha}

\end{document}